\documentclass[final,review,onefignum,onetabnum]{siamart190516}

\usepackage{lipsum,version}
\usepackage{amsfonts}
\usepackage{graphicx}
\usepackage{subfigure}
\usepackage{epstopdf}
\usepackage{algorithmic}

\ifpdf
  \DeclareGraphicsExtensions{.eps,.pdf,.png,.jpg}
\else
  \DeclareGraphicsExtensions{.eps}
\fi

\renewcommand\theequation{\thesection.\arabic{equation}}
\theoremstyle{plain}
\newtheorem{thm}{Theorem}[section]
\newtheorem{lem}{Lemma}[section]

\newtheorem{remark}{\textbf{Remark}}[section]

\usepackage{subfigure}
\usepackage{threeparttable}

\usepackage{amssymb}
\usepackage{lineno} 

    \title{A generalized SAV approach with relaxation for dissipative systems\thanks{This work is supported in part by NSFC 11971407,  NSF DMS-2012585 and AFOSR FA9550-20-1-0309.}}

   \author{Yanrong Zhang\thanks{School of Mathematical Sciences and Fujian Provincial Key Laboratory on Mathematical Modeling and High Performance Scientific Computing, Xiamen University, Xiamen, Fujian, 361005, China. Email: yanrongzhang@stu.xmu.edu.cn}.
       \and Jie Shen\thanks{Corresponding Author. Department of Mathematics, Purdue University, West Lafayette, IN 47907, USA.  Email: shen7@purdue.edu.}}

\begin{document}

\maketitle

\begin{abstract}
The scalar auxiliary variable (SAV) approach \cite{shen2018scalar} and its generalized version GSAV proposed in  \cite{huang2020highly} are very popular methods to construct efficient and accurate energy stable schemes for  nonlinear dissipative systems. However, the discrete value of the SAV is not directly linked to the free energy of the dissipative system, and may lead to inaccurate solutions if the time step is not sufficiently small.
Inspired by the relaxed SAV method proposed in  \cite{jiang2022improving} for gradient flows,  we propose in this paper a generalized SAV approach with relaxation (R-GSAV) for general dissipative systems. The R-GSAV approach preserves all the advantages of the GSAV appraoch, in addition, it dissipates a modified energy that is directly linked to the original free energy. 
We prove that the $k$-th order implicit-explicit (IMEX) schemes based on R-GSAV are unconditionally energy stable, and we carry out  a rigorous error analysis for $k=1,2,3,4,5$. We present ample numerical results to  demonstrate the improved accuracy and effectiveness of the R-GSAV approach.
\end{abstract}



\pagestyle{myheadings}
\thispagestyle{plain}

 \section{Introduction} \label{sec:Intro}
\setcounter{equation}{0}

When designing  numerical schemes for nonlinear dissipative systems, it is crucial to preserve the  energy dissipation law  at the discrete level in order to eliminate non-physics numerical solutions. How to design efficient and accurate energy stable schemes for nonlinear dissipative systems has been a subject of extensive research in the 
past few  decades. 
Existing  popular approaches include, but not limited to:
(i) Convex splitting approach \cite{elliott1993global, eyre1998unconditionally, shen2012second,baskaran2013convergence}: it leads to unconditionally energy  stable for a large class of gredient flows but it requires solving a nonlinear system at each time step; 
 (ii) Stabilized linearly implicit approach \cite{zhu1999coarsening,shen2010numerical}: 
	It leads to  unconditionally energy stable schemes for gredient flows with global Lipschitz conditions and only  requires solving linear systems with constant coefficients at each time step; 
(iii) Exponential time differencing (ETD) approach \cite{du2019maximum, du2021maximum, wang2016efficient}: 
	It can lead to unconditionally energy stable schemes for certain mildly nonlinear systems and requires the diagonalization of the discrete Laplace operator;
(iv) Invariant energy quadratization (IEQ) approach \cite{yang2016linear, yang2017efficient}: it leads to unconditionally energy stable schemes for a large class of gradient flows, but requires solving a coupled linear system with variable coefficients at each time step;
(v) Scalar auxiliary variable (SAV) approach \cite{shen2018scalar, shen2019new}:  it leads to unconditionally energy stable schemes for a large class of gradient flows, but only requires solving two decoupled linear systems with constant coefficients at each time step. We refer to  \cite{shen2019new,DuF20,SheY20,du2021maximum} and references therein for a more complete literature on this subject.
Note that unconditional energy stable schemes constructed from  the above approaches are mostly limited to first- or second-order (see, however, higher-order versions based on Runge-Kutta \cite{MR4033693} or Gaussian collocation \cite{MR4053864} which require solving coupled linear or nonlinear systems). In \cite{huang2020highly} (see also \cite{huang2021implicit}), a generalized SAV approach is proposed for general dissipative systems. Its higher-order versions are also unconditionally energy stable and only requires solving one decoupled linear system with constant coefficients at each time step.

Thanks to their simplicity, efficiency, accuracy and generality, the SAV and GSAV approaches have received much attention recently, they, along with their  various variations/extensions, have been used to construct unconditionally energy stable schemes for a large class of nonlinear systems, including   various gradient flows (see, for instance, \cite{MR4147832, cheng2020new,  MR4158702, li2020stability,SheY20,MR4064802,MR4131868}),  gradient flows with other global constraints  (see, for instance, \cite{cheng2020global}),  Navier-Stokes equations and related systems (see, for instance, \cite{lin2019numerical,li2020sav,MR4244094,MR4350535}), time fractional PDEs \cite{MR4344588}, conservative or Hamiltonian  systems (see, for instance, \cite{antoine2021scalar, MR3979103, MR4239814,feng2021high}), and many more.  However, in the original SAV approach and its various variants, the discrete value of the SAV is not directly linked to its definition at the continuous level, and this may lead  to a loss of accuracy when the time step is not sufficiently small. In order to see this more clearly, we briefly describe the original SAV approach below.

 To fix the idea, we consider a free energy in the form
\begin{equation}
	E_{tot}(\phi)=\frac{1}{2}(\mathcal{L} \phi, \phi) + \int_{\Omega} F(\phi) \mathrm{d} \mathbf{x},
\end{equation}
where $\mathcal{L}$  is a linear self-adjoint elliptic operator, $F(\phi)$ is a nonlinear potential function.
Then, the gradient flow associated with the above free energy can be written as
\begin{equation} \label{eq:model-problem}
	\left\{\begin{aligned}
		& \frac{\partial \phi}{\partial t}=-\mathcal{G} \mu, \\
		& \mu=\frac{\delta E_{tot}(\phi)}{\delta \phi}=\mathcal{L} \phi+F^{\prime}(\phi),
	\end{aligned}\right.
\end{equation}
with periodic or homogeneous  Neumann boundary condition, and  $\mathcal{G}$ is a positive operator. 
Let  $E_{1}(\phi) = \int_{\Omega} F(\phi) \mathrm{d} \mathbf{x}$ and assume 
$E_1(\phi) + C_0 > 0$, where $C_0>0$ is a constant. 
The key idea of the original SAV approach \cite{shen2018scalar,shen2019new}   is to introduce a SAV
$
	r(t)=\sqrt{E_{1}(\phi)+C_{0}}>0,
$
and expand  \eqref{eq:model-problem} into the following equivalent system
\begin{equation}\label{eq:model-problem-SAV}
	\left\{\begin{aligned} 
		& \frac{\partial \phi}{\partial t} =-\mathcal{G} \mu, \\ 
		& \mu =\mathcal{L} \phi+ \frac{r(t)}{\sqrt{E_{1}(\phi)+C_{0}}} F^{\prime}(\phi), \\ 
		& r_{t} =\frac{1}{2 \sqrt{E_{1}(\phi)+C_{0}}} \int_{\Omega} F^{\prime}(\phi) \phi_{t} \mathrm{d} \mathbf{x}. 
	\end{aligned}\right.
\end{equation}
Then, instead of descretizing  \eqref{eq:model-problem}, we can descretize the expanded system \eqref{eq:model-problem-SAV} which, with the additional SAV $r(t)$, allows us to construct efficient and  unconditional energy  stable schemes. For example,  a first-order semi-discrete scheme for  \eqref{eq:model-problem-SAV} can be constructed as follows 
\begin{eqnarray}
	\label{eq:SAV-1}
	& & \frac{\phi^{n+1}-\phi^{n}}{\delta t} =-\mathcal{G} \mu^{n+1}, \\
	\label{eq:SAV-2}
	& & \mu^{n+1} =\mathcal{L} \phi^{n+1}+\frac{r^{n+1}}{\sqrt{E_{1}(\phi^{n})+C_{0}}} F^{\prime}(\phi^{n}), \\ 
	\label{eq:SAV-3}
	& & \frac{ r^{n+1}-r^{n}}{\delta t} =\frac{1}{2 \sqrt{E_{1}(\phi^{n})+C_{0}}} \int_{\Omega} F^{\prime}(\phi^{n}) \frac{ \phi^{n+1}-\phi^{n}}{\delta t} \mathrm{d} \mathbf{x}. 
\end{eqnarray}
It can be easily shown that the above scheme is unconditionally energy stable with a modified energy $\tilde E(\phi^n)=\frac{1}{2}(\mathcal{L} \phi^n, \phi^n)+r^n$, and only requires solving two linear systems with constant coefficients. 

 While the expanded system and the original system are equivalent at the continuous level,   $r^{n+1}$ is not directly linked to   
$\sqrt{\int_\Omega F(\phi^{n+1}) dx +C_0}$ and may take very different values, if the time step is not sufficiently small, and lead to inaccurate solutions. In particular, for fixed $\delta t$, the ratio $\frac{ r^{n+1}}{\sqrt{\int_\Omega F(\phi^{n+1}) dx +C_0}}$ may converge to a value different from 1, leading to a wrong steady state solution (see Table 1 in \cite{MR3979103}).
One possible remedy for this is to monitor the ratio $\frac{ r^{n+1}}{\sqrt{\int_\Omega F(\phi^{n+1}) dx +C_0}}$ adaptively to ensure that it is always close to 1 at every time step. Another remedy is to use a Lagrange multiplier  SAV appraoch \cite{cheng2020new} such that it dissipates the original energy but it involves solving a nonlinear algebraic equation  at each time step which may not admit a suitable solution when the time step is not sufficiently small.

Recently, an interesting relaxed SAV (R-SAV) approach was introduced in \cite{jiang2022improving}. The idea is to add a relaxation step to the original SAV approach to link $r^{n+1}$ with $\sqrt{\int_\Omega F(\phi^{n+1}) dx +C_0}$ so that the updated sequence  $r^{n+1}$ is directly linked to $\sqrt{\int_\Omega F(\phi^{n+1}) dx +C_0}$ in some way and is still dissipative. The cost of this relaxation step is negligible while numerical results in \cite{jiang2022improving} show that the R-SAV approach can significantly improve the accuracy of the original SAV approach. However, this R-SAV approach is based on the original SAV approach which has two limitations/shortcomings: (i) it only applies to gradient flows; and (ii) it requires solving two linear systems at each time step. The generalized SAV (GSAV) approach proposed in \cite{huang2020highly,huang2021implicit} overcomes the above limitations/shortcomings while keeping the essential advantages of the original SAV approach, but it also suffers from the same problem that the computed SAV  is not directly linked to the free energy and may lead to non accurate solutions as in the original SAV approach.

The main purpose of this paper is to construct a relaxed GSAV  (R-GSAV) approach which links the SAV directly to the free energy, and  enjoys the following advantages: 
\begin{itemize}
\item  It is unconditionally energy stable with respect to a modified energy that is closer to the original free energy, and provides a much improved accuracy when compared with the GSAV approach;
\item  it only requires solving one linear system with constant coefficients as opposed to the two linear systems by the R-SAV approach, so its computational cost is essentially half of the R-SAV approach; 
\item  it can be applied to general dissipative systems, and  its higher-order versions are shown to be  unconditionally energy stable with optimal error estimates. 
\end{itemize}
Moreover, our numerical results indicate that, for the ample numerical experiments that we tested,  the modified energy of our R-GSAV schemes equals to the original free energy at almost all times.

The rest of this paper is organized as follows. 
In section \ref{sec:Brief-review}, we provide a brief review of the original SAV approach and R-SAV appraoch for gradient flows, and the GSAV approach for general dissipative systems. 
In Section \ref{sec:R-IMEX-SAV}, we present  the R-GSAV approach for general dissipative systems, and prove that the $k$th-order implicit-explicit (IMEX) schemes based on the  R-GSAV approach is unconditionally stable. 
In Section \ref{sec:Error-estimate}, we carry out an error analysis for the $k$th order  ($1\le k\le 5$) IMEX  schemes for Allen-Cahn type  and Cahn-Hilliard type equations. In Section \ref{sec:msav}, we extend the R-GSAV approach to cases where multiple SAVs are used. We present in
 Section \ref{sec:Numerical-examples} comparisons of R-GSAV  approach with original SAV, R-SAV and GSAV approaches, and provide ample numerical examples to validate its effectiveness. 

 \section{A brief review of the original SAV, relaxed SAV and generalized SAV approaches}\label{sec:Brief-review}
 In order to motivate our new schemes, we briefly review below the  original SAV, relaxed SAV and generalized SAV approached.
\subsection{The original SAV  approach} 
A brief description of the original SAV approach  for 
 gradient flows  is already provided in the introduction where a first-order scheme is introduced. 
Similarly, we can construct $k$-th order IMEX schemes for the  expanded system \eqref{eq:model-problem-SAV} as follows:
\begin{eqnarray}
\label{eq:SAV-BDFk-1}
& & \frac{\alpha_{k} \phi^{n+1}-A_{k}\left(\phi^{n}\right)}{\delta t} =-\mathcal{G} \mu^{n+1}, \\
\label{eq:SAV-BDFk-2}
& & \mu^{n+1} =\mathcal{L} \phi^{n+1}+\frac{r^{n+1}}{\sqrt{E_{1}(B_{k}(\phi^{n}))+C_{0}}} F^{\prime}(B_{k}(\phi^{n})), \\ 
\label{eq:SAV-BDFk-3}
& & \frac{\alpha_{k} r^{n+1}-A_{k}\left(r^{n}\right)}{\delta t} =\frac{1}{2 \sqrt{E_{1}(B_{k}(\phi^{n}))+C_{0}}} \int_{\Omega} F^{\prime}(B_{k}(\phi^{n})) \frac{\alpha_{k} \phi^{n+1}-A_{k}\left(\phi^{n}\right)}{\delta t} \mathrm{d} \mathbf{x}, 
\end{eqnarray}
where $\alpha_{k}$,  $A_{k}$ and $B_{k}$ 
can  be derived  by Taylor expansion. For the readers' convenience, we provide them for $k=1,2,3$ below:

First-order:
\begin{equation} \label{eq:perameter-BDF1}
	\alpha_{1}=1, \quad A_{1}\left(\phi^{n}\right)=\phi^{n}, \quad B_{1}\left(\phi^{n}\right)=\phi^{n};
\end{equation}
 
Second-order:
\begin{equation}\label{eq:perameter-BDF2}
	\alpha_{2}=\frac{3}{2}, \quad A_{2}\left(\phi^{n}\right)=2 \phi^{n}-\frac{1}{2} \phi^{n-1}, \quad B_{2}\left(\phi^{n}\right)=2 \phi^{n}-\phi^{n-1}.
\end{equation}

Third-order:
\begin{equation}\label{eq:perameter-BDF3}
	\alpha_{3}=\frac{11}{6}, \quad A_{3}\left(\phi^{n}\right)=3 \phi^{n}-\frac{3}{2} \phi^{n-1}+\frac{1}{3} u^{n-2}, \quad B_{3}\left(\bar{\phi}^{n}\right)=3 \bar{\phi}^{n}-3 \bar{\phi}^{n-1}+\bar{\phi}^{n-2}.
\end{equation}

It has been shown that the scheme \eqref{eq:SAV-BDFk-1}-\eqref{eq:SAV-BDFk-3} for $k=1,2$ is unconditionally energy stable with a modified energy. 
For example, the modified energy is $\tilde E(\phi^{n+1}) =\frac 12 (\mathcal L \phi^{n+1},\phi^{n+1}) +r^{n+1}$ for $k=1$, where $r^{n+1}$ is only weakly linked to
$\sqrt{E_{1}(B_{k}(\phi^{n+1}))+C_{0}}$.  
The consequence is that when the time step is not sufficiently small, the modified energy can deviate far away from the  original energy, leading to inaccurate solutions.
\subsection{The relaxed SAV approach} 
The relaxed SAV (R-SAV) approach proposed  in  \cite{jiang2022improving} is as follows:

Step 1. Calculate the solution $\left(\phi^{n+1}, \tilde r^{n+1}\right)$ based on original SAV approach
\begin{eqnarray}
\label{eq:R-SAV-BDFk-1}
& & \frac{\alpha_{k} \phi^{n+1}-A_{k}\left(\phi^{n}\right)}{\delta t} =-\mathcal{G} \mu^{n+1}, \\
\label{eq:R-SAV-BDFk-2}
& & \mu^{n+1} =\mathcal{L} \phi^{n+1}+\frac{\tilde r^{n+1}}{\sqrt{E_{1}(B_{k}(\phi^{n}))+C_{0}}} F^{\prime}(B_{k}(\phi^{n})), \\ 
\label{eq:R-SAV-BDFk-3}
& & \frac{\alpha_{k} \tilde r^{n+1}-A_{k}\left(r^{n}\right)}{\delta t} =\frac{1}{2 \sqrt{E_{1}(B_{k}(\phi^{n}))+C_{0}}} \int_{\Omega} F^{\prime}(B_{k}(\phi^{n})) \frac{\alpha_{k} \phi^{n+1}-A_{k}\left(\phi^{n}\right)}{\delta t} \mathrm{d} \mathbf{x}, 
\end{eqnarray}

Step 2. Update the SAV $r^{n+1}$ by a relaxation
\begin{equation} \label{eq:R-SAV-BDFk-4}
r^{n+1} = \zeta_{0} \tilde r^{n+1} + \left(1-\zeta_{0}\right) E_{1}\left(\phi\right), \quad \zeta_{0} \in \mathcal{V}.
\end{equation}
Here, $\mathcal{V}$ is a set defined by 
\begin{equation}
	\mathcal{V}=\left\lbrace \zeta\in [0,1] \; s.t. \; \left|r^{n+1}\right|^{2} -\left|\tilde r^{n+1}\right|^{2} \leq \delta t \gamma \left(\mathcal{G} \mu^{n+1}, \mu^{n+1}\right) \right\rbrace
\end{equation}
for BDF$1$ scheme, and 
\begin{equation}
	\mathcal{V}=\left\lbrace \zeta\in [0,1] \; s.t. \;  \frac{1}{2}\left(\left|r^{n+1}\right|^{2} + \left|2r^{n+1}-r^{n}\right|^{2} -\left|\tilde r^{n+1}\right|^{2} - \left|2 \tilde r^{n+1}-\tilde r^{n}\right|^{2}\right) \leq \delta t \gamma \left(\mathcal{G} \mu^{n+1}, \mu^{n+1}\right) \right\rbrace
\end{equation}
for BDF$2$ scheme, where $\gamma \in [0, 1]$ is a tunable parameter. 

\textbf{Remark 2.1 (Optimal choice for $\zeta_{0}$)}. 
Here we explain the optimal choice for relaxation parameter $\zeta_{0}$. 
Let's take BDF$2$ as an example, $\zeta_{0}$ can be chosen as a solution of the following optimization problem,
\begin{equation}
\xi_{0}=\min _{\xi \in[0,1]} \xi, \quad \text { s.t. }\frac{1}{2}\left(\left|r^{n+1}\right|^{2}+\left|2 r^{n+1}-r^{n}\right|^{2}-\left(\left|\tilde{r}^{n+1}\right|^{2}+\left|2 \tilde{r}^{n+1}-r^{n}\right|^{2}\right)\right) \leq \delta \operatorname{t\gamma}\left(\mathcal{G} \mu^{n+1}, \mu^{n+1}\right),
\end{equation}
with $r^{n+1} = \zeta_{0} \tilde r^{n+1} + \left(1-\zeta_{0}\right) E_{1}\left(\phi^{n+1}\right)$. 
This can be simplified as
\begin{equation}
\xi_{0}=\min _{\xi \in[0,1]} \xi, \quad \text { s.t. } a \xi^{2}+b \xi+c \leq 0,
\end{equation}
where the coefficients are 
\begin{equation}
\begin{array}{l}
a=\frac{5}{2}\left(\tilde{r}^{n+1}-E_{1}\left(\phi^{n+1}\right)\right)^{2}, \\ 
b=\left(\tilde{r}^{n+1}-E_{1}\left(\phi^{n+1}\right)\right)\left(5 E_{1}\left(\phi^{n+1}\right)-2 r^{n}\right), \\ 
c=\frac{1}{2}\left(\left(E_{1}\left(\phi^{n+1}\right)\right)^{2}+\left(2 E_{1}\left(\phi^{n+1}\right)-r^{n}\right)^{2}-\left(\tilde{r}^{n+1}\right)^{2}-\left(2 \tilde{r}^{n+1}-r^{n}\right)^{2}\right)-\delta t \gamma\left(\mathcal{G} \mu^{n+1}, \mu^{n+1}\right).\end{array}
\end{equation}
Notice that $a+b+c=-\delta t \gamma\left(\mathcal{G} \mu^{n+1}, \mu^{n+1}\right) \leq 0$, we have $1\in \mathcal{V}$ which means $\mathcal{V} \neq \emptyset$. 
Given $a \neq 0$, the optimization problem can be solved as 
\begin{equation}
\xi_{0}=\max \left\{0, \frac{-b-\sqrt{b^{2}-4 a c}}{2 a}\right\}.
\end{equation}
Since $$b^{2}-4 a c=\left(E_{1}\left(\phi^{n+1}\right) - \tilde{r}^{n+1}\right)\left(4\left(r^{n}\right)^{2} - 20r^{n}\tilde{r}^{n+1} + 25\left(\tilde{r}^{n+1}\right)^{2} + 10 \delta t \gamma \left(\mathcal{G} \mu^{n+1}, \mu^{n+1}\right)\right), $$ 
we usually take the $\gamma$ to be close to 1. 

It can be shown that the scheme \eqref{eq:R-SAV-BDFk-1}-\eqref{eq:R-SAV-BDFk-4} with $k=1,2$ is unconditionally energy stable in the sense that\\
(\romannumeral1) For $k=1$, $R_{R-SAV-BDF1}^{n+1} - R_{R-SAV-BDF1}^{n}\leq-\delta t\left(1-\gamma\right)\left(\mathcal{G} \mu^{n+1}, \mu^{n+1}\right)$, where $$R_{R-SAV-BDF1}^{n+1}=\frac{1}{2}\left(\mathcal{L} \phi^{n+1}, \phi^{n+1}\right) +\left|r^{n+1}\right|^{2};$$\\
(\romannumeral2) For $k=2$, $R_{R-SAV-BDF2}^{n+1} - R_{R-SAV-BDF2}^{n}\leq-\delta t\left(1-\gamma\right)\left(\mathcal{G} \mu^{n+1}, \mu^{n+1}\right)$,
where 
$$
\begin{aligned}
	R_{R-SAV-BDF2}^{n+1}&=\frac{1}{4}\left(\left(\mathcal{L} \phi^{n+1}, \phi^{n+1}\right) + \left(\mathcal{L}\left(2\phi^{n+1}-\phi^{n}\right), 2\phi^{n+1}-\phi^{n}\right)\right) \\
	&  +\frac{1}{2}\left(\left|r^{n+1}\right|^{2} + \left|2r^{n+1}-r^{n}\right|^{2}\right).
\end{aligned}
$$ 
Note that by the definition in \eqref {eq:R-SAV-BDFk-4}, $r^{n+1}$ is directly linked to $E_1(\phi^{n+1})$. Hence, the modified energy above is directly linked to the original free energy.

\subsection{The generalized SAV (GSAV)  approach}

The GSAV approach was proposed in \cite{huang2021implicit} for  the following general dissipative systems:
\begin{align}\label{eq: new1}
 \frac{\partial \phi}{\partial t}+\mathcal{A} \phi +g(\phi)=0, 
\end{align} 
where $\mathcal{A}$ is a positive operator and $g(\phi)$ is a semi-linear or quasil-inear operator. We assume it  satisfies an energy dissipation law as follows
\begin{align}
	\frac{d E_{tot}(\phi)}{dt} =-\mathcal{K}(\phi),\label{eq: new2}
\end{align} 
where $E_{tot}(\phi)$ is a free energy with lower bound $-C_0$ and $\mathcal{K}(\phi)>0$ for all $\phi$.
Setting $E(\phi) = E_{tot}(\phi)+C_0$ and introducing a SAV $R(t)=E(\phi)$,  we rewrite the equation \eqref{eq: new1} as the following system
\begin{equation}\label{eq:model-problem-new-SAV}
\left\{\begin{aligned}
& \frac{\partial \phi}{\partial t}+\mathcal{A} (\phi)+g(\phi)=0, \\
& \frac{\mathrm{d} E(\phi)}{\mathrm{d} t}=-\frac{R(t)}{E(\phi)}\mathcal{K}(\phi).
\end{aligned}\right.
\end{equation}
Then, the BDF$k$ GSAV schemes are as follows:
\begin{eqnarray}
\label{eq:new-SAV-BDFk-1}
& & \frac{\alpha_{k} \bar{\phi}^{n+1}-A_{k}\left(\phi^{n}\right)}{\delta t}+\mathcal{A} \bar{\phi}^{n+1} +g\left[B_{k}\left(\bar{\phi}^{n}\right)\right]=0, \\
\label{eq:new-SAV-BDFk-3}
& & \frac{1}{\delta t}\left(R^{n+1}-R^{n}\right)=-\frac{R^{n+1}}{E\left(\bar{\phi}^{n+1}\right)} \mathcal{K}( \bar{\phi}^{n+1}), \\
\label{eq:new-SAV-BDFk-4} 	
& & \xi^{n+1}=\frac{R^{n+1}}{E\left(\bar{\phi}^{n+1}\right)}, \\
\label{eq:new-SAV-BDFk-5}
& & \phi^{n+1}=\eta_{k}^{n+1} \bar{\phi}^{n+1} \text { with } \eta_{k}^{n+1}=1-\left(1-\xi^{n+1}\right)^{k+1},
\end{eqnarray}
where $\alpha_{k}$, the operators $A_{k}$ and $B_{k}$ are as above.

It is shown in \cite{huang2021implicit} that the above scheme is unconditional energy stable with a modified energy. However, as the original SAV approach, the dynamics of SAV $R^{n+1}$ is only weakly linked to the energy $E(\phi^{n+1})$ and may deviate from it when the time step is not sufficiently small.

\section{The relaxed generalized  (R-GSAV) SAV approach} \label{sec:R-IMEX-SAV}
Inspired by the R-SAV approach described above, we construct below the relaxed GSAV approach, which not only inherits all the advantages of the GSAV approach, but can also significantly improve its accuracy.

%

Given $\phi^{n-k}, ..., \phi^{n}, R^{n-k}, ..., R^{n}$, we compute $\phi^{n+1}, R^{n+1}$ via the following two steps:

\textbf{Step 1:} Determine an intermediate solution $(\phi^{n+1}, \tilde{R}^{n+1})$ by using  the GSAV method.
\begin{eqnarray}
\label{eq:R-IMEX-SAV-BDFk-1}
& & \frac{\alpha_{k} \bar{\phi}^{n+1}-A_{k}\left(\phi^{n}\right)}{\delta t}+\mathcal{A} \bar{\phi}^{n+1}+g\left[B_{k}\left(\bar{\phi}^{n}\right)\right]=0, \\
\label{eq:R-IMEX-SAV-BDFk-3}
& & \frac{1}{\delta t}\left(\tilde{R}^{n+1}-R^{n}\right)=-\frac{\tilde{R}^{n+1}}{E\left(\bar{\phi}^{n+1}\right)}  \mathcal{K}( \bar{\phi}^{n+1}), \\ 
\label{eq:R-IMEX-SAV-BDFk-4}	
& & \xi^{n+1}=\frac{\tilde{R}^{n+1}}{E\left(\bar{\phi}^{n+1}\right)}, \\
\label{eq:R-IMEX-SAV-BDFk-5}
& & \phi^{n+1}=\eta_{k}^{n+1} \bar{\phi}^{n+1} \text { with } \eta_{k}^{n+1}=1-\left(1-\xi^{n+1}\right)^{k+1},
\end{eqnarray}

\textbf{Step 2:} Update the SAV $R^{n+1}$ via the following relaxation: 
\begin{equation} \label{eq:R-IMEX-SAV-BDFk-6}
	R^{n+1} = \zeta_{0} \tilde{R}^{n+1} + (1-\zeta_{0})E(\phi^{n+1}), \quad \zeta_{0} \in \mathcal{V},
\end{equation}
where
\begin{eqnarray}
\label{eq:set-condition-2}
\mathcal{V}=\left\lbrace \zeta\in [0,1] \; s.t. \;   \frac{R^{n+1}-\tilde{R}^{n+1}}{\delta t} \leq -\gamma^{n+1} \mathcal{K}( {\phi}^{n+1})+ \frac{\tilde{R}^{n+1}}{E(\bar{\phi}^{n+1})} \mathcal{K}( \bar{\phi}^{n+1})\right\rbrace,
\end{eqnarray}
with  $\gamma^{n+1} \geq 0$ to be determined so that $\mathcal{V}$ is not empty.

 We explain below how to choose   $\zeta_{0}$ and  $\gamma^{n+1}$. 
Plugging \eqref{eq:R-IMEX-SAV-BDFk-6} into the inequality of \eqref{eq:set-condition-2}, we find that if we choose  $\zeta_0$ and $\gamma^{n+1}$ such that the following condition is satisfied:
\begin{equation}\label{cond_zeta}
(\tilde{R}^{n+1}-E(\phi^{n+1}))\zeta_0 \leq \tilde{R}^{n+1}-E(\phi^{n+1})-\delta t \gamma^{n+1} \mathcal{K}( {\phi}^{n+1})+ \delta t \frac{\tilde{R}^{n+1}}{E(\bar{\phi}^{n+1})} \mathcal{K}( \bar{\phi}^{n+1}),
\end{equation}
then, $\zeta_0\in \mathcal{V}$. The next theorem summarizes the choice of $\zeta_0$ and $\gamma^{n+1}$.

\begin{thm}\label{Th: stability}
We choose $\zeta_{0}$ in \eqref{eq:R-IMEX-SAV-BDFk-6} and $\gamma^{n+1}$ in \eqref{eq:set-condition-2} as follows:
\begin{enumerate}
\item If $\tilde{R}^{n+1} = E(\phi^{n+1})$, we set $\zeta_{0}=0$ and $\gamma^{n+1}=\frac{\tilde{R}^{n+1} \mathcal{K}( \bar{\phi}^{n+1})}{E(\bar{\phi}^{n+1}) \mathcal{K}( {\phi}^{n+1})}$;
\item If $\tilde{R}^{n+1} > E(\phi^{n+1})$, we set $\zeta_{0}=0$ and
\begin{equation}\label{eq:gamma}
\gamma^{n+1} = \frac{\tilde{R}^{n+1}-E(\phi^{n+1})}{\delta t \mathcal{K}(\phi^{n+1})} + \frac{\tilde{R}^{n+1} \mathcal{K}( \bar{\phi}^{n+1})}{E(\bar{\phi}^{n+1}) \mathcal{K}( {\phi}^{n+1})};
\end{equation}

\item If $\tilde{R}^{n+1} < E(\phi^{n+1})$ and $\tilde{R}^{n+1}-E(\phi^{n+1})+ \delta t \frac{\tilde{R}^{n+1}}{E(\bar{\phi}^{n+1})} \mathcal{K}( \bar{\phi}^{n+1}) \geq 0$, we set $\zeta_{0}=0$ and 	
$\gamma^{n+1}$ the same as \eqref{eq:gamma}.  
	\item  If $\tilde{R}^{n+1} < E(\phi^{n+1})$ and $\tilde{R}^{n+1}-E(\phi^{n+1})+ \delta t \frac{\tilde{R}^{n+1}}{E(\bar{\phi}^{n+1})} \mathcal{K}( \bar{\phi}^{n+1})< 0$, we set $\zeta_{0}=1-\frac{\delta t \tilde{R}^{n+1} \mathcal{K}( \bar{\phi}^{n+1})}{E(\bar{\phi}^{n+1})\left(E(\phi^{n+1})-\tilde{R}^{n+1}\right)}$ and 
$\gamma^{n+1}=0$.
	\end{enumerate}
Then,    \eqref{cond_zeta} is satisfied in all cases above and  $\zeta_0\in \mathcal{V}$. Furthermore,  	given $R^{n} \geq 0$, 
we have $R^{n+1}\geq 0,\, \xi^{n+1} \geq 0$, and the scheme \eqref{eq:R-IMEX-SAV-BDFk-1}-\eqref{eq:R-IMEX-SAV-BDFk-6} with the above choice of  $\zeta_0$ and  $\gamma^{n+1}$  is unconditionally energy stable in the sense that
\begin{equation}\label{eq:stability}
	R^{n+1} - R^{n} \leq -\delta t \gamma^{n+1} \mathcal{K}( {\phi}^{n+1})\leq 0.
\end{equation}
Furthermore, we have 
\begin{equation}\label{eq:stability2}
	R^{n+1}\le E(\phi^{n+1})\quad\forall n\ge 0.
\end{equation}
In addition, 
if $E(\phi)=(\mathcal{L}\phi,\phi)+E_1(\phi)$ with $\mathcal{L}$ being a linear positive operator and  $E_1(\phi)$ bounded from below, there exists $M_{k}>0$ such that
\begin{equation}\label{eq:numerical-solution-bound}
\left(\mathcal{L} \phi^{n}, \phi^{n}\right) \leq M_{k}^{2}, \quad \forall n.
\end{equation}
\end{thm}

\begin{proof}
	It can be directly verified that,  with the above choice of   $\zeta_{0}$  and $\gamma^{n+1}$,   \eqref{cond_zeta} is satisfied in all cases so that   $\zeta_0\in \mathcal{V}$.
	
Given $R^{n} \geq 0$.	 Since $E(\bar{\phi}^{n+1})>0$, it follows from \eqref{eq:R-IMEX-SAV-BDFk-3} that 
\begin{equation}\label{tilder}
\tilde{R}^{n+1}=\frac{R^{n}}{1+\delta t \frac{ \mathcal{K}( \bar{\phi}^{n+1})}{E(\bar{\phi}^{n+1})}} \geq 0.
\end{equation}
Then we derive from \eqref{eq:R-IMEX-SAV-BDFk-4} that $\xi^{n+1}\ge 0$, and we  derive from \eqref{eq:R-IMEX-SAV-BDFk-6} that $R^{n+1} \geq 0$.

Combining \eqref{eq:R-IMEX-SAV-BDFk-3} and \eqref{eq:set-condition-2}, we obtain \eqref{eq:stability}. 

For Cases 1-3,  we have $\zeta_0=0$ so $R^{n+1}= E(\phi^{n+1})$. For Case 4, since 
$\zeta_{0}=1-\frac{\delta t \tilde{R}^{n+1} \mathcal{K}( \bar{\phi}^{n+1})}{E(\bar{\phi}^{n+1})\left(E(\phi^{n+1})-\tilde{R}^{n+1}\right)}\in [0,1]$ and $\tilde R^{n+1}\le  E(\phi^{n+1})$, we derive from \eqref{eq:R-IMEX-SAV-BDFk-6} that $R^{n+1}\le E(\phi^{n+1})$.

The proof of \eqref{eq:numerical-solution-bound} is essentially the same as the proof of  Theorem 1 of the GSAV scheme in \cite{huang2021implicit}. For the readers' convenience, we provide the proof below.

Denote $M:=R^0=E[\phi(\cdot,0)]$, then we derive from  \eqref{eq:stability} and \eqref{tilder} that  $\tilde R^{n+1} \le M,\, \forall n$.

Without loss of generality, we can assume $E_1(\phi)>1$ for all $\phi$. It then follows from \eqref{eq:R-IMEX-SAV-BDFk-4}  that
\begin{equation}\label{eq: xibound}
	|\xi^{n+1}|=\frac{\tilde R^{n+1}}{E(\bar{\phi}^{n+1})} \le \frac{2M}{(\mathcal{L}\bar{\phi}^{n+1}, \bar{\phi}^{n+1})+2}.
\end{equation}
Then $\eta_k^{n+1}=1-\left(1-\xi^{n+1}\right)^{k+1}=\xi^{n+1}P_{k}(\xi^{n+1})$ with   $P_{k}$ being a polynomial  of degree $k$. We derive from \eqref{eq: xibound} that  there exists $M_k>0$ such that
\begin{equation*}\label{eq: etabound}
	|\eta_k^{n+1}|= |\xi^{n+1}P_{k}(\xi^{n+1})| \le \frac{{M_k}}{(\mathcal{L}\bar \phi^{n+1}, \bar \phi^{n+1})+2},
\end{equation*}
which, along with  $\phi^{n+1}=\eta_k^{n+1}\bar \phi^{n+1}$, implies
\begin{equation*}
	\begin{split}
		(\mathcal{L}\phi^{n+1}, \phi^{n+1})&=(\eta_k^{n+1})^2(\mathcal{L}\bar{\phi}^{n+1}, \bar{\phi}^{n+1}) \\
		&\le \big(\frac{{M_k}}{(\mathcal{L}\bar \phi^{n+1}, \bar \phi^{n+1})+2}\big)^2(\mathcal{L}\bar{\phi}^{n+1}, \bar{\phi}^{n+1}) \le M^2_k.
	\end{split}
\end{equation*}
The proof is complete.
\end{proof}

\begin{remark}
From an energy approximation point of view, it is best that $R^{n+1}=E(\phi^{n+1})$. However, simply setting 
$R^{n+1}=E(\phi^{n+1})$ at each time step destroys the energy stability. 	We observe from the above statements  that   in most cases (Cases 1, 2, 3),  we have  $\zeta_0=0$ which implies $R^{n+1}=E(\phi^{n+1})$,
and then thanks to  \eqref{eq:stability} and \eqref{eq:stability2}, we have
\begin{equation}
	E(\phi^{n+1})=R^{n+1}\le R^n\le E(\phi^n) \quad (\text{Cases 1, 2, 3}).
\end{equation}
Only in   Case 4 where $\tilde R^{n+1}$ is significantly smaller than $E(\phi^{n+1})$ , $R^{n+1}$ is a value between  ($\tilde R^{n+1}$, $E(\phi^{n+1})$), and we can not prove $E(\phi^{n+1})\le E(\phi^n)$.
Hence, the original energy is proved to be dissipative in most situations, which is a significant improvement over  the GSAV scheme.
\end{remark}

\section{Error estimate}\label{sec:Error-estimate}
In this section, we will derive error estimates for the R-GSAV schemes applied to  Allen-Cahn type equations and Cahn-Hilliard equations by using the stability results in  Theorem \ref{Th: stability}. 

 We recall first some preliminary results which play a critical role in error analysis.

\begin{lem}\label{lem:stability-tool}  \cite{nevanlinna1981multiplier}
	For $1 \leq k \leq 5$, there exist $0 \leq \tau_k<1$, a positive definite symmetric matrix $G=(g_{ij})\in\mathcal{R}^{k,k}$ and real numbers $a_0, ..., a_k$ such that 
$$
\begin{aligned}
\left(\alpha_{k} \phi^{n+1}-A_{k}\left(\phi^{n}\right), \phi^{n+1}-\tau_{k} \phi^{n}\right) &=\sum_{i, j=1}^{k} g_{i j}\left(\phi^{n+1+i-k}, \phi^{n+1+j-k}\right) \\ &-\sum_{i, j=1}^{k} g_{i j}\left(\phi^{n+i-k}, \phi^{n+j-k}\right)+\left\|\sum_{i=0}^{k} a_{i} \phi^{n+1+i-k}\right\|^{2}, 
\end{aligned}
$$
where the smallest possible values of $\tau_k$ are 
	\begin{equation} \label{tau_k}
		 \tau_1=\tau_2=0, \quad \tau_3=0.0836, \quad \tau_4 =0.2878, \quad \tau_5=0.8160. 
		 \end{equation}
\end{lem}
and $\alpha_k, A_k$ are constant and operator related to BDFk IMEX schemes as described in the last section. 

The following regularity results for \eqref{eq:model-problem} are given in \cite{shen2018convergence, temam2012infinite}. 
\begin{thm}\label{Th:solution-regularity}
	Assume $\phi^0 \in H^2(\Omega)$ and the following holds 
	\begin{equation} \label{eq:assumption-1}
	\left|F^{\prime \prime}(x)\right|<C\left(|x|^{p}+1\right), \quad p>0 \text { arbitrary if } d=1,2 ; \quad 0<p<4 \text { if } d=3.
	\end{equation}
	Then for $\mathcal{G}=I$, problem \eqref{eq:model-problem} has a unique solution for any $T>0$ in the space
	\begin{equation}
	C\left([0, T] ; H^{2}(\Omega)\right) \cap L^{2}\left(0, T ; H^{3}(\Omega)\right).
	\end{equation}
	Furthermore,  assume 
	\begin{equation} \label{eq:assumption-2}
	\left|F^{\prime \prime \prime}(x)\right|<C\left(|x|^{p^{\prime}}+1\right), \quad p^{\prime}>0 \text { arbitrary if } d=1,2 ; \quad 0<p^{\prime}<3 \text { if } d=3.
	\end{equation}
	Then for $\mathcal{G}=-\Delta$, there exists a unique solution for any $T>0$ such that 
	\begin{equation}
		\phi \in C\left([0, T] ; H^{2}(\Omega)\right) \cap L^{2}\left(0, T ; H^{4}(\Omega)\right).
	\end{equation}
\end{thm}

We also recall the following useful results to deal with the nonlinear term \cite{shen2018convergence}.
\begin{thm}
	Assume that $\|\phi\|_{H^1} \leq M$.
	\begin{itemize}
		\item Assume that \eqref{eq:assumption-1} holds. Then for any $\phi \in H^3$, there exists $0 \leq \sigma<1$ and a constant $C(M)$ such that the following inequality holds
		\begin{equation}
			\|\nabla F^{\prime}(\phi)\|^{2} \leq C(M)\left(1+\|\nabla \Delta \phi\|^{2 \sigma}\right),
		\end{equation}
		\item Assume that \eqref{eq:assumption-1} and \eqref{eq:assumption-2} hold. Then, for any $\phi \in H^{4}$, there exists $0 \leq \sigma<1$ and a constant $C(M)$ such that the following inequality holds
		\begin{equation}
			\|\Delta F^{\prime}(\phi)\|^{2} \leq C(M)\left(1+\|\Delta^{2} \phi\|^{2 \sigma}\right).
		\end{equation}
	\end{itemize}
\end{thm}

We consider first the Allen-Cahn type equation
\begin{equation} \label{eq:Allen-Cahn-eq}
\frac{\partial \phi}{\partial t}-\Delta \phi+\lambda \phi+F^{\prime}(\phi)=0 \quad(\boldsymbol{x}, t) \in \Omega \times(0, T],
\end{equation}
which is \eqref{eq:model-problem} with $\mathcal{L}=-\Delta+\lambda I$ and $\mathcal{G}=I$. It can also be written in the form of \eqref{eq: new1} with $\mathcal{A}=-\Delta+\lambda I$ and $g(\phi)=F^{\prime}(\phi)$.
The corresponding \eqref{eq: new2} is with $E_{tot}(\phi)=\int_\Omega \frac12|\nabla\phi|^2+\lambda |\phi|^2+F(\phi) dx$ and $\mathcal{K}(\phi)=\|-\Delta  \phi+\lambda  \phi+ F^{\prime}(\phi)\|^2$.


In the following, we follow a similar procedure as in  \cite{huang2021implicit} to carry out a unified error analysis for the R-GSAV/BDF$k$ ($1 \leq k \leq 5$) defined by \eqref{eq:R-IMEX-SAV-BDFk-1}-\eqref{eq:R-IMEX-SAV-BDFk-6}. 

\begin{thm}
	Given initial conditions $\bar{\phi}^{i}=\phi^{i}=\phi(t^i), r^{i}=E(\phi^{i}), i=0, 1, ..., k-1$. 
	Let $\bar{\phi}^{n+1}$ and $\phi^{n+1}$ be computed with the R-GSAV/BDF$k$ ($1 \leq k \leq 5$) scheme \eqref{eq:R-IMEX-SAV-BDFk-1}-\eqref{eq:R-IMEX-SAV-BDFk-6} for \eqref{eq:Allen-Cahn-eq} with 
	\begin{equation}
	\eta_{1}^{n+1}=1-\left(1-\xi^{n+1}\right)^{3}, \quad \eta_{k}^{n+1}=1-\left(1-\xi^{n+1}\right)^{k+1}(k=2,3,4,5).
	\end{equation}
	We assume \eqref{eq:assumption-1} holds and
	\begin{equation}
	\phi^{0} \in H^{3}, \quad \frac{\partial^{j} \phi}{\partial t^{j}} \in L^{2}\left(0, T ; H^{1}\right) 1 \leq j \leq k+1.
	\end{equation}
	Then for $\delta t < \min \lbrace \frac{1}{1+2C_0^{k+2}}, \frac{1-\tau_k}{3k} \rbrace$, we have 
	$$
	\left\|\bar{\phi}^{n}-\phi(\cdot, t^{n})\right\|_{H^{2}}, \quad \left\|\phi^{n}-\phi(\cdot, t^{n})\right\|_{H^{2}} \leq C \delta t^{k}, \quad\forall  n \leq T/\delta t,
	$$
	where $\tau_k$ is given in \eqref{tau_k}, and the constants $C_0, C$ are  independent of $\delta t$.
\end{thm}
\begin{proof}

	We denote $t^n=n \delta t, \bar{e}^{n}=\bar{\phi}^{n}-\phi(\cdot, t^{n}), e^{n}=\phi^{n}-\phi(\cdot, t^{n}), \tilde{s}^{n}=\tilde{R}^{n}-R(t^n), s^{n}=R^{n}-R(t^n).$

	We will prove by introduction
	\begin{equation}\label{eq:xi-results}
	\left|1-\xi^{q}\right| \leq C_{0} \delta t, \quad \forall q \leq T / \delta t
	\end{equation}
	 where $C_0$ is dependent on $T, \Omega$ and the exact solution $\phi$ but is independent of $\delta t$. 
	Under the assumption, \eqref{eq:xi-results} holds for $q=0$. 
	Assuming
	\begin{equation} \label{eq:xi-assumption}
	\left|1-\xi^{q}\right| \leq C_{0} \delta t, \quad \forall q \leq m,
	\end{equation}
	we need to prove 
	\begin{equation}\label{eq:error-estimate-8}
	\left|1-\xi^{m+1}\right| \leq C_{0} \delta t.
	\end{equation}
	The   \textbf{Step 1} and \textbf{Step 2} below are essentially the same as in \cite{huang2021implicit}. 
	So we only list the results that will be used  here and refer to \cite{huang2021implicit} for more details.
	\medskip
	
\noindent	\textbf{Step 1: $H^2$ bound for $\phi^{n}$ and $\bar{\phi}^{n}$ for all $n \leq m$.} 
			The first step of the scheme  \eqref{eq:R-IMEX-SAV-BDFk-1}-\eqref{eq:R-IMEX-SAV-BDFk-3} for  \eqref{eq:Allen-Cahn-eq} is
	\begin{equation}\label{eq:Allen-Cahn-BDFk-version-1}
		\frac{\alpha_{k} \bar{\phi}^{n+1}-A_{k}\left(\phi^{n}\right)}{\delta t}=\Delta \bar{\phi}^{n+1}-\lambda \bar{\phi}^{n+1}+F^{\prime}\left[B_{k}\left(\bar{\phi}^{n}\right)\right],
	\end{equation}
	Multiplying the above by $\eta_{k}^{n+1}$, we  obtain 
	\begin{equation} \label{eq:Allen-Cahn-BDFk-version-2}
		\frac{\alpha_{k} \phi^{n+1}-\eta_{k}^{n+1} A_{k}\left(\phi^{n}\right)}{\delta t}=\Delta \phi^{n+1}-\lambda \phi^{n+1}+\eta_{k}^{n+1} F^{\prime}\left[B_{k}\left(\bar{\phi}^{n}\right)\right].
	\end{equation}
	Under the assumption \eqref{eq:xi-assumption}, it can be shown that for  $\delta t$ small enough such that 
	\begin{equation}\label{eq:error-estimate-9}
	\delta t \leq \min \left\{\frac{1}{2 C_{0}^{k+1}}, 1\right\},
	\end{equation}
	we have 
	\begin{equation}\label{eq:eta-results}
	1-\frac{\delta t^{k}}{2} \leq\left|\eta_{k}^{q}\right| \leq 1+\frac{\delta t^{k}}{2}, \quad \left|1-\eta_{k}^{q}\right| \leq \frac{\delta t^{k}}{2}, \quad \forall q \leq m.
	\end{equation}
	Taking the inner product of \eqref{eq:Allen-Cahn-BDFk-version-2} with $\Delta^{2} \phi^{q}-\tau_{k} \Delta^{2} \phi^{q-1}$ and using Theorem \ref{Th: stability}, Lemma \ref{lem:stability-tool} and the property of symmetic positive definite matrix  $G=(g_{ij})$, we can obtain
	\begin{equation}\label{init_bound}
		\|\phi^{n}\|_{H^2} \leq C_1, \quad \forall \delta t < 1, n \leq m.
	\end{equation}
	We derive from the above and \eqref{eq:eta-results} that
	\begin{equation}\label{eq:error-estimate-11}
		\left\|\bar{\phi}^{n}\right\|_{H^{2}} \leq 2 C_{1}, \quad \forall \delta t<1, n \leq m.
	\end{equation}
	
	\medskip

\noindent	\textbf{Step 2: estimate for $\|\bar{e}^{n+1}\|_{H^2}$ for all $n \leq m$.} 
	Subtrating \eqref{eq:Allen-Cahn-BDFk-version-1} from \eqref{eq:Allen-Cahn-eq} , and using \eqref{init_bound} and \eqref{eq:error-estimate-11}, we can derive that for $\delta t < \frac{1}{C_2}$, we have
	\begin{equation}\label{eq:error-estimate-6}
	\left\|\bar{e}^{n+1}\right\|_{H^{2}} \leq \sqrt{C_{2}\left(1+C_{0}^{2 k+2}\right)} \delta t^{k}, \quad \forall 0 \leq n \leq m,
	\end{equation}
	\begin{equation}\label{eq:error-estimate-10}
	\left\|\bar{\phi}^{n+1}\right\|_{H^{2}} \leq \bar{C}, \quad \forall 0 \leq n \leq m,
	\end{equation}
	and
	\begin{equation}
	\left\|F^{\prime}\left(\bar{\phi}^{n+1}\right)\right\|,\left\|F^{\prime \prime}\left(\bar{\phi}^{n+1}\right)\right\| \leq \bar{C} \quad \forall 0 \leq n \leq m.
	\end{equation}
\medskip
\noindent	\textbf{Step 3: estimate for $1-\xi^{m+1}$.}
	By direct calculation,
	\begin{equation} \label{eq:error-estimate-5}
	R_{t t}=\int_{\Omega}\left(\left|\nabla \phi_{t}\right|^{2}+\nabla \phi \cdot \nabla \phi_{t t}+\lambda \phi_{t}^{2}+\lambda \phi \phi_{t t}+F^{\prime \prime}(\phi) \phi_{t}^{2}+F^{\prime}(\phi) \phi_{t t}\right) \mathrm{d} \mathbf{x}.
	\end{equation} 
	It follows from \eqref{eq:R-IMEX-SAV-BDFk-3} that the equation  for the errors can be written as 
	\begin{equation} \label{eq:error-estimate-1}
	\tilde{s}^{n+1}-s^{n}=\delta t\left(\left\|h\left[\phi\left(t^{n+1}\right)\right]\right\|^{2}-\frac{\tilde{R}^{n+1}}{E\left(\bar{\phi}^{n+1}\right)}\left\|h\left(\bar{\phi}^{n+1}\right)\right\|^{2}\right)+T_{1}^{n},
	\end{equation}
	where $h(\phi)=\mu=-\Delta \phi+\lambda \phi+F^{\prime}(\phi)$, and 
	\begin{equation}
	T_{1}^{n}=R\left(t^{n}\right)-R\left(t^{n+1}\right)+\delta t R_{t}\left(t^{n+1}\right)=\int_{t n}^{t^{n+1}}\left(s-t^{n}\right) R_{t t}(s) \mathrm{d} s.
	\end{equation}
	Since $R^{n}=\zeta_0 \tilde{R}^{n} + \left(1-\zeta_0\right)E\left(\phi^{n}\right)$ and $R\left(t^{n}\right)=\zeta_0 R\left(t^{n}\right) + \left(1-\zeta_0\right)E\left(\phi\left(t^{n}\right)\right)$, we obtain
	\begin{equation} \label{eq:error-estimate-2}
		s^{n} = \zeta_0 \tilde{s}^{n} + \left(1-\zeta_0\right)\left[E\left(\phi^{n}\right)-E\left(\phi\left(t^n\right)\right)\right].
	\end{equation}
	Plugging \eqref{eq:error-estimate-2} into \eqref{eq:error-estimate-1}, we obtain
	\begin{equation}
		\tilde{s}^{n+1}-\zeta_0 \tilde{s}^{n} - \left(1-\zeta_0\right)\left[E\left(\phi^{n}\right)-E\left(\phi\left(t^n\right)\right)\right]=\delta t\left(\left\|h\left[\phi\left(t^{n+1}\right)\right]\right\|^{2}-\frac{\tilde{R}^{n+1}}{E\left(\bar{\phi}^{n+1}\right)}\left\|h\left(\bar{\phi}^{n+1}\right)\right\|^{2}\right)+T_{1}^{n},
	\end{equation}
	by using the triangle inequality principle, we derive 
	\begin{equation}\label{eq:error-estimate-3}
	\begin{aligned}
	&	\left|\tilde{s}^{n+1}\right|-\zeta_0 \left|\tilde{s}^{n}\right| 
		 \leq \left|\tilde{s}^{n+1}-\zeta_0 \tilde{s}^{n}\right| \\
		& \leq \delta t\left|\left\|h\left[\phi\left(t^{n+1}\right)\right]\right\|^{2}-\frac{\tilde{R}^{n+1}}{E\left(\bar{\phi}^{n+1}\right)}\left\|h\left(\bar{\phi}^{n+1}\right)\right\|^{2}\right| + \left(1-\zeta_0\right)\left|E\left(\phi^{n}\right)-E\left(\phi\left(t^n\right)\right)\right|+\left|T_{1}^{n}\right|.
	\end{aligned}
	\end{equation}
	Taking the sum of \eqref{eq:error-estimate-3} for $n$ from $0$ to $m$, and noting that $\tilde{s}^0=0$, we have
	\begin{equation}\label{eq:error-estimate-4}
	\begin{aligned}
		& \left|\tilde{s}^{m+1}\right| + \sum_{n=1}^{m}\left(1-\zeta_0\right)\left|\tilde{s}^{n}\right|   \leq \delta t \sum_{n=0}^{m}\left|\left\|h\left[\phi\left(t^{n+1}\right)\right]\right\|^{2}-\frac{\tilde{R}^{n+1}}{E\left(\bar{\phi}^{n+1}\right)}\left\|h\left(\bar{\phi}^{n+1}\right)\right\|^{2}\right| \\
		& + \sum_{n=0}^{m} \left(1-\zeta_0\right)\left|E\left(\phi^{n}\right)-E\left(\phi\left(t^n\right)\right)\right|+ \sum_{n=0}^{m} \left|T_{1}^{n}\right|.
	\end{aligned}
	\end{equation}
Similarly  to the analysis of GSAV approach in \cite{huang2021implicit}, we can bound the right hand terms of \eqref{eq:error-estimate-4} as follows.
	First, thanks to \eqref{eq:error-estimate-5}, we have
	\begin{equation}
	\left|T_{1}^{n}\right| \leq C \delta t \int_{t^{n}}^{t^{n+1}}\left|R_{t t}\right| d s \leq C \delta t \int_{t^{n}}^{t^{n+1}}\left(\left\|\phi_{t}(s)\right\|_{H^{1}}^{2}+\left\|\phi_{t t}(s)\right\|_{H^{1}}\right) \mathrm{d} s.
	\end{equation} 
	Next, by \eqref{eq:R-IMEX-SAV-BDFk-3} and \eqref{eq:stability}, we have $\tilde{R}^{n+1}<C$ and
	\begin{equation}
	\begin{array}{l}
	\left|\left\|h\left[\phi\left(t^{n+1}\right)\right]\right\|^{2}-\frac{\tilde{R}^{n+1}}{E\left(\bar{\phi}^{n+1}\right)}\left\|h\left(\bar{\phi}^{n+1}\right)\right\|^{2}\right| \\ 
	\leq\left\|h\left[\phi\left(t^{n+1}\right)\right]\right\|^{2}\left|1-\frac{\tilde{R}^{n+1}}{E\left(\bar{\phi}^{n+1}\right)}\right|+\frac{\tilde{R}^{n+1}}{E\left(\bar{\phi}^{n+1}\right)}\left|\left\|h\left[\phi\left(t^{n+1}\right)\right]\right\|^{2}-\left\|h\left(\bar{\phi}^{n+1}\right)\right\|^{2}\right| \\ 
	:=P_{1}^{n}+P_{2}^{n}.
	\end{array}
	\end{equation}	
	By Theorem \ref{Th:solution-regularity}, we have $\left\|h\left[\phi\left(t^{n+1}\right)\right]\right\|^{2}<C$, and by $E(u)>\underline{C}>0$, we find
	\begin{equation}\label{eq:error-estimate-12}
	\begin{aligned} P_{1}^{n} 
	& \leq C\left|1-\frac{\tilde{R}^{n+1}}{E\left(\bar{\phi}^{n+1}\right)}\right| \\ 
	& \leq C\left|\frac{R\left(t^{n+1}\right)}{E\left[\phi\left(t^{n+1}\right)\right]}-\frac{\tilde{R}^{n+1}}{E\left[\phi\left(t^{n+1}\right)\right]}\right|+C\left|\frac{\tilde{R}^{n+1}}{E\left[\phi\left(t^{n+1}\right)\right]}-\frac{\tilde{R}^{n+1}}{E\left(\bar{\phi}^{n+1}\right)}\right| \\ 
	& \leq C\left(\left|E\left[\phi\left(t^{n+1}\right)\right]-E\left(\bar{\phi}^{n+1}\right)\right|+\left|\tilde{s}^{n+1}\right|\right). 
	\end{aligned}
	\end{equation}
	and 
	\begin{equation}\label{eq:error-estimate-13}
	\begin{aligned}
	\left|E\left[\phi\left(t^{n+1}\right)\right]-E\left(\bar{\phi}^{n+1}\right)\right| \leq C \bar{C}\left(\left\|\nabla \bar{e}^{n+1}\right\|+\left\|\bar{e}^{n+1}\right\|\right). 
	\end{aligned}
	\end{equation} 
		On the other hand,  
	\begin{equation}
	\begin{aligned} 
	P_{2}^{n} \leq C \bar{C}\left(\left\|\Delta \bar{e}^{n+1}\right\|+\left\|\bar{e}^{n+1}\right\|\right),
	\end{aligned}
	\end{equation}
and
	\begin{equation}
		\left|E\left(\phi^{n}\right)-E\left(\phi\left(t^{n}\right)\right)\right| \leq \left|E\left(\phi^{n}\right)-E\left(\bar{\phi}^{n}\right)\right| + \left|E\left(\bar{\phi}^{n}\right)-E\left(\phi\left(t^{n}\right)\right)\right|.
	\end{equation}
Thanks to \eqref{eq:error-estimate-8}, \eqref{eq:error-estimate-9}, Theorem \ref{Th: stability} and \eqref{eq:error-estimate-11},  we have
	\begin{equation}
	\begin{aligned}
		\left|E\left(\phi^{n}\right)-E\left(\bar{\phi}^{n}\right)\right|
		& \leq \frac{1}{2}\left(\left\|\nabla \phi^{n}\right\|+\left\|\nabla \bar{\phi}^{n}\right\|\right)\left\|\nabla \phi^{n}-\nabla \bar{\phi}^{n}\right\| \\ 
		& +\frac{\lambda}{2}\left(\left\|\phi^{n}\right\|+\left\|\bar{\phi}^{n}\right\|\right)\left\|\phi^{n}-\bar{\phi}^{n}\right\| +\int_{\Omega} F\left(\phi^{n}\right) - F\left(\bar{\phi}^{n}\right) \mathrm{d} \mathbf{x}\\ 
		& \leq C\left(M_{k}+C_1\right)\left( \left\|\nabla \phi^{n}-\nabla \bar{\phi}^{n}\right\| + \left\|\phi^{n}-\bar{\phi}^{n}\right\|\right) \\
		& \leq C\left(M_{k}+C_1\right) \left|1-\eta_{k}^{n}\right|\left\|\bar{\phi}^{n}\right\|_{H^{1}} \\
		& \leq C\bar{C}\left(M_{k}+C_1\right)C_0^{k+1}\delta t^{k+1}, 
	\end{aligned}
	\end{equation}
	and 
	\begin{equation}\label{eq:error-estimate-7}
	\left|E\left(\bar{\phi}^{n}\right)-E\left(\phi\left(t^{n}\right)\right)\right|
	\leq C \bar{C}\left(\left\|\nabla \bar{e}^{n}\right\|+\left\|\bar{e}^{n}\right\|\right).
	\end{equation}
	Now, combing \eqref{eq:error-estimate-6}, \eqref{eq:error-estimate-4}-\eqref{eq:error-estimate-7}, we obtain 
	\begin{equation}
	\begin{aligned}
	\left|\tilde{s}^{m+1}\right| 
	& \leq C \delta t \sum_{n=0}^{m}\left|\tilde{s}^{n+1}\right|+C \bar{C} \delta t \sum_{n=0}^{m}\left\|\bar{e}^{n+1}\right\|_{H^{2}} +  C\bar{C}\left(M_{k}+C_1\right)C_0^{k+1}\delta t^{k}\\
	& + C \delta t \int_{0}^{T}\left(\left\|\phi_{t}(s)\right\|_{H^{1}}^{2}+\left\|\phi_{t t}(s)\right\|_{H^{1}}\right) \mathrm{d} s \\ 
	& \leq C \delta t \sum_{n=0}^{m}\left|\tilde{s}^{n+1}\right|+C \bar{C} \left(\sqrt{C_{2}\left(1+C_{0}^{2 k+2}\right)} +  \left(M_{k}+C_1\right)C_0^{k+1}\right)\delta t^{k} +C \delta t. 
	\end{aligned}
	\end{equation}
	Applying the discrete Gronwall lemma to the above inequality with $\delta t<\frac{1}{2C}$, we derive
	\begin{equation}\label{eq:error-estimate-14}
	\begin{aligned}
	\left|\tilde{s}^{m+1}\right| 
	&\leq C \exp \left((1-C \delta t)^{-1}\right) \delta t\left(\bar{C} \left(\sqrt{C_{2}\left(1+C_{0}^{2 k+2}\right)} + \left(M_{k}+C_1\right)C_0^{k+1} \right) \delta t^{k-1}+1\right) \\ 
	& \leq C_{3} \delta t\left(\bar{C} \left(\sqrt{C_{2}\left(1+C_{0}^{2 k+2}\right)} + \left(M_{k}+C_1\right)C_0^{k+1} \right) \delta t^{k-1}+1\right), 
	\end{aligned}
	\end{equation}
	where $C_3$ is independent of $C_0$ and $\delta t$, can be defined as 
	\begin{equation}
		C_{3}:=C \exp(2),
	\end{equation}
	then $\delta t<\frac{1}{2C}$ can be guaranteed by 
	\begin{equation}
		\delta t<\frac{1}{C_3}.
	\end{equation}
	Hence, using \eqref{eq:error-estimate-12}, \eqref{eq:error-estimate-13}, \eqref{eq:error-estimate-14} and \eqref{eq:error-estimate-6}, we have
	\begin{equation*}
	\begin{aligned}
	\left|1-\xi^{m+1}\right| 
	& \leq C\left(\left|E\left[\phi\left(t^{n+1}\right)\right]-E\left(\bar{\phi}^{n+1}\right)\right|+\left|\tilde{s}^{n+1}\right|\right) \\ 
	& \leq C\left(\bar{C}\left\|\bar{e}^{m+1}\right\|_{H^{1}}+\left|\tilde{s}^{m+1}\right|\right) \\ 
	& \leq C \delta t\left(\bar{C} \sqrt{C_{2}\left(1+C_{0}^{2 k+2}\right)} \delta t^{k-1}+C_{3} \left(\bar{C} \left(\sqrt{C_{2}\left(1+C_{0}^{2 k+2}\right)} + \left(M_{k}+C_1\right)C_0^{k+1} \right) \delta t^{k-1}+1\right)\right) \\ 
	& \leq C_{4} \delta t\left(\sqrt{1+C_{0}^{2 k+2}} \delta t^{k-1}+1\right), 
	\end{aligned}
	\end{equation*}
	where $C_4$ is independent of $C_0$ and $\delta t$. 
	
	For the cases $k \geq 2$, we can define $C_0$ exactly the same as \cite{huang2021implicit} with the condition $\delta t \leq \frac{1}{1+C_0^{k+1}}$ to obtain $\left|1-\xi^{m+1}\right|<C_0 \delta t$. 
		For the case $k=1$, we set $\eta_{1}^{n+1}=1-\left(1-\xi^{n+1}\right)^3$, we can repeat the same process as the cases $k \geq 2$ and arrive at a similar result.
	
	Finally, we can show that
	\begin{equation*}
	\begin{aligned}
	\left\|e^{m+1}\right\|_{H^{2}}^{2} 
	& \leq 2\left\|\bar{e}^{m+1}\right\|_{H^{2}}^{2}+2\left\|\phi^{m+1}-\bar{\phi}^{m+1}\right\|_{H^{2}}^{2} \\ 
	& \leq 2\left\|\bar{e}^{m+1}\right\|_{H^{2}}^{2} + 2 \left|\eta_{k}^{m+1}-1\right|^{2}\left\|\bar{\phi}^{m+1}\right\|^{2}_{H^{2}} \\
	& \leq 2 C_{2}\left(1+C_{0}^{2(k+1)}\right) \delta t^{2 k}+2 \bar{C}^{2} C_{0}^{2(k+1)} \delta t^{2(k+1)}  \leq C \delta t^{2k},
	\end{aligned}
	\end{equation*}
	provided that  $\delta t<\min\left\lbrace \frac{1}{1+2C_0^{k+2}}, \frac{1-\tau_k}{3k} \right\rbrace$. 
	The proof is complete.	
\end{proof}

Similar results can also be established for the  Cahn-Hilliard type equation
\begin{equation}\label{eq:Cahn-Hilliard-eq}
\frac{\partial \phi}{\partial t}=\Delta (-\Delta  \phi+\lambda  \phi+ F^{\prime}(\phi)) \quad(\boldsymbol{x}, t) \in \Omega \times(0, T],
\end{equation}
which is  \eqref{eq:model-problem} with $\mathcal{L}=-\Delta+\lambda I$ and with $\mathcal{G}=-\Delta$, with initial condition $\phi(x, 0)=\phi^0(x)$ 
and periodic or Neumann boundary condition. It can also be written in the form of \eqref{eq: new1} with $\mathcal{A}=-\Delta(-\Delta+\lambda I)$ and $g(\phi)=-\Delta(F^{\prime}(\phi))$.
The corresponding \eqref{eq: new2} is with $E_{tot}(\phi)=\int_\Omega \frac12|\nabla\phi|^2+\lambda |\phi|^2+F(\phi) dx$ and $\mathcal{K}(\phi)=\|\nabla (-\Delta  \phi+\lambda  \phi+ F^{\prime}(\phi))\|^2$.

\begin{thm}
	Given initial condition $\bar{\phi}^{i}=\phi^{i}=\phi(t^i), r^{i}=E(\phi^{i}), i=0, 1, ..., k-1$. 
	Let $\bar{\phi}^{n+1}$ and $\phi^{n+1}$ be computed with the R-GSAV/BDF$k$ ($1 \leq k \leq 5$) scheme \eqref{eq:R-IMEX-SAV-BDFk-1}-\eqref{eq:R-IMEX-SAV-BDFk-6} for \eqref{eq:Cahn-Hilliard-eq} with 
	\begin{equation}
	\eta_{1}^{n+1}=1-\left(1-\xi^{n+1}\right)^{3}, \quad \eta_{k}^{n+1}=1-\left(1-\xi^{n+1}\right)^{k+1}(k=2,3,4,5).
	\end{equation}
	We assume \eqref{eq:assumption-1} and \eqref{eq:assumption-2} holds and
	\begin{equation}
	\phi \in C\left([0, T] ; H^{3}\right), \quad \frac{\partial^{j} \phi}{\partial t^{j}} \in L^{2}\left(0, T ; H^{2}\right) 1 \leq j \leq k, \quad \frac{\partial^{k+1} \phi}{\partial t^{k+1}} \in L^{2}\left(0, T ; H^{1}\right).
	\end{equation}	
	Then for  $\delta t < \min \lbrace \frac{1}{1+4C_0^{k+2}}, \frac{1-\tau_k}{3k} \rbrace$, we have 
	$$
\left\|\bar{\phi}^{n}-\phi(\cdot, t^{n})\right\|_{H^{2}}, \quad \left\|\phi^{n}-\phi(\cdot, t^{n})\right\|_{H^{2}} \leq C \delta t^{k}, \quad\forall  n \leq T/\delta t,
$$
where  $\tau_k$ is given in \eqref{tau_k}, and the constants  $C_0,\, C$ are  independent of $\delta t$.
\end{thm}
The above results can be established by combining the proofs of the above theorem and  Theorem in \cite{huang2021implicit}, we leave the detail to the interested readers.

\section{Extension to the multiple SAV approach}
\label{sec:msav}
In some cases, the nonlinear part of the free energy may contain disparate terms such that schemes with a single SAV may require excessively small time steps to obtain correct simulations \cite{cheng2018multiple}. 
It is shown in \cite{cheng2018multiple} that  the multiple SAV (MSAV) approach can overcome this difficulty. 
 
In this section, we demonstrate how to construct relaxed MGSAV schemes for gradient flow. 
Without loss of generality, we consider the following gradient flow with  two disparate  nonlinear terms (extension to more than two disparate  nonlinear terms is straightforward):
\begin{equation} \label{eq:model-problem-MGSAV}
	\left\{\begin{aligned}
		& \frac{\partial \phi}{\partial t}=-\mathcal{G} \mu, \\
		& \mu=\mathcal{L} \phi+F_{1}^{\prime}(\phi) + F_{2}^{\prime}(\phi),
	\end{aligned}\right.
\end{equation}
where $\mathcal{L}$  is a linear self-adjoint elliptic operator, $F_{1}(\phi), F_{2}(\phi)$ are nonlinear potential function,  $\mathcal{G}$ is a positive linear operator.  
The system \eqref{eq:model-problem-MGSAV} satisfies an energy dissipation law as follows
\begin{align}
	\frac{d E_{tot}(\phi)}{dt} =-\left(\mathcal{G}\mu, \mu\right),
\end{align} 
where
\begin{equation}
	E_{tot}(\phi)=\frac{1}{2}(\mathcal{L} \phi, \phi) + \int_{\Omega} F_{1}(\phi) \mathrm{d} \mathbf{x}+\int_{\Omega} F_{2}(\phi) \mathrm{d} \mathbf{x}
\end{equation}
is a free energy with lower bound $-C_0$.  
Setting $E(\phi)=E_{tot}(\phi)+C_{0}=E_{1}(\phi)+E_{2}(\phi)$ with $E_{1}(\phi)=\frac{1}{2}(\mathcal{L} \phi, \phi) + \int_{\Omega} F_{1}(\phi) \mathrm{d} \mathbf{x}+C_{1}>0,\, E_{2}(\phi)=\int_{\Omega} F_{2}(\phi) \mathrm{d} \mathbf{x}+C_{2}>0$ and introducing two SAVs $R_{1}(t)=E_{1}(\phi),\, R_{2}(t)=E_{2}(\phi)$, we can rewrite the equation \eqref{eq:model-problem-MGSAV} as
\begin{equation}\label{eq:model-problem-reformulation-MGSAV}
	\left\{\begin{aligned}
		& \frac{\partial \phi}{\partial t}=-\mathcal{G} \mu, \\
		& \mu=\frac{\delta E}{\delta \phi}=\mathcal{L} \phi+F_{1}^{\prime}(\phi)+F_{2}^{\prime}(\phi), \\
		& \frac{\mathrm{d}R_{1}(t)}{\mathrm{d}t}=-\frac{R_{1}(t)+R_{2}(t)}{E_{1}(\phi)+E_{2}(\phi)}\left(\mathcal{G}\frac{\delta E_{1}}{\delta \phi}, \mu\right),\\
		& \frac{\mathrm{d}R_{2}(t)}{\mathrm{d}t}=-\frac{R_{1}(t)+R_{2}(t)}{E_{1}(\phi)+E_{2}(\phi)}\left(\mathcal{G}\frac{\delta E_{2}}{\delta \phi}, \mu\right).
	\end{aligned}\right.
\end{equation}


Note that the above MSAV formulation is different from that used in \cite{cheng2018multiple}.
Then we construct the relaxed MGSAV BDF$k$ schemes as follows:

Given $\phi^{n-k}, ..., \phi^{n}, R_{1}^{n-k}, ..., R_{1}^{n},  R_{2}^{n-k}, ..., R_{2}^{n}$, we compute $\phi^{n+1}, R_{1}^{n+1}, R_{2}^{n+1}$ via the following two steps:

\textbf{Step 1:}  
\begin{itemize}
	\item solve $\phi_{1}^{n+1}$ and $\phi_{2}^{n+1}$ from
	\begin{eqnarray}
		&& \frac{\alpha_{k}\phi_{1}^{n+1}-\frac{1}{2}A_{k}(\phi^{n})}{\delta t}=-\mathcal{G} \left(\mathcal{L} \phi_{1}^{n+1}+F_{1}^{\prime}(B_{k}(\phi^{n}))\right), \label{eq:R-MGSAV-BDFk-1}\\
		&& \frac{\alpha_{k}\phi_{2}^{n+1}-\frac{1}{2}A_{k}(\phi^{n})}{\delta t}=-\mathcal{G} \left(\mathcal{L} \phi_{2}^{n+1}+F_{2}^{\prime}(B_{k}(\phi^{n}))\right); \label{eq:R-MGSAV-BDFk-2}
	\end{eqnarray}
	and set
	\begin{eqnarray}
		\label{eq:R-MGSAV-BDFk-8}
		&&	\bar \phi^{n+1}=\phi_{1}^{n+1}+\phi_{2}^{n+1},\\
		&&	\bar \mu^{n+1} =\mathcal{L} \bar \phi^{n+1}+F_{1}^{\prime}(\bar \phi^{n+1})+F_{2}^{\prime}(\bar \phi^{n+1});
	\end{eqnarray}
	\item  solve $\tilde{R}_{1}^{n+1}$ and $\tilde{R}_{2}^{n+1}$ from
	\begin{eqnarray}
		\label{eq:R-MGSAV-BDFk-3}
		&& \frac{\tilde{R}_{1}^{n+1}-R_{1}^{n}}{\delta t}=-\frac{\tilde{R}_{1}^{n+1}+\tilde{R}_{2}^{n+1}}{E_{1}(\bar\phi^{n+1})+E_{2}(\bar\phi^{n+1})}\left(\mathcal{G}\frac{\delta E_{1}}{\delta \phi}\left(\bar \phi^{n+1}\right), \bar \mu^{n+1}\right), \\
		\label{eq:R-MGSAV-BDFk-4}
		&& \frac{\tilde{R}_{2}^{n+1}-R_{2}^{n}}{\delta t}=-\frac{\tilde{R}_{1}^{n+1}+\tilde{R}_{2}^{n+1}}{E_{1}(\bar\phi^{n+1})+E_{2}(\bar\phi^{n+1})}\left(\mathcal{G}\frac{\delta E_{2}}{\delta \phi}\left(\bar \phi^{n+1}\right), \bar \mu^{n+1}\right);
	\end{eqnarray}
	
	\item set
	\begin{eqnarray}
		\label{eq:R-MGSAV-BDFk-5}
		&& \xi_{k, 1}^{n+1}=\frac{\tilde{R}_{1}^{n+1}}{E_{1}(\bar\phi^{n+1})}, \quad \xi_{k, 2}^{n+1}=\frac{\tilde{R}_{2}^{n+1}}{E_{2}(\bar\phi^{n+1})},\\
		\label{eq:R-MGSAV-BDFk-6}
		&& \eta_{k, 1}^{n+1}=1-(1-\xi_{k, 1}^{n+1})^{k+1}, \quad \eta_{k, 2}^{n+1}=1-(1-\xi_{k, 2}^{n+1})^{k+1},\\
		\label{eq:R-MGSAV-BDFk-7}
		&& \phi^{n+1}=\eta_{k, 1}^{n+1}\phi_{1}^{n+1}+\eta_{k, 2}^{n+1}\phi_{2}^{n+1},\\ 
		&& \mu^{n+1} =\mathcal{L} \phi^{n+1}+F_{1}^{\prime}(\phi^{n+1})+F_{2}^{\prime}(\phi^{n+1});
	\end{eqnarray}
\end{itemize}


\textbf{Step 2:} Update $R_{1}^{n+1}, R_{2}^{n+1}$ by
\begin{equation} \label{eq:R-MGSAV-BDFk-9}
	R_{1}^{n+1} = \zeta_{0} \tilde{R}_{1}^{n+1} + (1-\zeta_{0})E_{1}(\phi^{n+1}), \quad R_{2}^{n+1} = \zeta_{0} \tilde{R}_{2}^{n+1} + (1-\zeta_{0})E_{2}(\phi^{n+1}), \quad \zeta_{0} \in \mathcal{V}.
\end{equation}
Here $\mathcal{V}$ is a set defined by 
\begin{eqnarray}\label{eq:R-MGSAV-set-condition}
	&& \qquad \mathcal{V}=\left\lbrace \zeta\in [0,1] \; s.t. \; \frac{\left(R_{1}^{n+1}+R_{2}^{n+1}\right)-\left(\tilde{R}_{1}^{n+1}+\tilde{R}_{2}^{n+1}\right)}{\delta t} \leq \right.\\
	&& \qquad \qquad \left.-\gamma^{n+1}\left(\mathcal{G}\mu^{n+1}, \mu^{n+1}\right) + \frac{\tilde{R}_{1}^{n+1}+\tilde{R}_{2}^{n+1}}{E_{1}(\bar{\phi}^{n+1})+E_{2}(\bar{\phi}^{n+1})}\left(\mathcal{G}\bar\mu^{n+1}, \bar\mu^{n+1}\right), \quad \gamma^{n+1} \geq 0 \right\rbrace, \nonumber
\end{eqnarray}
with  $\gamma^{n+1} \geq 0$ to be determined so that $\mathcal{V}$ is not empty. 

Setting $\tilde R^{n+1}=\tilde R_{1}^{n+1}+\tilde R_{2}^{n+1}, R^{n+1}=R_{1}^{n+1}+R_{2}^{n+1}, E(\bar \phi^{n+1})=E_{1}(\bar \phi^{n+1})+E_{2}(\bar \phi^{n+1})$ and plugging \eqref{eq:R-MGSAV-BDFk-9} into the inequality of \eqref{eq:R-MGSAV-set-condition}, we find that if we choose  $\zeta_0$ and $\gamma^{n+1}$ such that the following condition is satisfied:
\begin{equation}\label{eq:R-MGSAV-zeta}
	(\tilde{R}^{n+1}-E(\phi^{n+1}))\zeta_0 \leq \tilde{R}^{n+1}-E(\phi^{n+1})-\delta t \gamma^{n+1} \left(\mathcal{G}\mu^{n+1}, \mu^{n+1}\right)+ \delta t \frac{\tilde{R}^{n+1}}{E(\bar{\phi}^{n+1})} \left(\mathcal{G}\bar \mu^{n+1}, \bar \mu^{n+1}\right),
\end{equation} 
then $\zeta_0\in \mathcal{V}$. Following the same arguments as in the proof of  Theorem \ref{Th: stability}, we can prove the following results for the schemes \eqref{eq:R-MGSAV-BDFk-1}-\eqref{eq:R-MGSAV-BDFk-9}. 

\begin{thm}\label{Th:MSAV-stability}
	We choose $\zeta_{0}$ in \eqref{eq:R-MGSAV-BDFk-9} and $\gamma^{n+1}$ in \eqref{eq:R-MGSAV-set-condition} as follows:
	\begin{enumerate}
		\item If $\tilde{R}^{n+1} = E(\phi^{n+1})$, we set $\zeta_{0}=0$ and $\gamma^{n+1}=\frac{\tilde{R}^{n+1} \left(\mathcal{G}\bar \mu^{n+1}, \bar \mu^{n+1}\right)}{E(\bar{\phi}^{n+1}) \left(\mathcal{G}\mu^{n+1}, \mu^{n+1}\right)}$;
		\item If $\tilde{R}^{n+1} > E(\phi^{n+1})$, we set $\zeta_{0}=0$ and
		\begin{equation}\label{eq:R-MGSAV-gamma}
			\gamma^{n+1} = \frac{\tilde{R}^{n+1}-E(\phi^{n+1})}{\delta t \left(\mathcal{G} \mu^{n+1},  \mu^{n+1}\right)} + \frac{\tilde{R}^{n+1} \left(\mathcal{G}\bar \mu^{n+1}, \bar \mu^{n+1}\right)}{E(\bar{\phi}^{n+1}) \left(\mathcal{G}\mu^{n+1}, \mu^{n+1}\right)};
		\end{equation}
		\item If $\tilde{R}^{n+1} < E(\phi^{n+1})$ and $\tilde{R}^{n+1}-E(\phi^{n+1})+ \delta t \frac{\tilde{R}^{n+1}}{E(\bar{\phi}^{n+1})} \left(\mathcal{G}\bar \mu^{n+1}, \bar \mu^{n+1}\right) \geq 0$, we set $\zeta_{0}=0$ and 	
		$\gamma^{n+1}$ the same as \eqref{eq:R-MGSAV-gamma}.  
		\item  If $\tilde{R}^{n+1} < E(\phi^{n+1})$ and $\tilde{R}^{n+1}-E(\phi^{n+1})+ \delta t \frac{\tilde{R}^{n+1}}{E(\bar{\phi}^{n+1})} \left(\mathcal{G}\bar \mu^{n+1}, \bar \mu^{n+1}\right)< 0$, we set $\zeta_{0}=1-\frac{\delta t \tilde{R}^{n+1} \left(\mathcal{G}\bar \mu^{n+1}, \bar \mu^{n+1}\right)}{E(\bar{\phi}^{n+1})\left(E(\phi^{n+1})-\tilde{R}^{n+1}\right)}$ and 
		$\gamma^{n+1}=0$.
	\end{enumerate}
	Then, \eqref{eq:R-MGSAV-zeta} is satisfied in all cases above and  $\zeta_0\in \mathcal{V}$. Furthermore,  	given $R^{n} \geq 0$, 
	we have $R^{n+1}\geq 0$, and the scheme \eqref{eq:R-MGSAV-BDFk-1}-\eqref{eq:R-MGSAV-BDFk-9} with the above choice of  $\zeta_0$ and  $\gamma^{n+1}$  is unconditionally energy stable in the sense that
	\begin{equation}\label{eq:R-MGSAV-stability}
		R^{n+1} - R^{n} \leq -\delta t \gamma^{n+1} \left(\mathcal{G}\mu^{n+1}, \mu^{n+1}\right)\leq 0.
	\end{equation}
	Furthermore, we have 
	\begin{equation}
		R^{n+1}\le E(\phi^{n+1})\quad\forall n\ge 0.
	\end{equation}
	
\end{thm}
\section{Numerical results and discussions} \label{sec:Numerical-examples}
We  present in this section  some numerical results to validate the efficiency and accuracy of the R-GSAV approach, and provide detailed comparisons between the original SAV, R-SAV, GSAV and R-GSAV approaches. 

Unless specified otherwise, we consider examples with periodic boundary condition and use the Fourier spectral method for spatial discretization.  
The default value of the parameter $\gamma$ is set to $0.95$ for the R-SAV approaches.

\textbf{Example 1.} The Allen-Cahn equation 
\begin{equation}\label{eq:Allen-Cahn}
\frac{\partial \phi}{\partial t}=\alpha \Delta \phi-\left(1-\phi^{2}\right) \phi,
\end{equation}
with periodic boundary conditions.

{\em Case A. } We add an external force  $f$ to  \eqref{eq:Allen-Cahn} so that  its exact solution is
\begin{equation} \label{eq:AC-CH-exact-solution-example}
\phi(x, y, t)=\exp (\sin (\pi x) \sin (\pi y)) \sin (t).
\end{equation}
We set  $\Omega=(0, 2)\times(0, 2)$,  $\alpha=0.01^2$, and use $64^2$ Fourier modes for space discretization so that the spatial discretization error is negligible when compared with the time discretization error.

In Figure \ref{Fig:AC-order-test-three-method}, we plot the convergence rate of the $H^2$ error at $T = 1$ by using various first- and second-order schemes. 
We observe that (i) the expected convergence rates are obtained for all cases; (ii)  
the errors of  R-SAV (resp. R-GSAV) schemes are  significantly smaller than that of SAV (resp. GSAV)  schemes; 
(iii) the R-SAV approach is the most accurate but it requires solving two linear systems. 
In the left of Fig.\,\ref{Fig:AC-exact-solution-zeta}, we plot  the convergence rate of the $H^2$ error at $T = 1$ by using GSAV/BDFk and R-GSAV/BDFk ($k= 3, 4, 5$) schemes, and observe that all schemes achieve their desired order of accuracy, but the improvements by R-SAV and R-GSAV over SAV and GSAV  for higher-order schemes are not as significant as for lower-order schemes.
In the right of  Fig.\,\ref{Fig:AC-exact-solution-zeta}, we  presents evolution of relaxation parameter $\zeta_0$ using R-GSAV/BDF2 scheme with $\delta t=1e-3$, and observe that, except at an initial time interval, $\zeta_0$ takes the value zero. 



\begin{figure}[htbp]
	\centering
	\includegraphics[width=5.3cm]{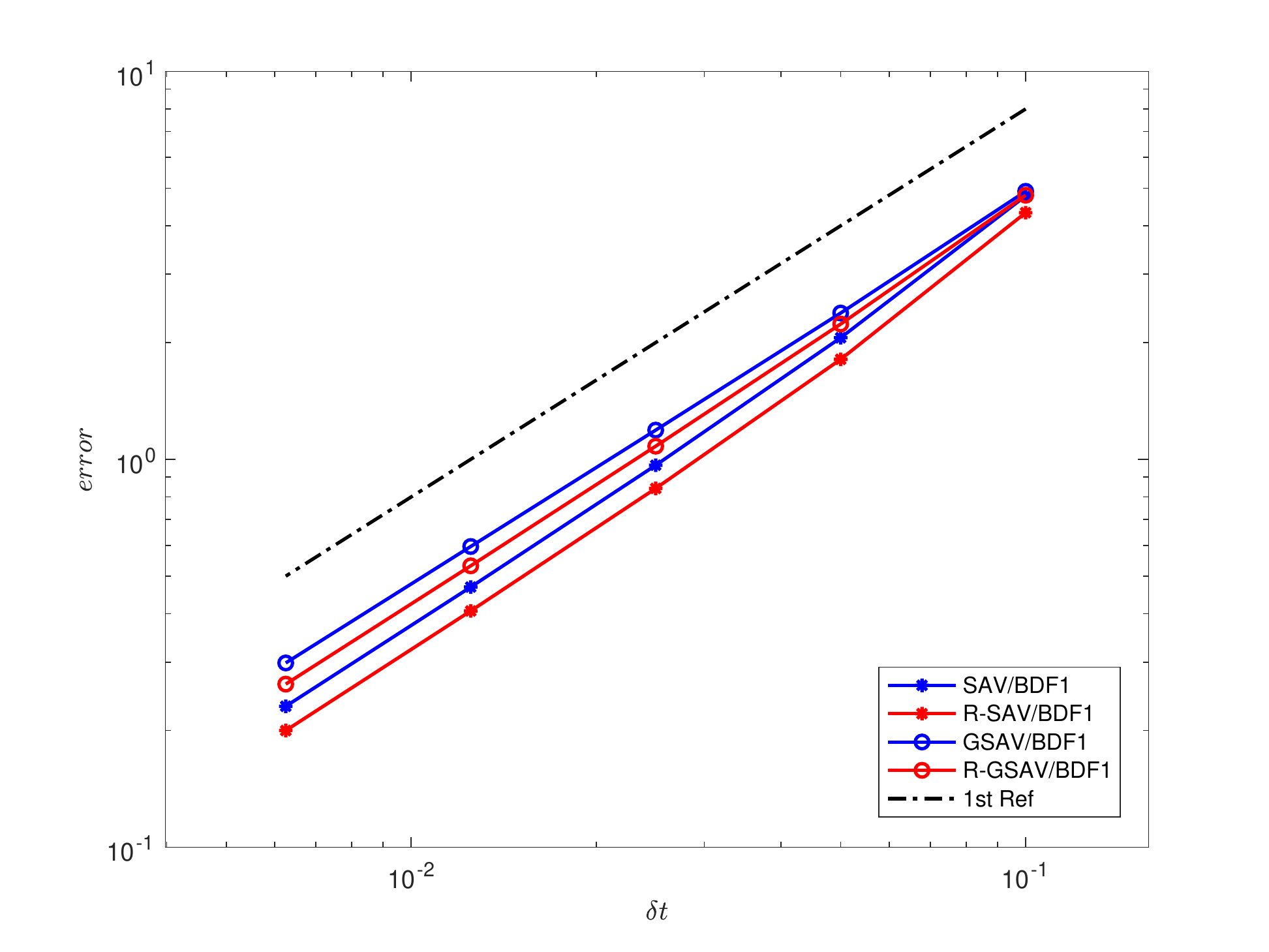}\hspace{-6mm}
	\includegraphics[width=5.3cm]{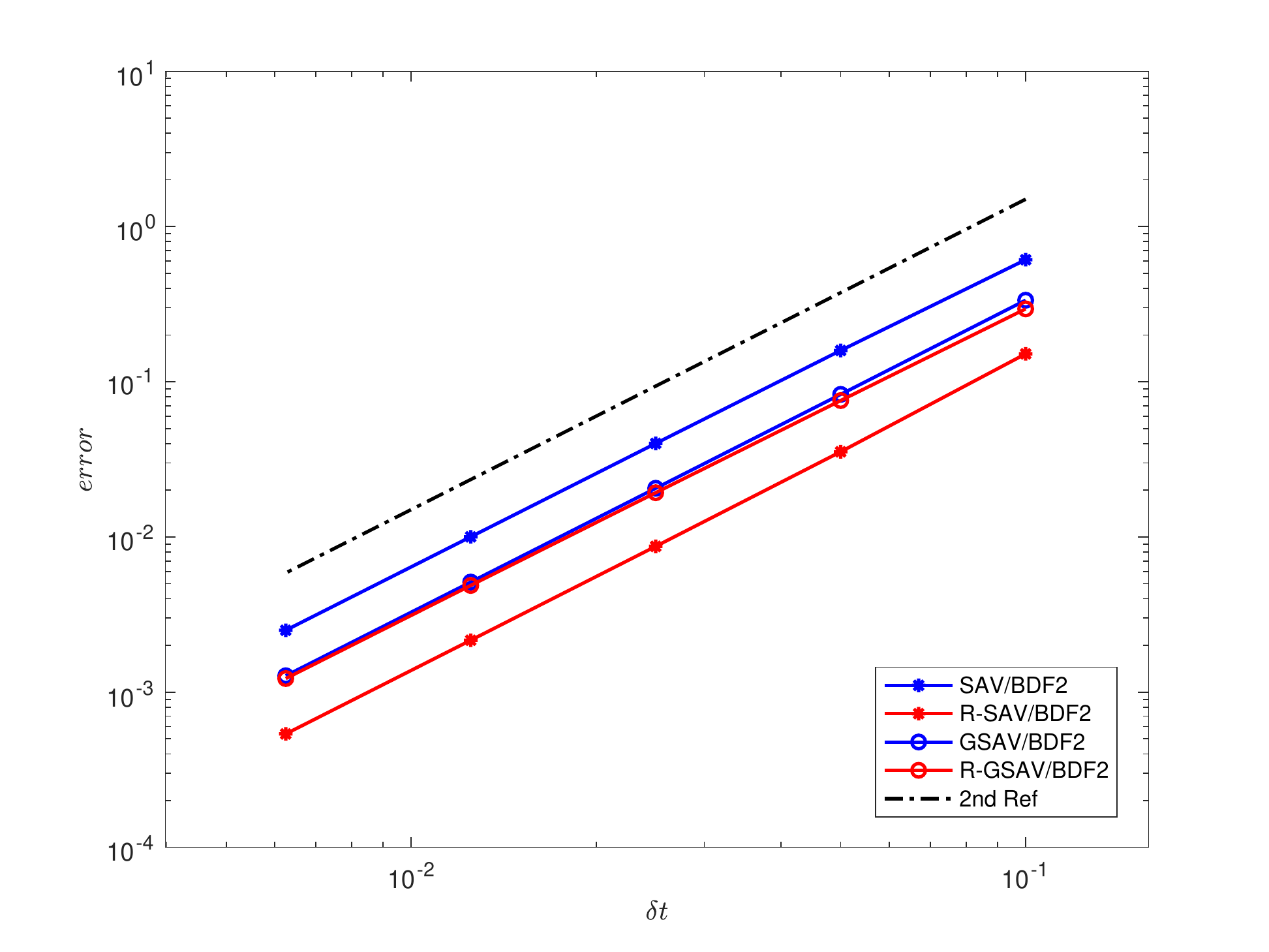}\hspace{-6mm}
	\hspace{-1cm}
	\caption{Example 1A. Convergence rates for Allen-Cahn equation using various schemes. Left: first-order; Right: Second-order.}
	\label{Fig:AC-order-test-three-method}
\end{figure} 

\begin{figure}[htbp]
	\centering
		\includegraphics[width=5.3cm]{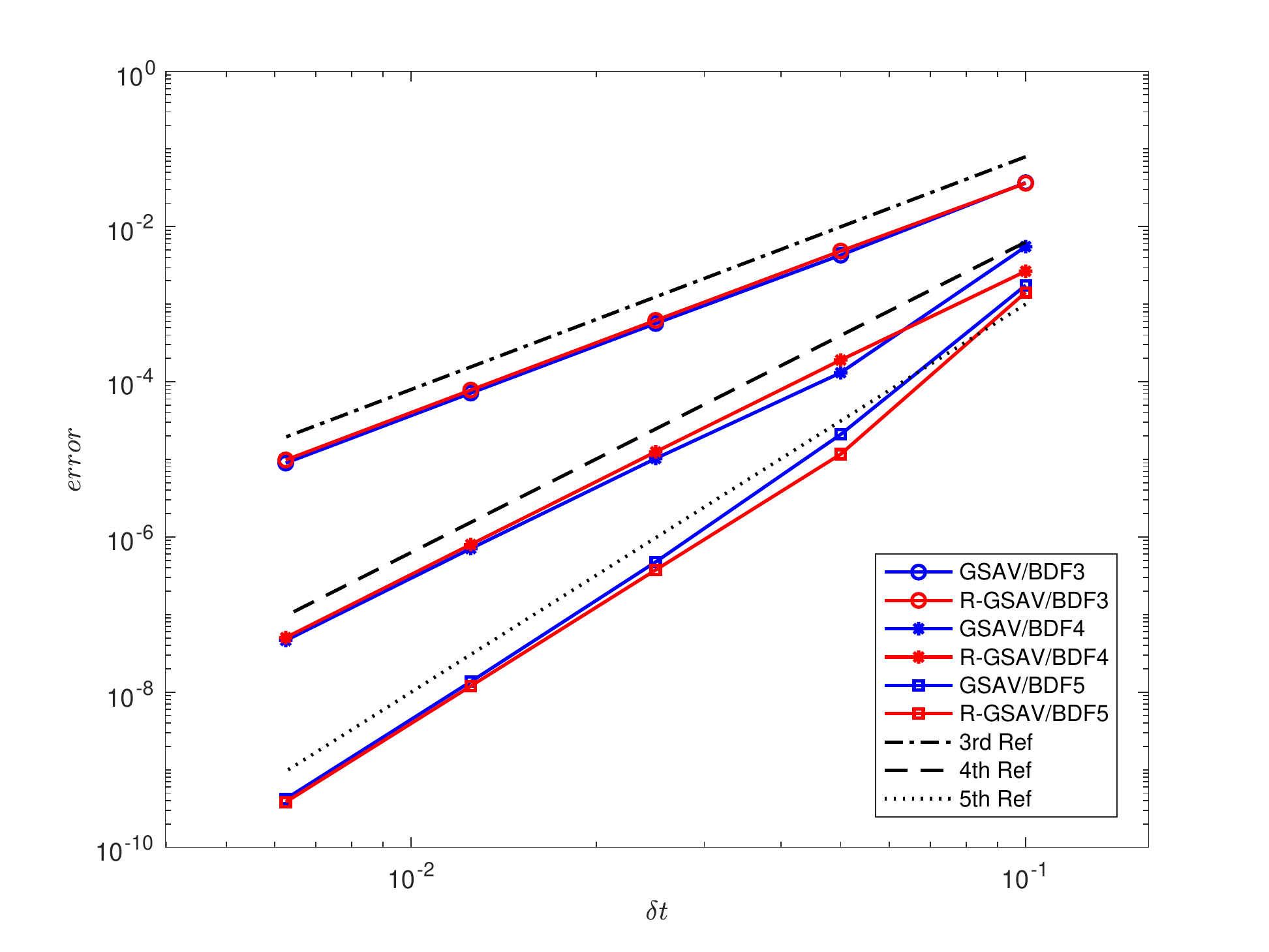}\hspace{-6mm}
	\includegraphics[width=5.3cm]{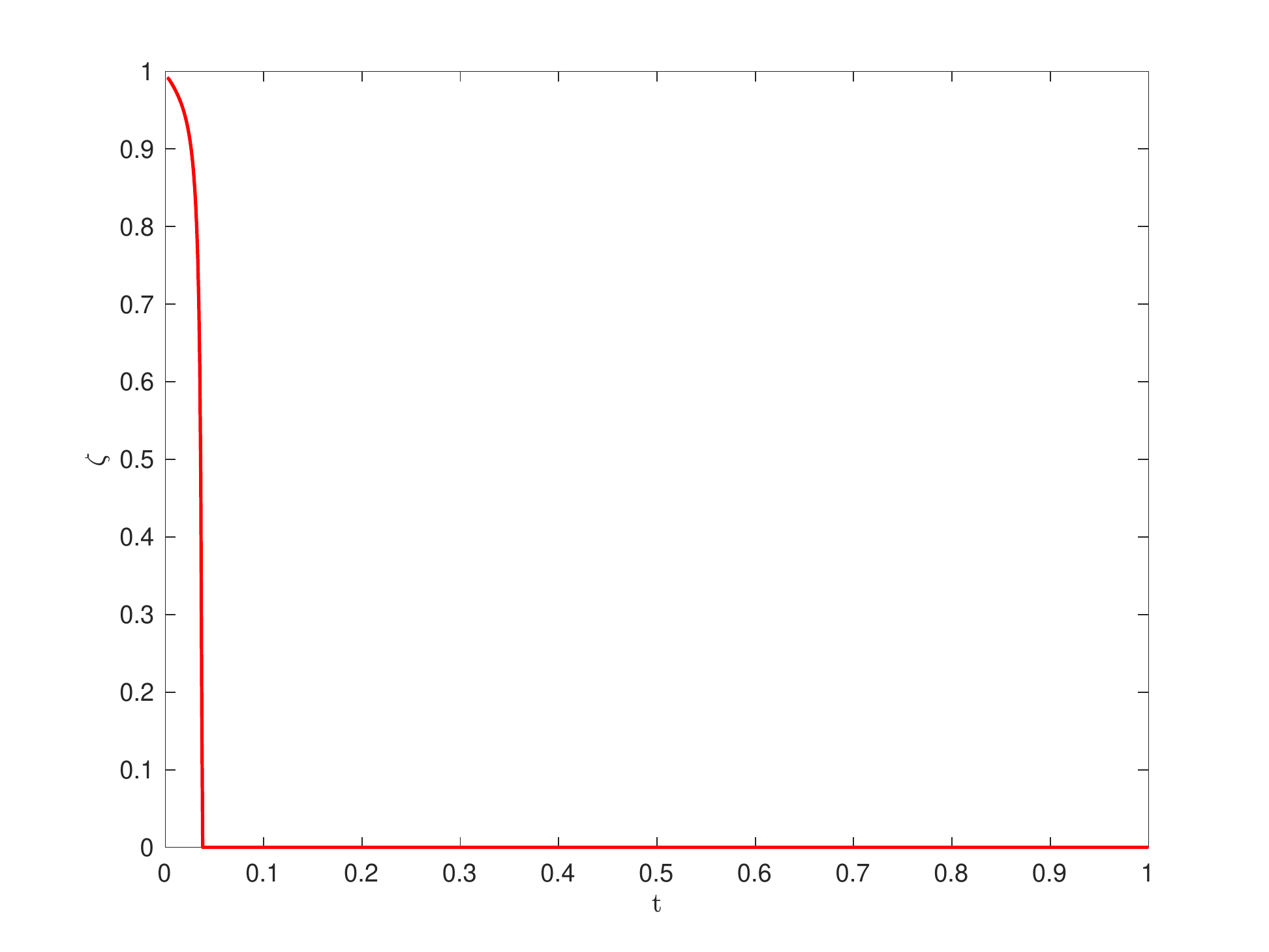}\hspace{-6mm}
	\hspace{-1cm}
	\caption{Example 1A. Left: GSAV/BDFk and R-GSAV/BDFk ($k=3, 4, 5$) schemes; Right: evolution of relaxation $\zeta_{0}$ using R-GSAV/BDF2 scheme with $\delta t=1e-3$.}
	\label{Fig:AC-exact-solution-zeta}
\end{figure}


{\em Case B. }  We set $\Omega=[0, L_x]\times[0, L_{y}]$ with $L_x=L_y=1$, and choose the initial condition as            \begin{equation}\label{eq:AC-CH-initial-condition-star-shape}
\begin{aligned}
	& \phi(x, y)=\tanh \frac{1.5+1.2 \cos (6 \theta)-2 \pi r}{\sqrt{2\alpha}}, \\ 
	& \theta=\arctan \frac{y-0.5 L_{y}}{x-0.5 L_{x}}, \quad r=\sqrt{\left(x-\frac{L_{x}}{2}\right)^{2}+\left(y-\frac{L_{y}}{2}\right)^{2}},
\end{aligned}
\end{equation}
where $(\theta,r)$ are the polar coordinates of $(x, y)$.
 The other parameters are $\alpha=0.01^2, m_0=0.1, \beta=1$ and $128^2$ Fourier modes. 
 We   use the results of the semi-implicit/BDF2 scheme with $ \delta t = 1e-5$ as the reference solution. 
The $L^{2}$-norm error of four schemes at $T=200$ with different time steps are shown in Table \ref{table:comparison-dt-schemes}. 
We observe that R-GSAV (resp. R-SAV) schemes can significantly reduce the error of the solution compared with GSAV (resp. SAV) schemes, and the effect of R-GSAV scheme on improving accuracy is more obvious. 
In Fig.\,\ref{Fig:AC-star-shape-energy}, we present  a comparison of energy (left) and energy error (middle) of  schemes. 
Fig.\,\ref{Fig:AC-star-shape-energy} (right) shows that the evolution of energy error at various time step.

\linespread{1.2}
\begin{table}[htbp] 
	\centering
	\caption{Example 1B. A comparison of  $L^2$-error by SAV/BDF2, R-SAV/BDF2, GSAV/BDF2 and R-GSAV/BDF2 schemes for Allen-Cahn equation at $T=200$ with different time step.}
	 \label{table:comparison-dt-schemes}
	\begin{tabular}{llllllll}
		\hline
		 & SAV & R-SAV & GSAV & R-GSAV \\
		\hline
	1e-1  & 4.30E-04 &  2.72E-04  & 1.27E-03   & 2.65E-04  \\
	1e-2  & 4.53E-05 &  2.93E-06  & 1.47E-04   & 2.90E-06  \\
	1e-3  & 6.26E-07 &  2.96E-08  & 2.38E-06   & 2.94E-08  \\
	\hline
	\end{tabular}
\end{table}

\begin{figure}[htbp]
	\centering
	\includegraphics[width=5.2cm]{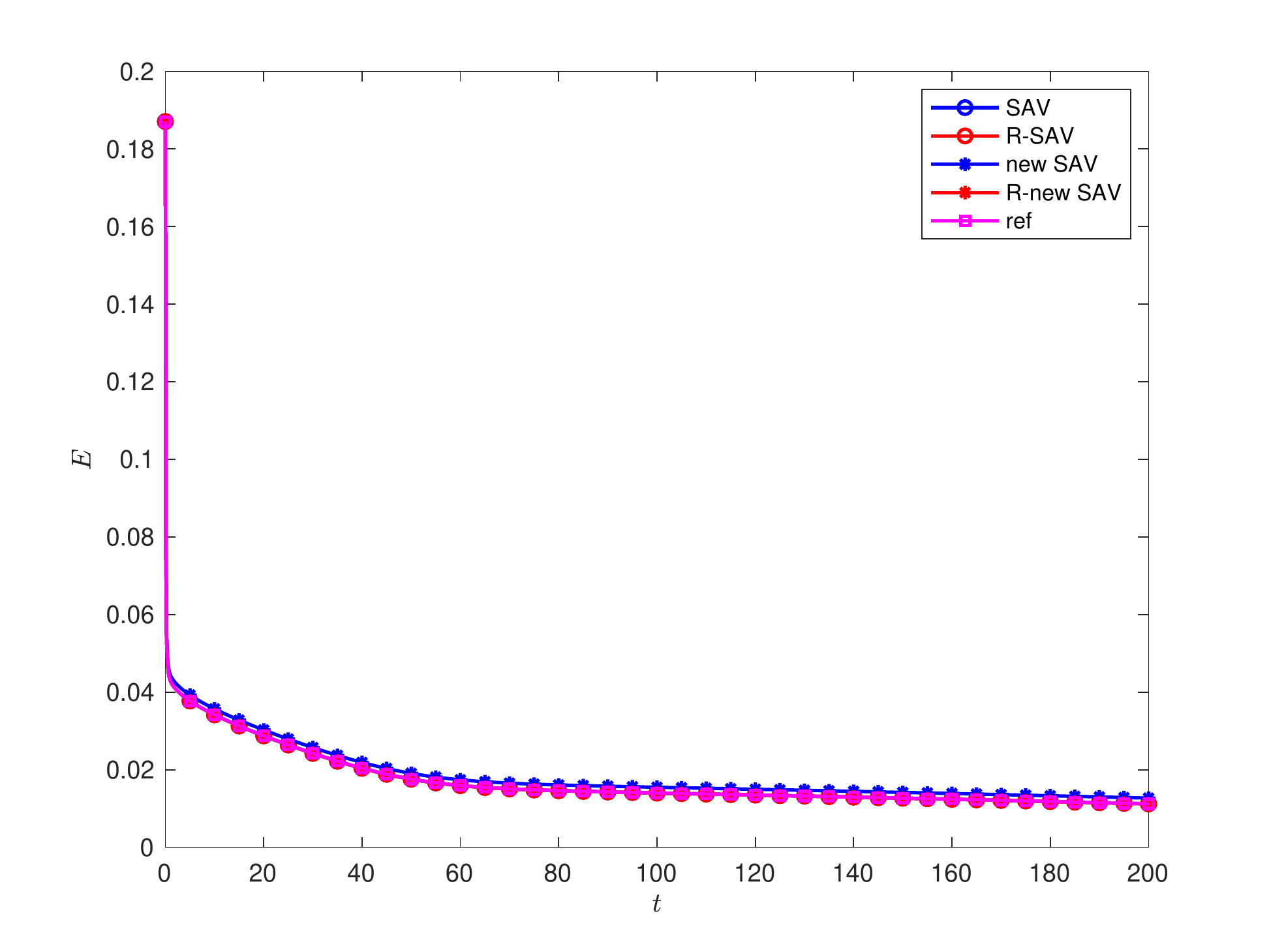} \hspace{-6mm}
	\includegraphics[width=5.2cm]{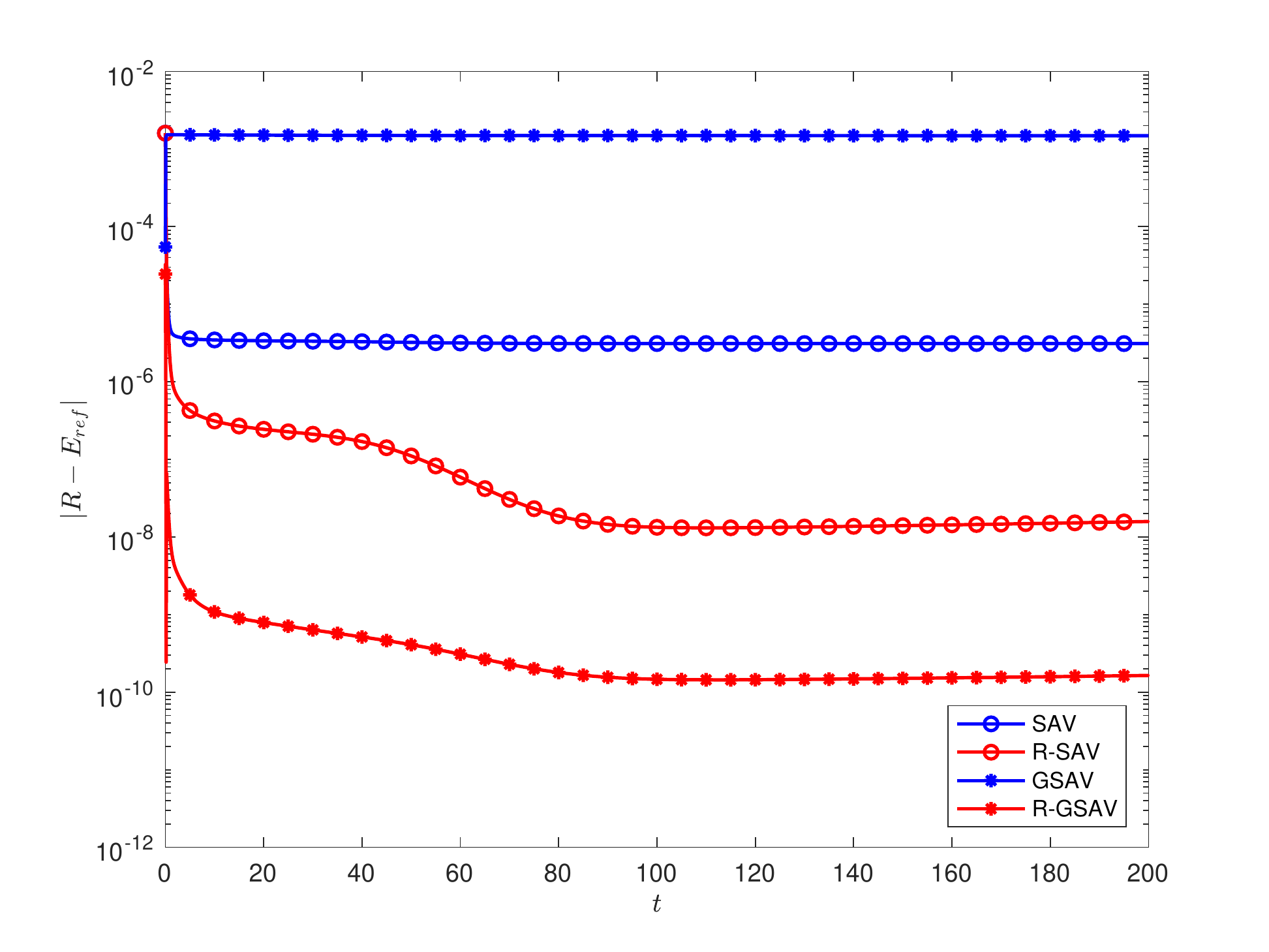} \hspace{-6mm}
	\includegraphics[width=5.2cm]{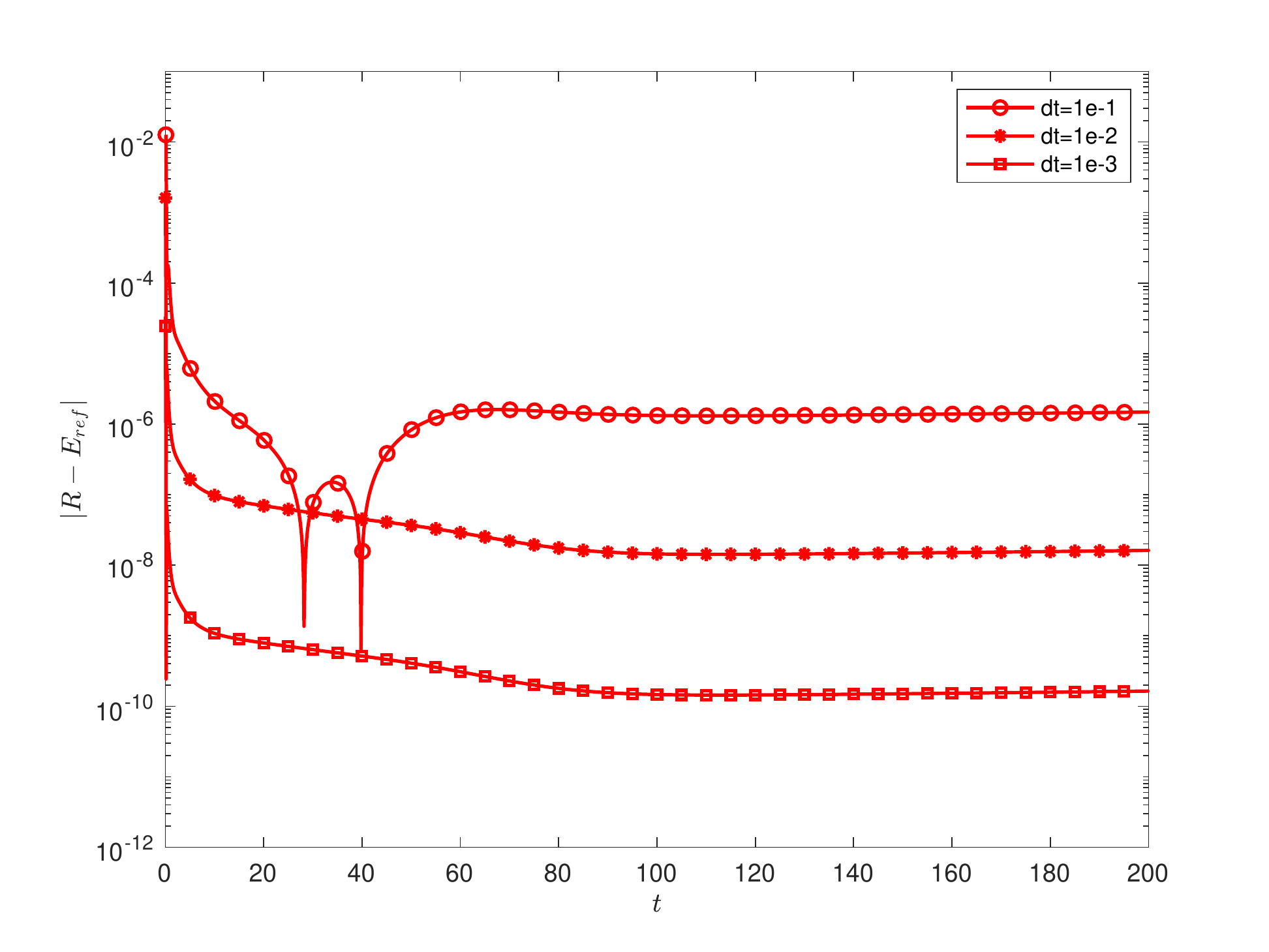}
	
	\caption{Example 1B. Allen-Cahn equation: a comparison of energy (left) and energy error (middle) of SAV/BDF2, R-SAV/BDF2, GSAV/BDF2 and R-GSAV/BDF2 scheme; and a comparison of energy error of R-GSAV/BDF2 scheme at different time steps (right).}
		\label{Fig:AC-star-shape-energy}
\end{figure} 

\textbf{Example 2.} The Cahn-Hilliard equation
\begin{equation}
\frac{\partial \phi}{\partial t}=-m_{0} \Delta\left(\alpha \Delta \phi-\left(1-\phi^{2}\right) \phi\right).
\end{equation}

{\em Case A.}
We set the exact solution to be  \eqref{eq:AC-CH-exact-solution-example}, and
set $\alpha=0.04, m_0=0.005$.
Convergence rates of different schemes are presented in Fig.\,\ref{Fig:CH-order-test-three-method}. The results are similar to those for the Allen-Cahn equation.

\begin{figure}[htbp]
	\centering
	\includegraphics[width=5.3cm]{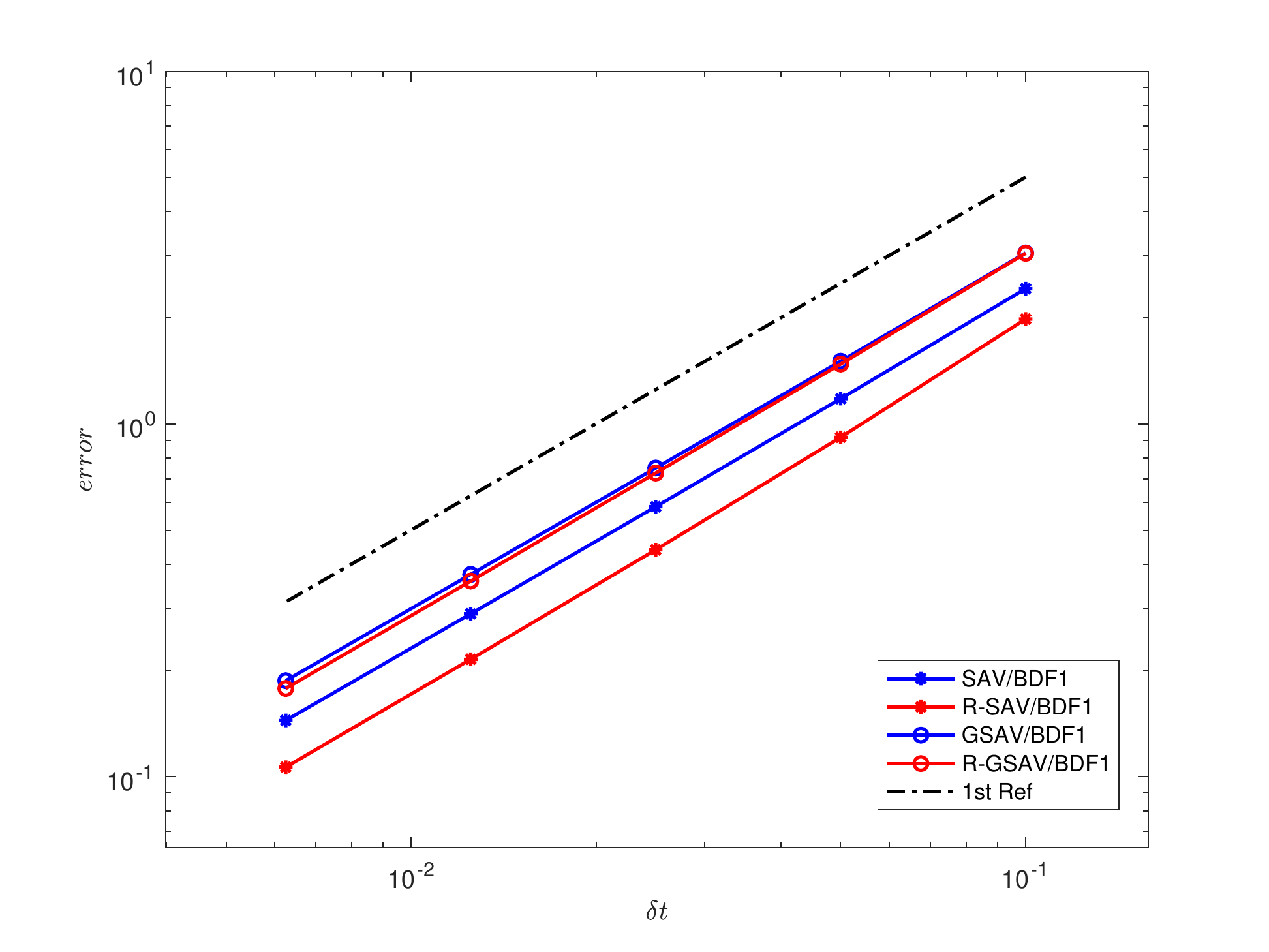}\hspace{-6mm}
	\includegraphics[width=5.3cm]{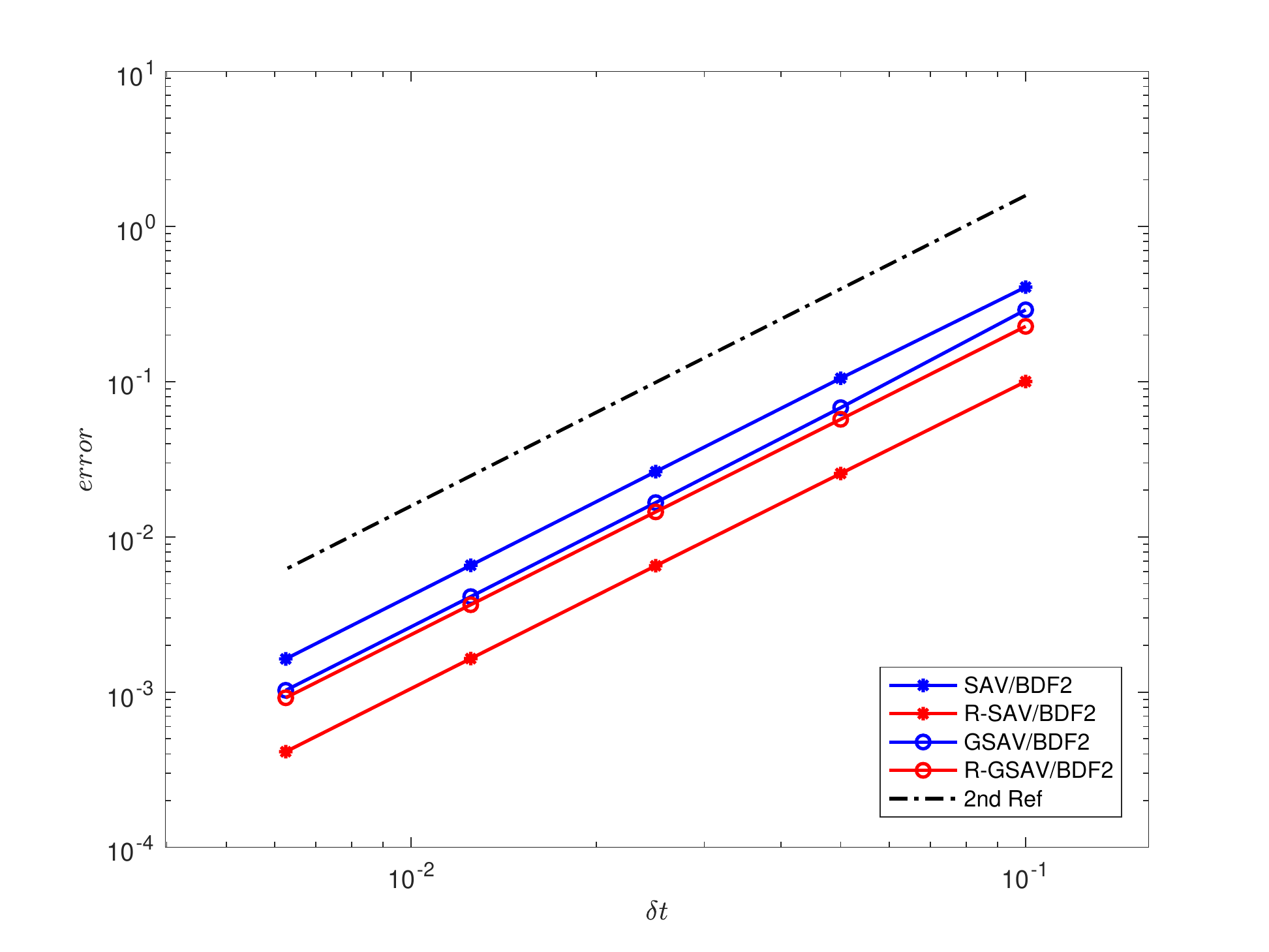}\hspace{-6mm}
	\includegraphics[width=5.3cm]{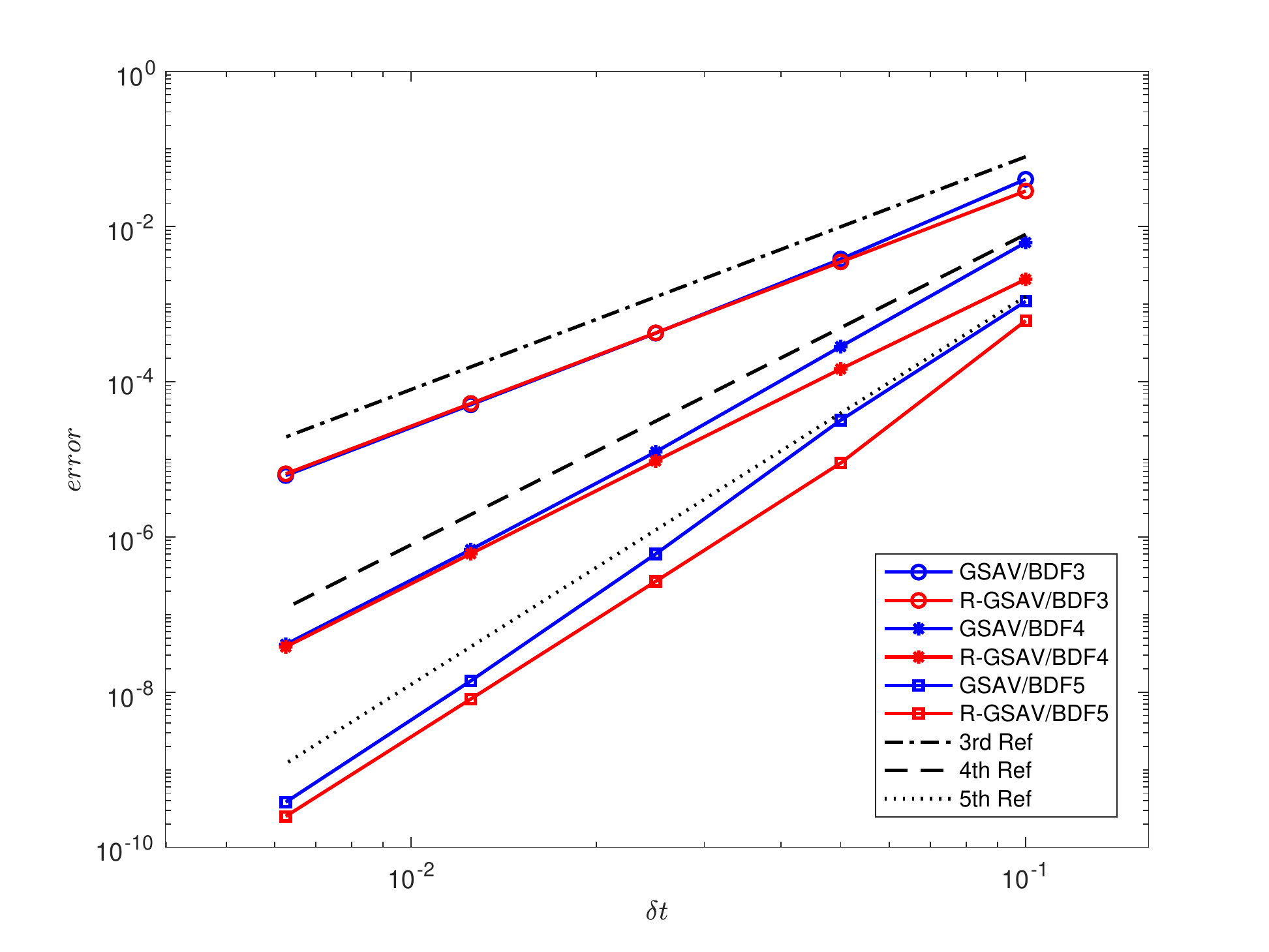}
	\hspace{-1cm}
	\caption{Example 2A. Convergence test for Cahn-Hilliard equation using SAV/BDFk, R-SAV/BDFk ($k=1, 2$), GSAV/BDFk and R-GSAV/BDFk ($k=1, 2, 3, 4, 5$) schemes.}
	\label{Fig:CH-order-test-three-method}
\end{figure}

{\em Case B}.  We set the initial condition as in \eqref{eq:AC-CH-initial-condition-star-shape}, and 
set $m_0=0.1, \alpha=0.01, \beta=1$.   
The other parameters are chosen to be the same as in Case B of Example 1.  
Numerical solutions at $T=0.1$  using the GSAV/BDF2  and  R-GSAV/BDF2 schemes  with $\delta t=1e-3$  are plotted in Fig.\,\ref{Fig:CH-comparison-star-shape} along with the  reference solution obtained by semi-Implicit/BDF2 scheme with time step $\delta t=1e-5$. 
	We observe that  with $\delta t=1e-3$, the solution by the GSAV scheme  is totally wrong  while the solution by the R-GSAV scheme is indistinguishable with the reference solution. We also observe that for this example  $\zeta_0=0$ at all times.

\begin{figure}[htbp]
\centering
$$\includegraphics[width=4.5cm]{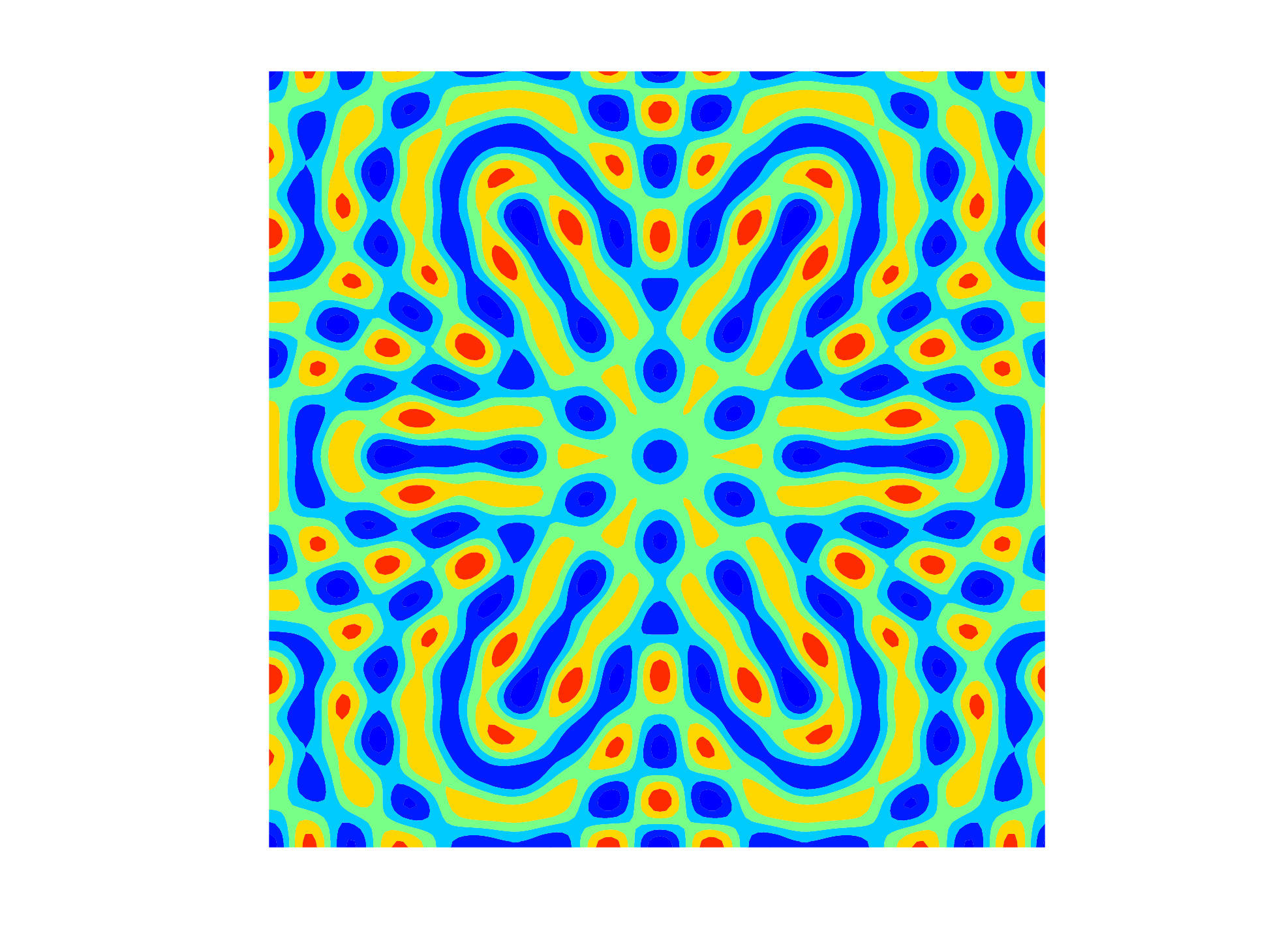}\hspace{-9mm}
	\includegraphics[width=4.5cm]{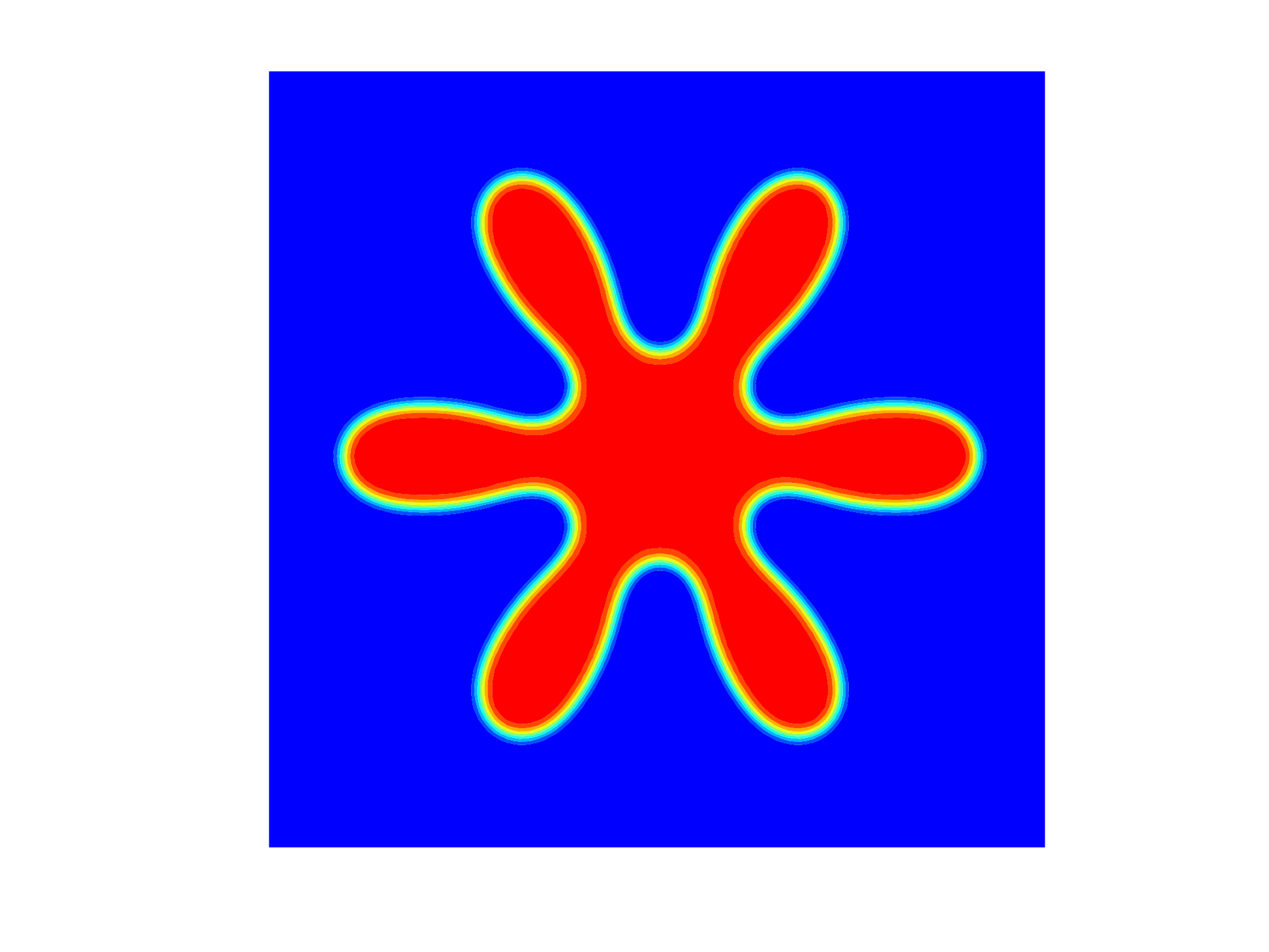}\hspace{-9mm}
	\includegraphics[width=4.5cm]{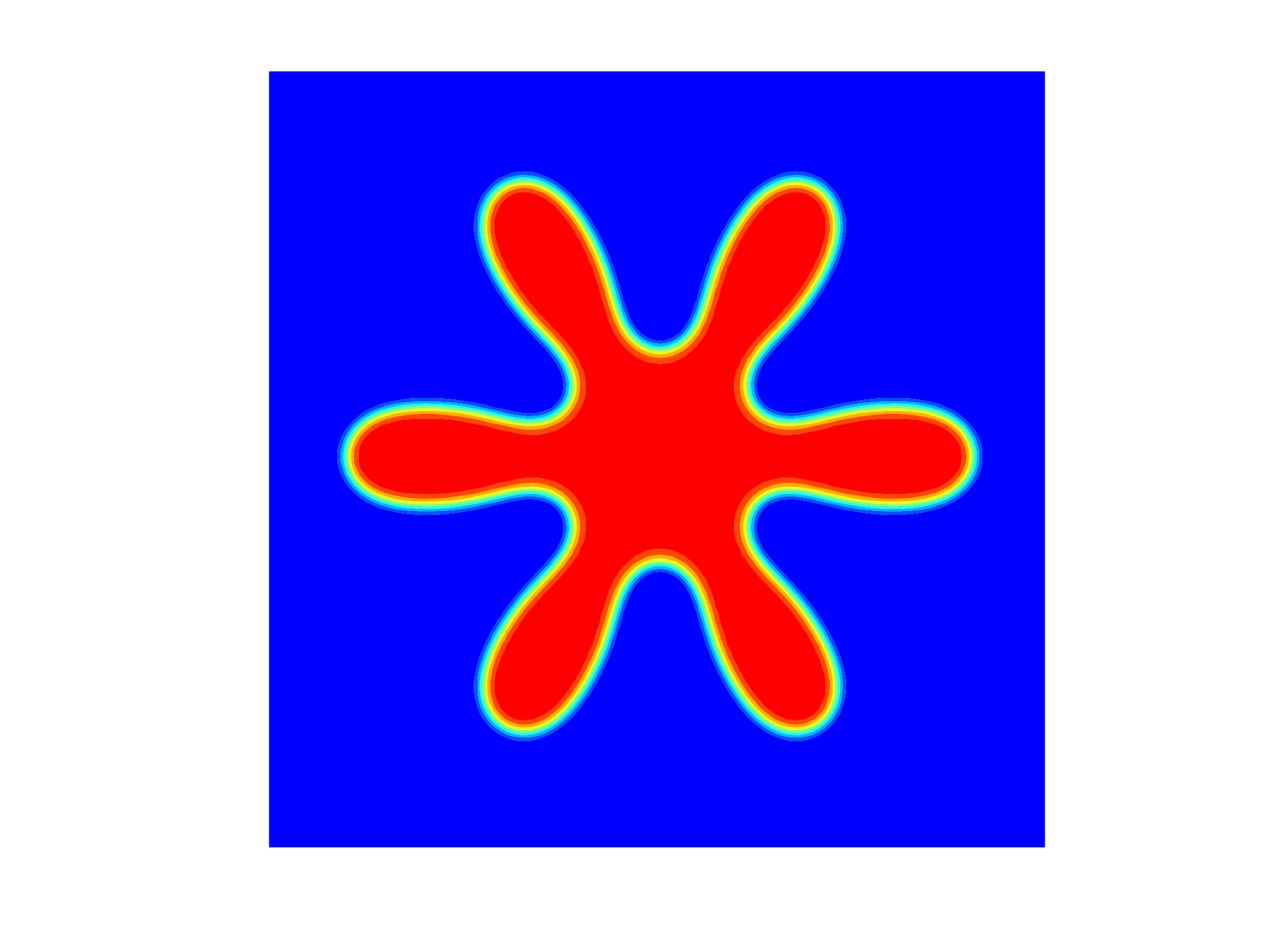}\hspace{-9mm}
		\includegraphics[width=4.5cm]{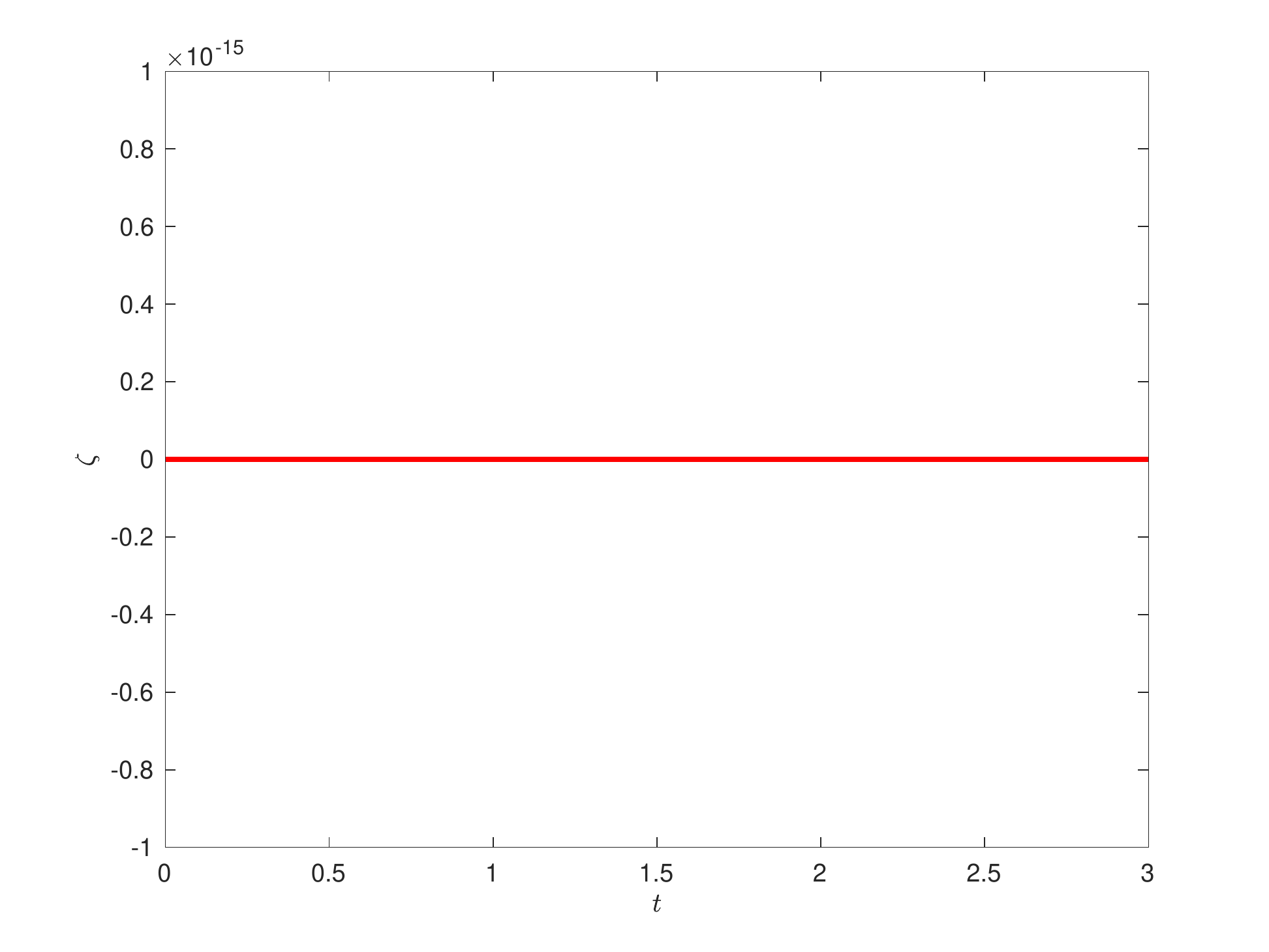}$$
	\label{Fig:CH-star-shape-zeta}
\label{Fig:CH-comparison-star-shape}
\caption{Example 2B. Profiles of $\phi$ at $T=0.1$. Dynamics driven by the Cahn-Hilliard equation using GSAV/BDF2 scheme (first), R-GSAV/BDF2 scheme (second) and reference solution (third), and the evolution of $\zeta_0$ (fourth).}
\end{figure}

\textbf{Example 3.}  In order to show that the R-GSAV approach can be used to simulate more complex nonlinear phenomena,  we consider, as an example,  the phase-field crystal model 
\begin{equation}
\left\{
\begin{array}{l}\frac{\partial \phi}{\partial t}=M \Delta \mu,  \quad \mathbf{x} \in \Omega, t>0, \\ 
\mu=(\Delta+\beta)^{2} \phi+\phi^{3}-\epsilon \phi, \quad \mathbf{x} \in \Omega, t>0, \\ 
\phi(\mathbf{x}, 0)=\phi_{0}(\mathbf{x}),
\end{array}\right.
\end{equation} 
which is a gradient flow
based on the following total free energy
\begin{equation}
	E(\phi)=\int_{\Omega}\left(\frac{1}{2} \phi(\Delta+\beta)^{2} \phi+\frac{1}{4} \phi^{4}-\frac{\epsilon}{2} \phi^{2}\right) \mathrm{d} \mathbf{x},
\end{equation}
where $M>0$ is the mobility coefficient. 
In the following simulations, we choose $M=1, \beta=1$. 

{\em Case A.}  Crystal growth in a super-cooled liquid in $2D$. 
We set the initial condition  to be
\begin{equation}
\phi\left(x_{l}, y_{l}, 0\right)=\bar{\phi}+C_{1}\left(\cos \left(\frac{C_{2}}{\sqrt{3}} y_{l}\right) \cos \left(C_{2} x_{l}\right)-0.5 \cos \left(\frac{2 C_{2}}{\sqrt{3}} y_{l}\right)\right), \quad l=1,2,3,
\end{equation}
where $x_l$ and $y_l$ define a local system of Cartesian coordinates that is oriented with the crystallite lattice, and the constant parameters $\bar{\phi} = 0.285, C_{1} = 0.446, C_{2} = 0.66$. 
Then, three crystallites in three small square patches with each length of $40$ which located at $(350, 400)$, $(200, 200)$, and $(600, 300)$ respectively, are defined perfectly. 
In order to generate crystallites with different orientations, we use the following affine transformation to produce rotation
\begin{equation}
x_{l}(x, y)=x \sin (\theta)+y \cos (\theta), \quad y_{l}(x, y)=-x \cos (\theta)+y \sin (\theta),
\end{equation}
where angles are chosen as $\theta=-\frac{\pi}{4}, 0, \frac{\pi}{4}$ respectively. 
We choose $1024^2$ Fourier modes to discretize the space and use relatively small the time step $\delta t =0.02$ for better accuracy. 
And we take the other parameters $\epsilon=0.25, T=2000$.
Fig.\,\ref{Fig:PFC-R-newSAV-rand} shows crystal growth in a supe-rcooled liquid driven by the PFC equation using the R-GSAV/BDF2 scheme. 
It also demonstrate that the different alignment of the crystallites causes defects and dislocations.  These results are consistent with those in \cite{li2020stability, yang2017linearly}. For this example, the relaxation parameter $\zeta_0$ is also zero at all times.

\begin{figure}[htbp]
\centering
\subfigure[profiles of $\phi$ at $T=0, 100, 200, 300$]{
	\hspace{-9mm}
	\includegraphics[width=4.7cm]{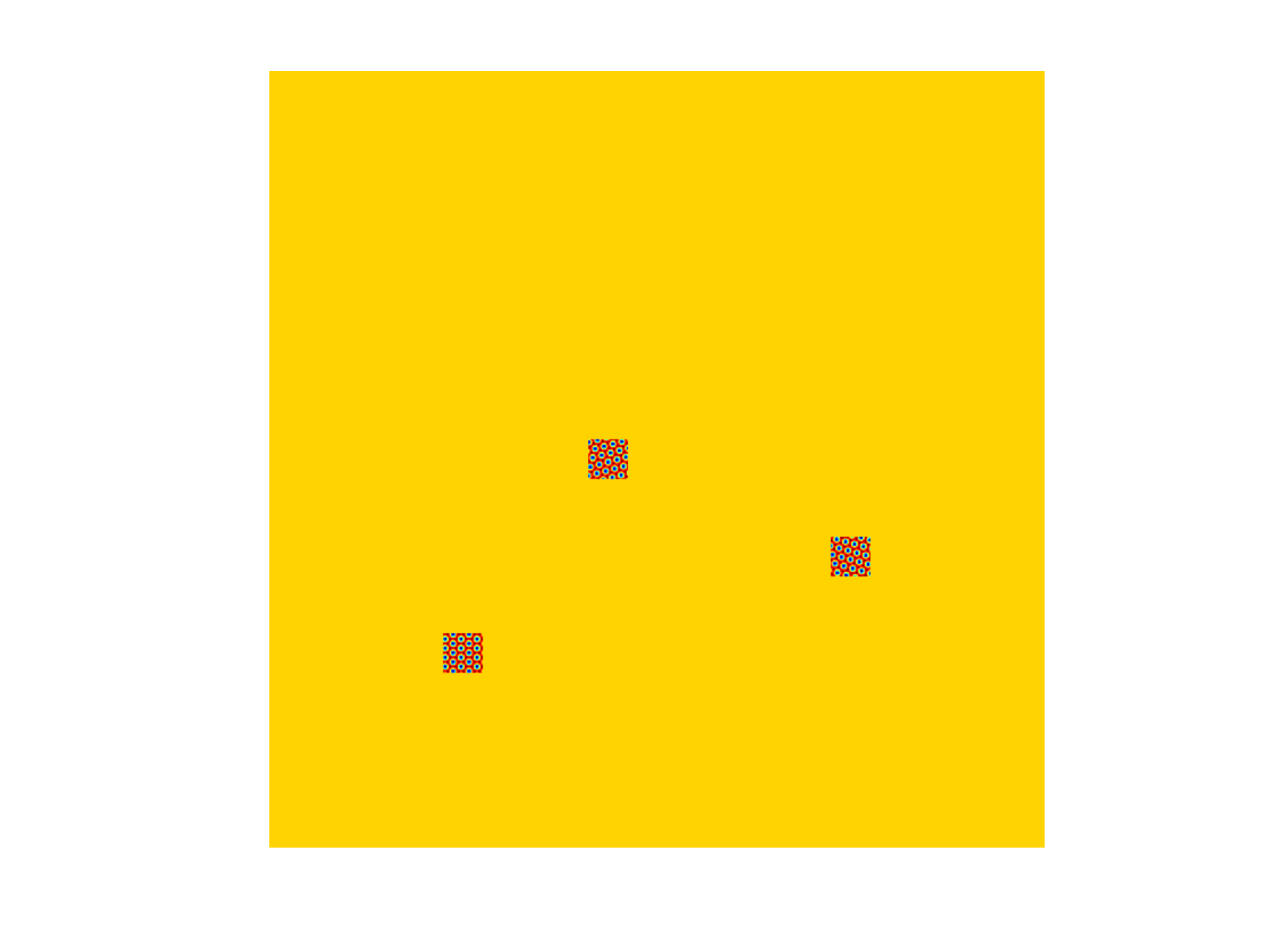}\hspace{-9mm}
	\includegraphics[width=4.7cm]{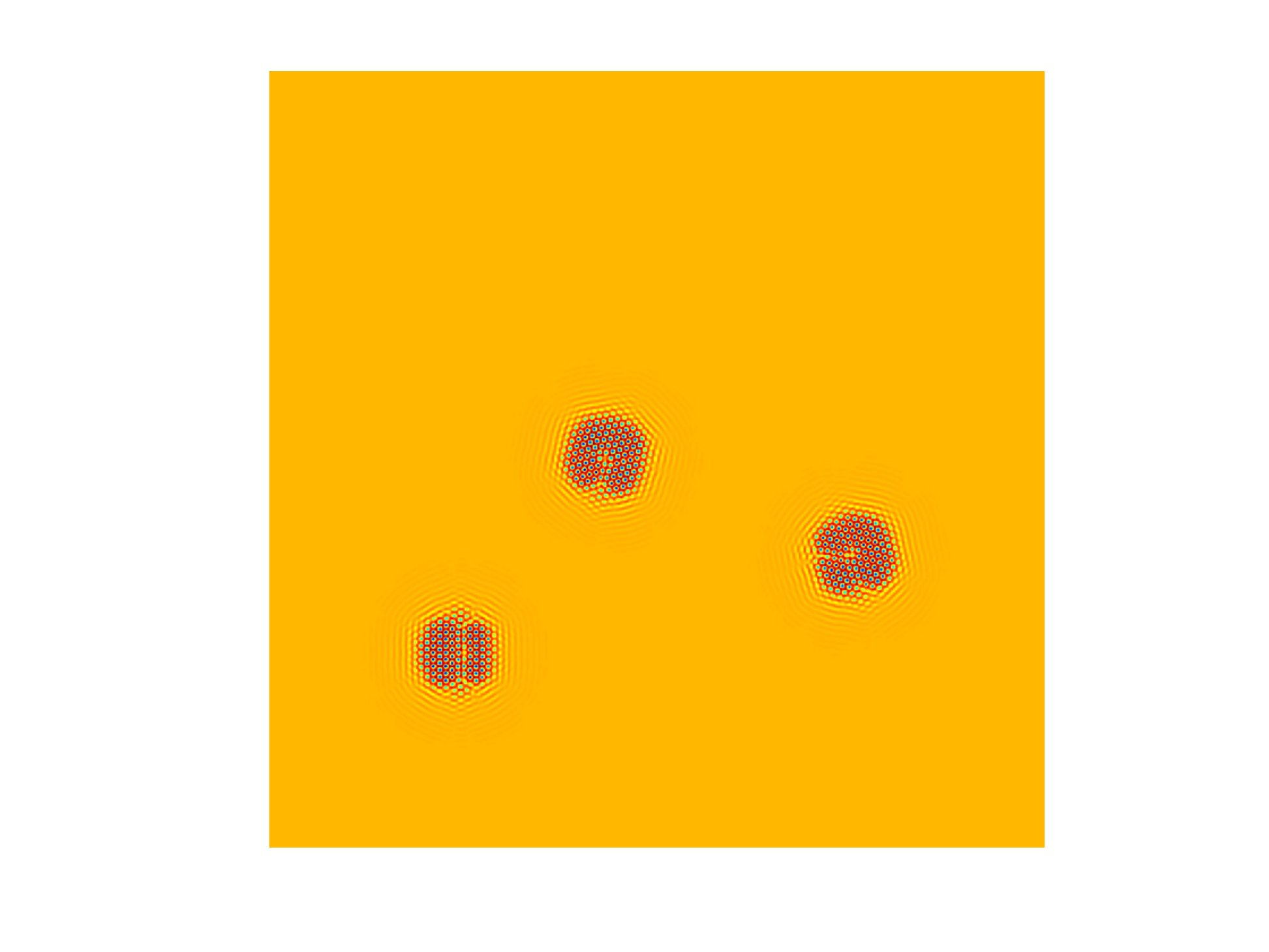}\hspace{-9mm}
	\includegraphics[width=4.7cm]{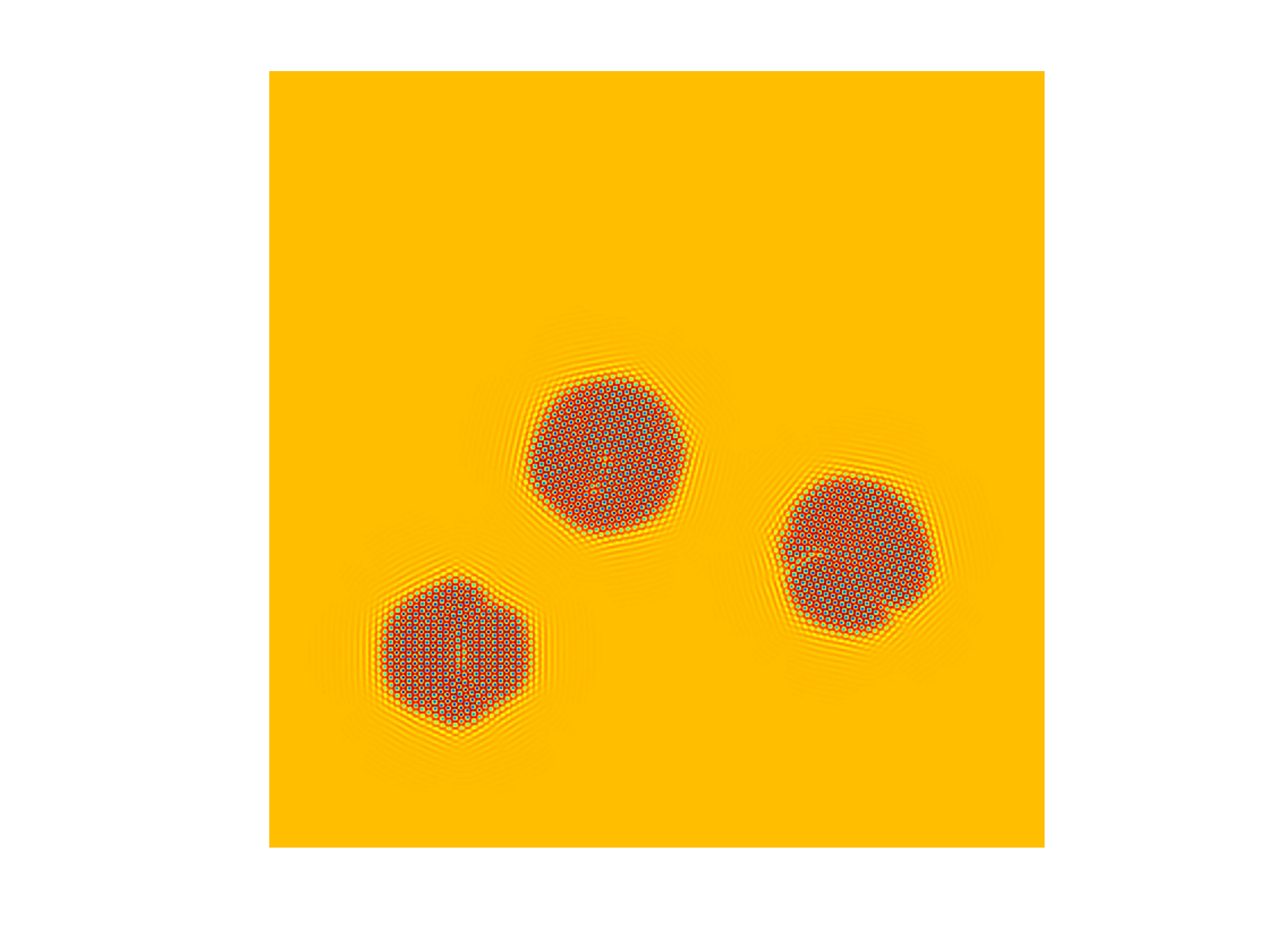}\hspace{-9mm}
	\includegraphics[width=4.7cm]{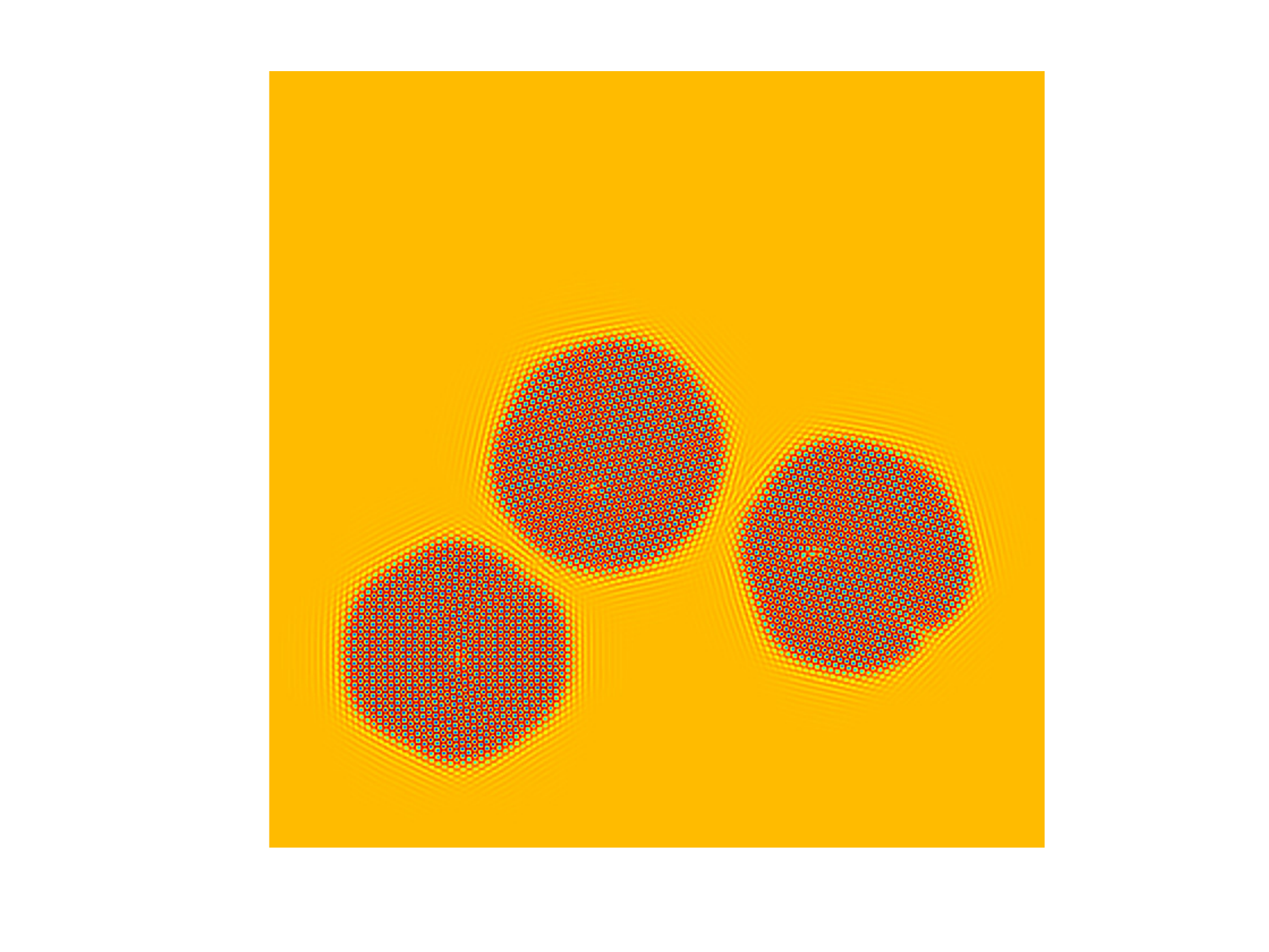}
} 
\subfigure[profiles of $\phi$ at $T=400, 500, 600, 700$]{
   \hspace{-9mm}
	\includegraphics[width=4.7cm]{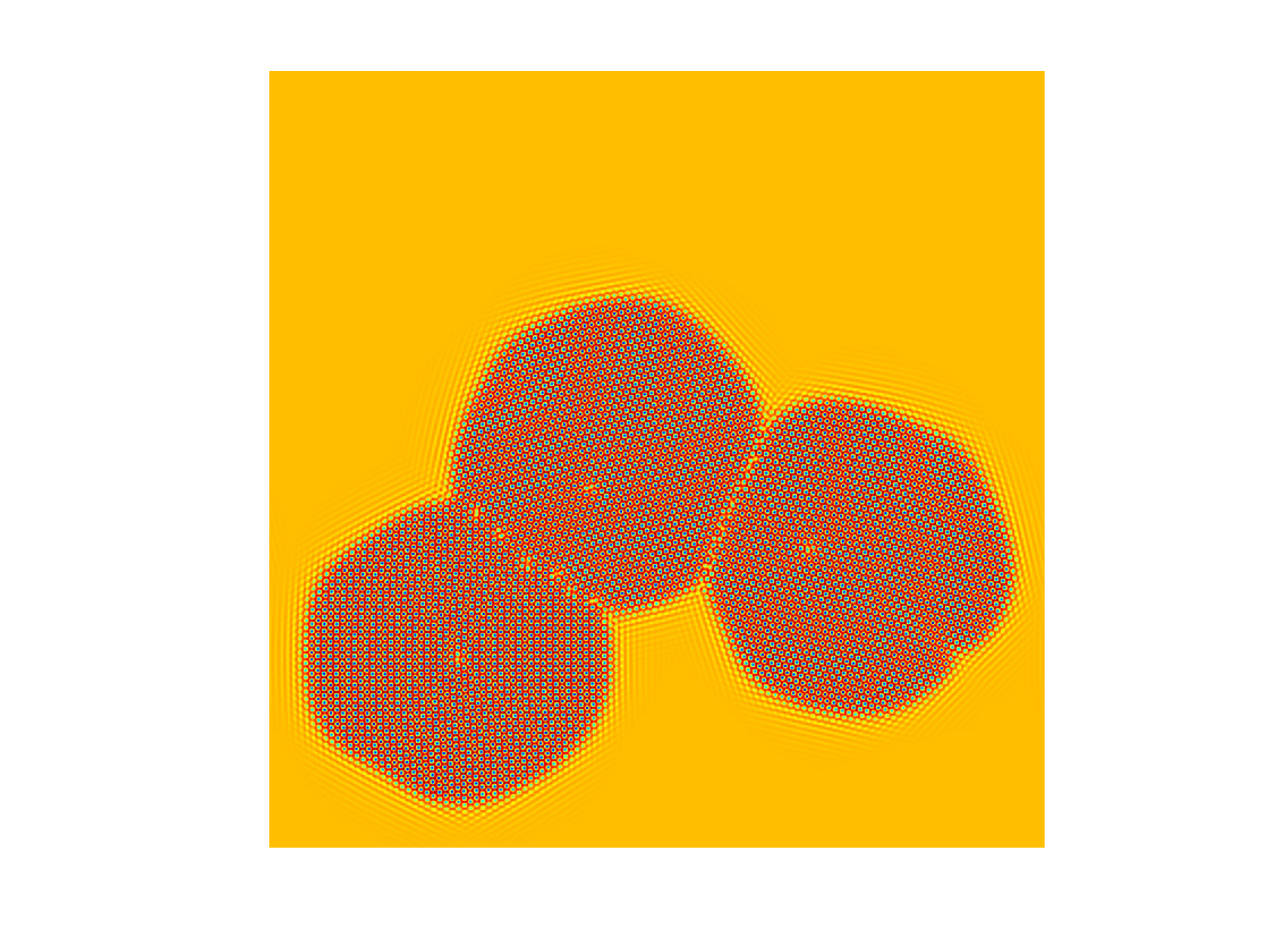}\hspace{-9mm}
	\includegraphics[width=4.7cm]{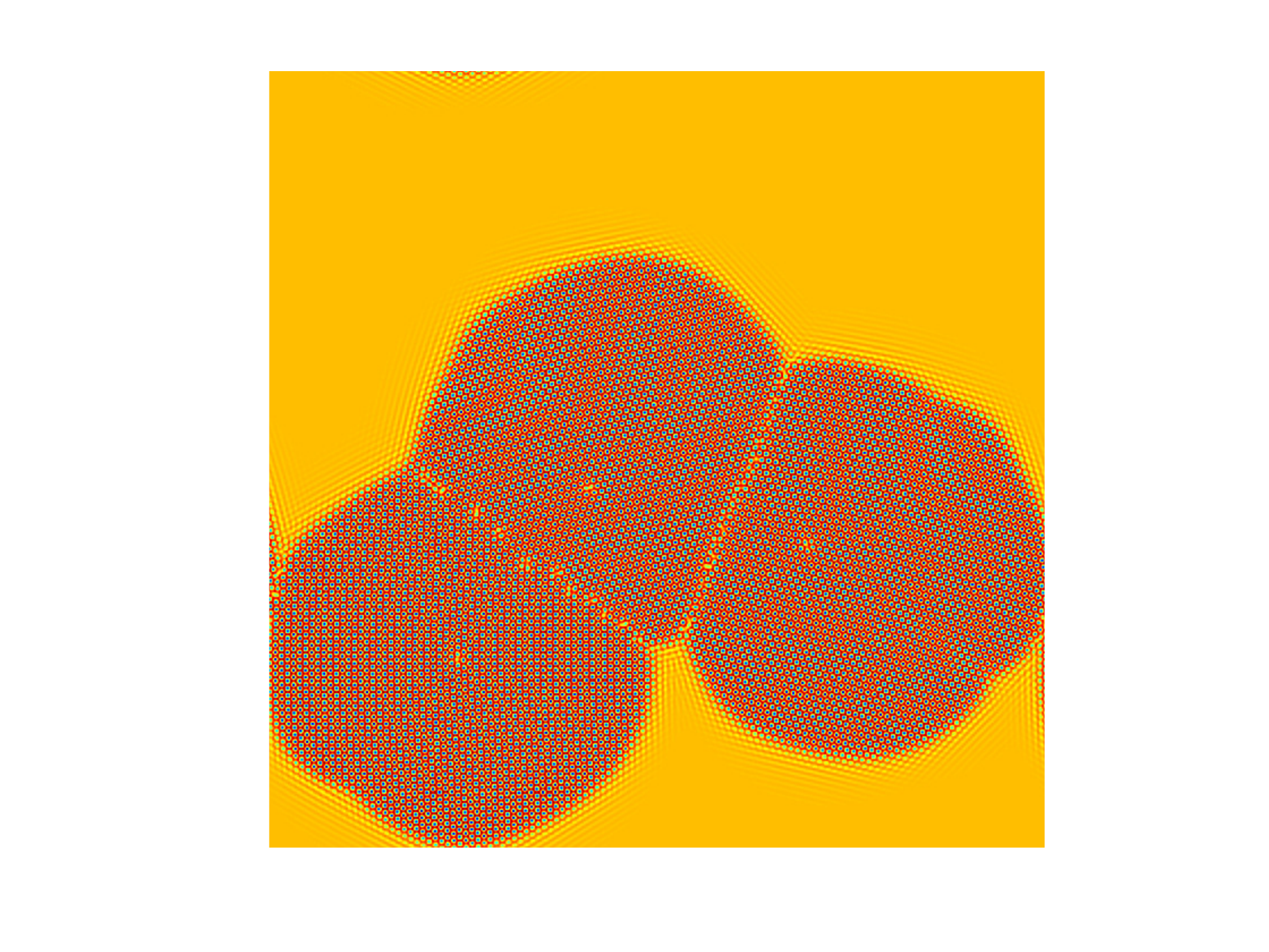}\hspace{-9mm}
	\includegraphics[width=4.7cm]{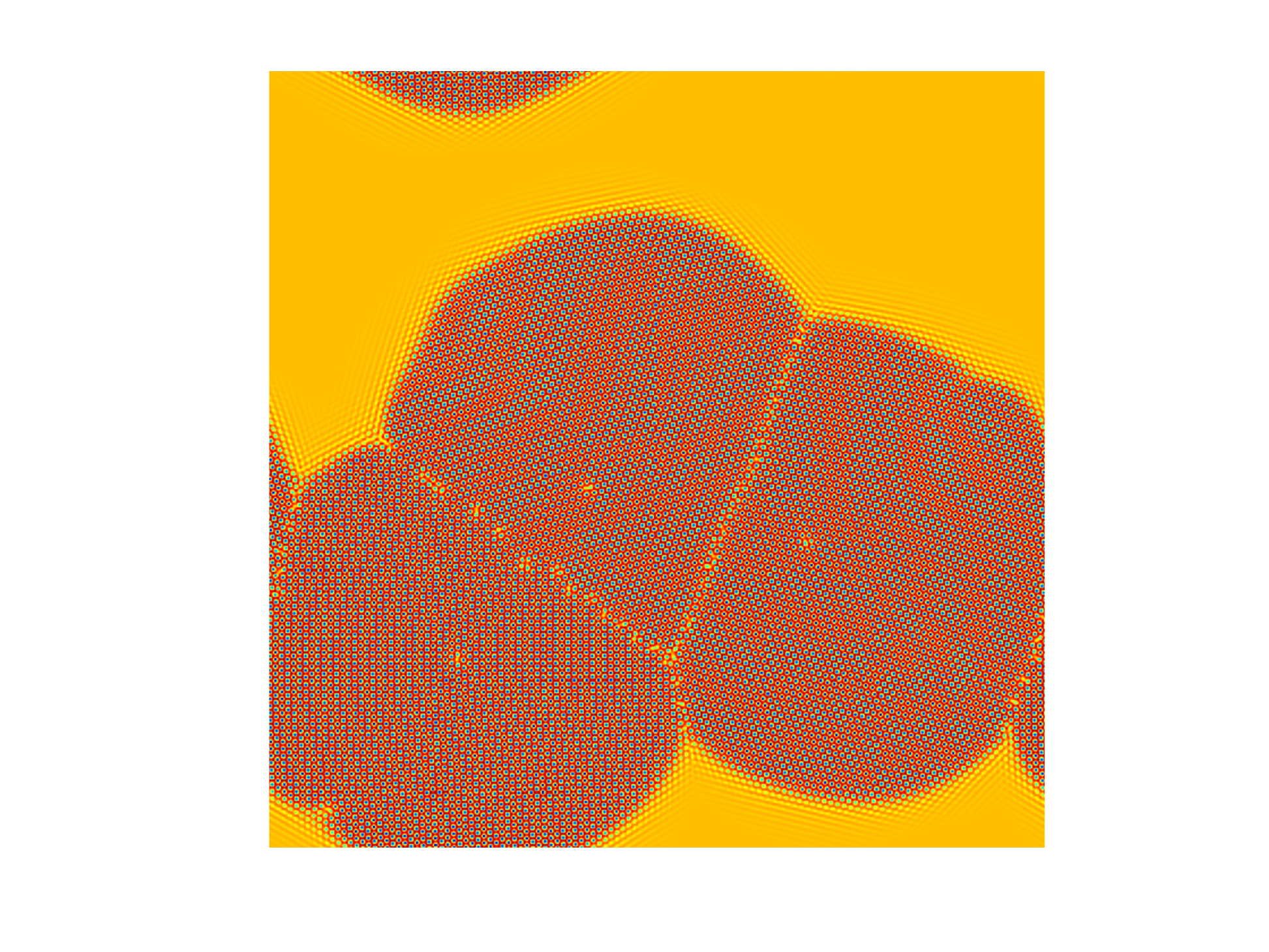}\hspace{-9mm}
	\includegraphics[width=4.7cm]{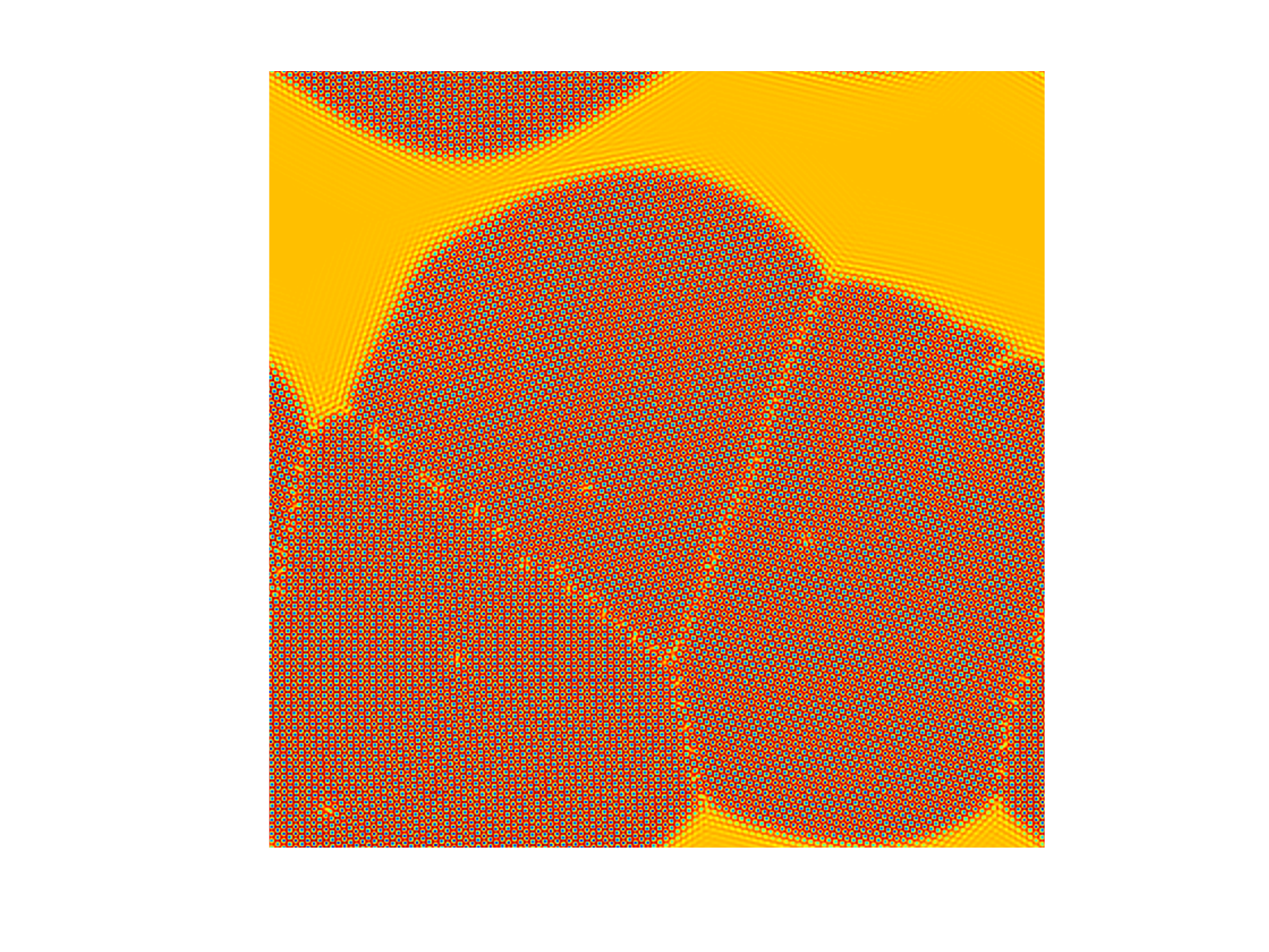}
}
\subfigure[profiles of $\phi$ at $T=800, 900, 1000, 2000$]{
   \hspace{-9mm}
	\includegraphics[width=4.7cm]{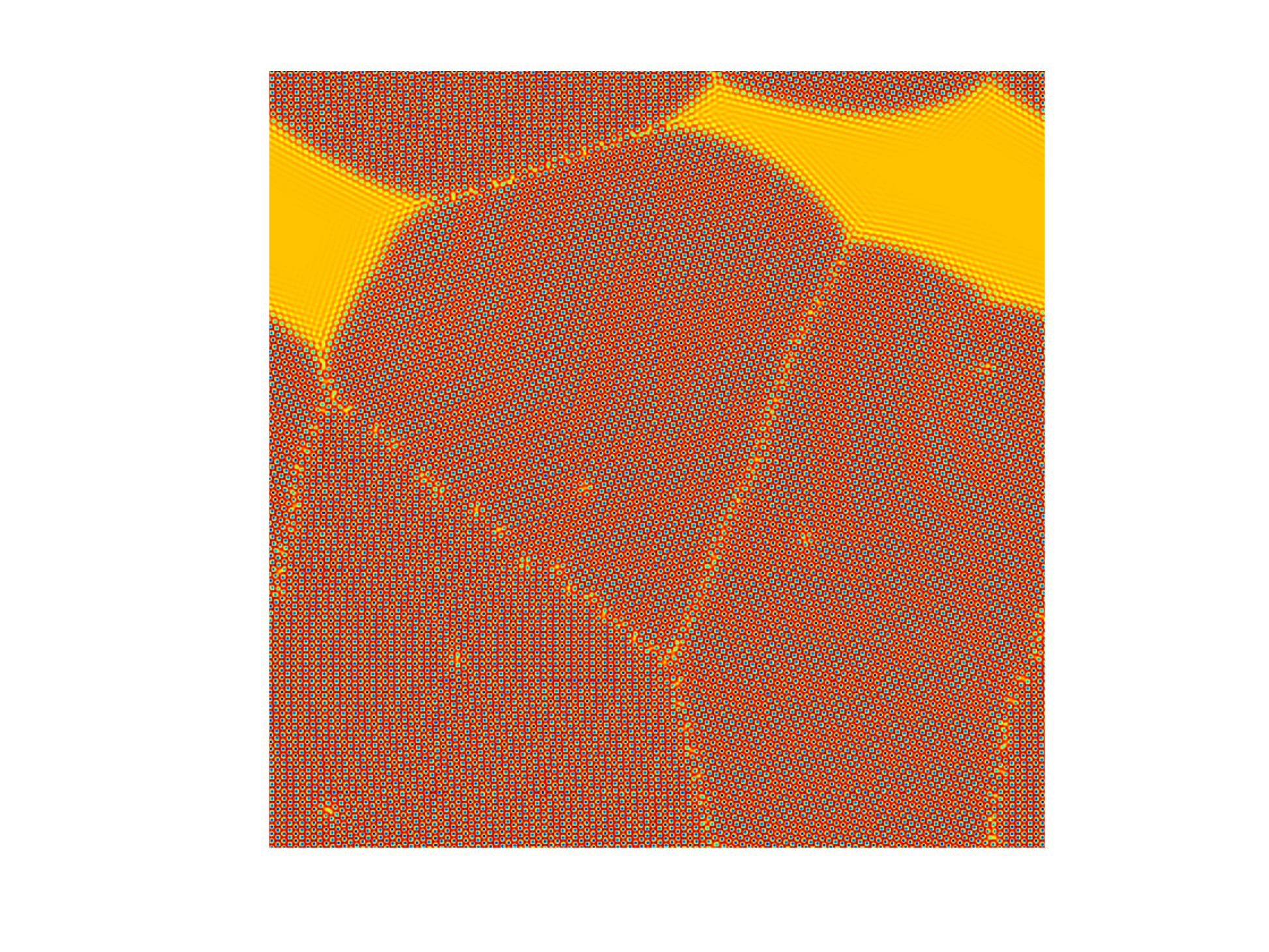}\hspace{-9mm}
	\includegraphics[width=4.7cm]{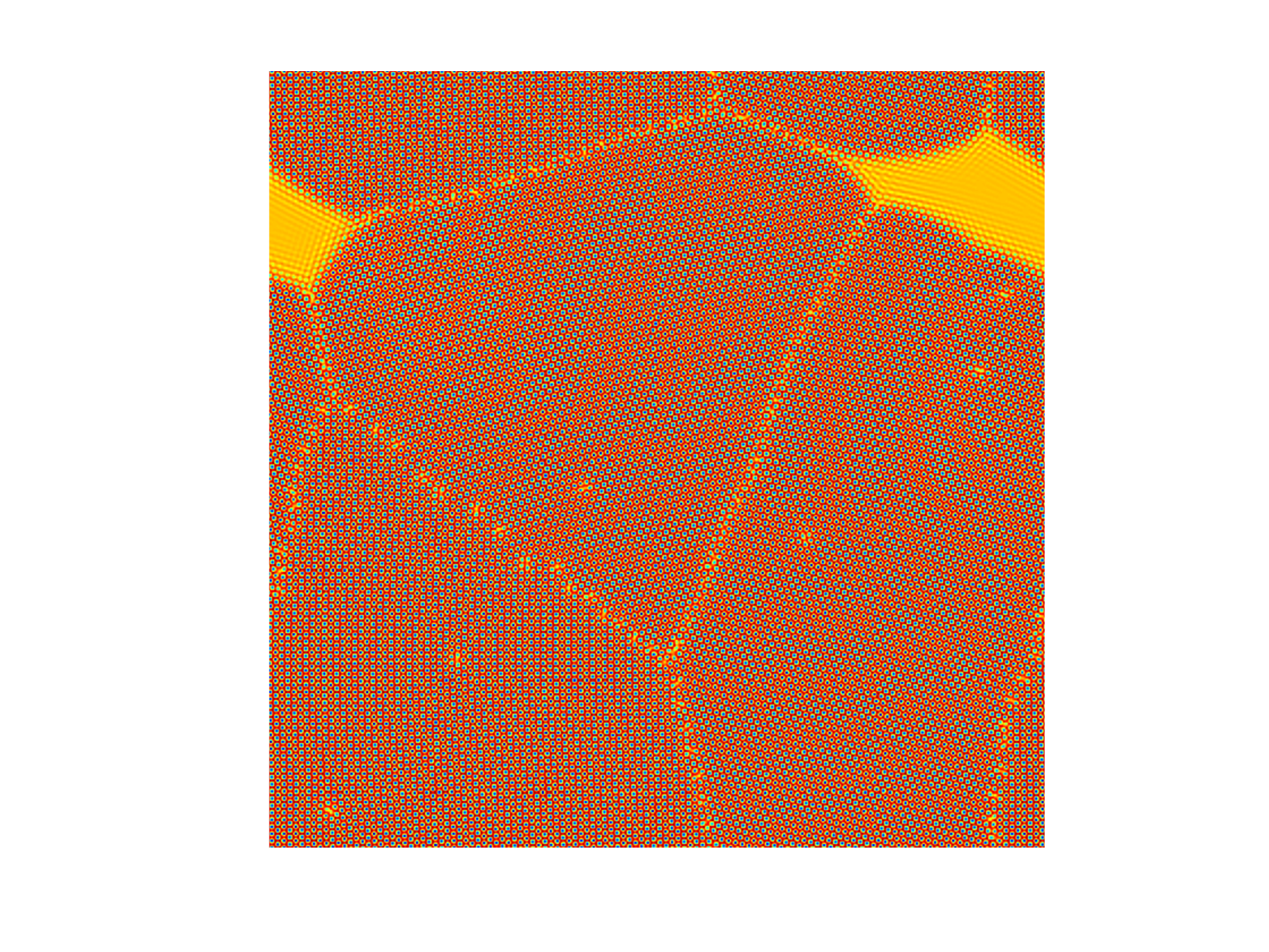}\hspace{-9mm}
	\includegraphics[width=4.7cm]{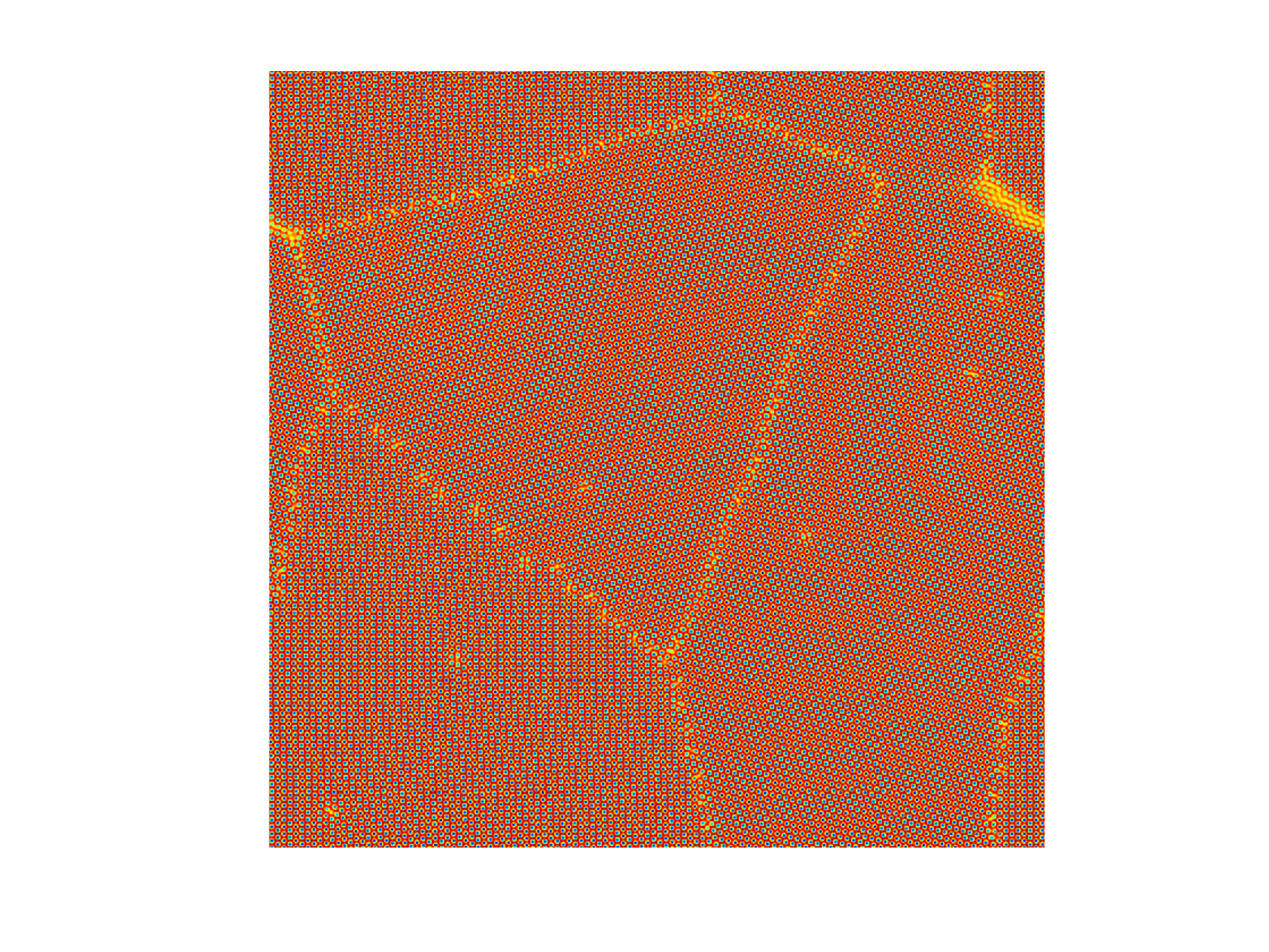}\hspace{-9mm}
	\includegraphics[width=4.7cm]{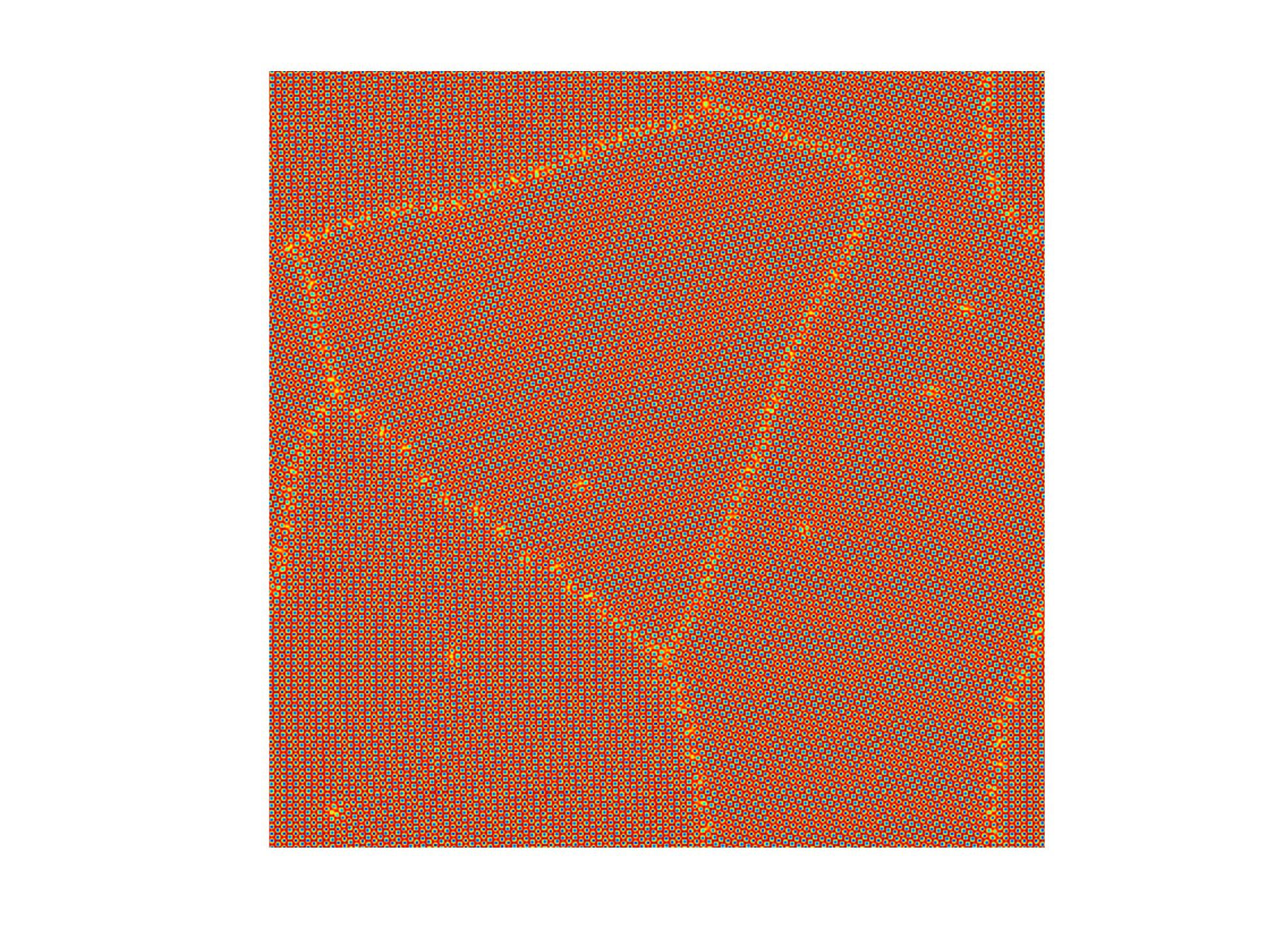}
}
\label{Fig:PFC-R-newSAV-rand}
\caption{Example 3A. Dynamics evolution of crystal growth in a supercooled liquid driven by the PFC equation using R-GSAV/BDF2 scheme. Snapshots of the numerical solution $\phi$ at $T=0, 100, 200, 300, 400, 500, 600, 700, 800, 900, 1000, 2000$, respectively.}
\end{figure}

{\em Case B.}  Phase transition behaviors in $3D$. 
We choose the initial data $\phi(x, y, t=0)=\bar{\phi}+0.01rand$ and computational domains $\left[0, 50\right]^{3}$. Other  parameters are chosen as $\epsilon=0.56, \delta=0.02, T=3000$ and $64^3$ Fourier modes. 
Fig.\,\ref{Fig:PFC-R-newSAV-rand-3D} shows the steady state
microstructure of the phase transition behavior for $\bar{\phi}=-0.20, -0.35$ and $-0.43$, respectively.
These results are also consistent with those in \cite{li2019efficient}.

\begin{figure}[htbp]
\centering
\subfigure[$\bar{\phi}=-0.2$]{
	\includegraphics[width=5.3cm]{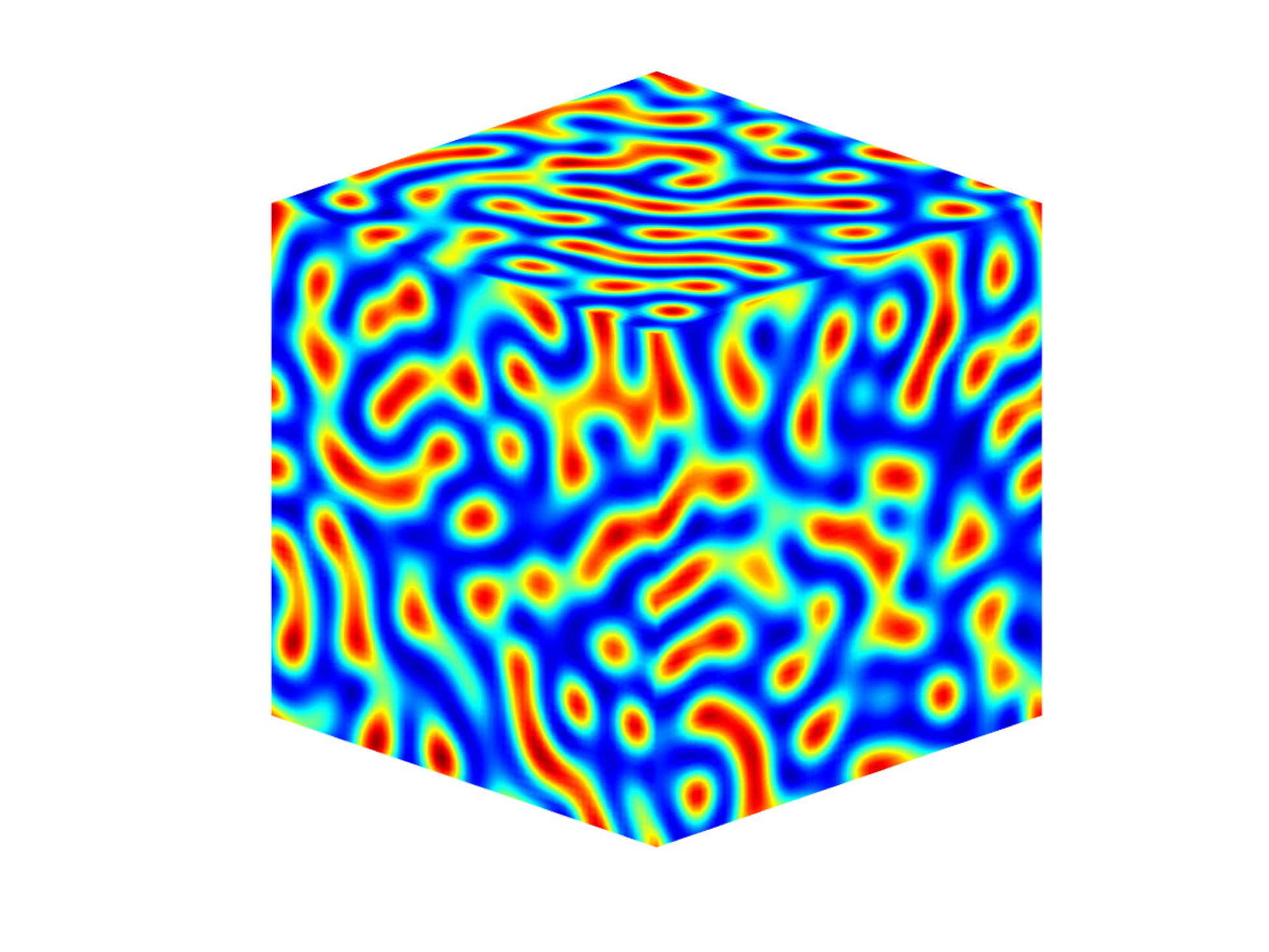}
	\includegraphics[width=5.3cm]{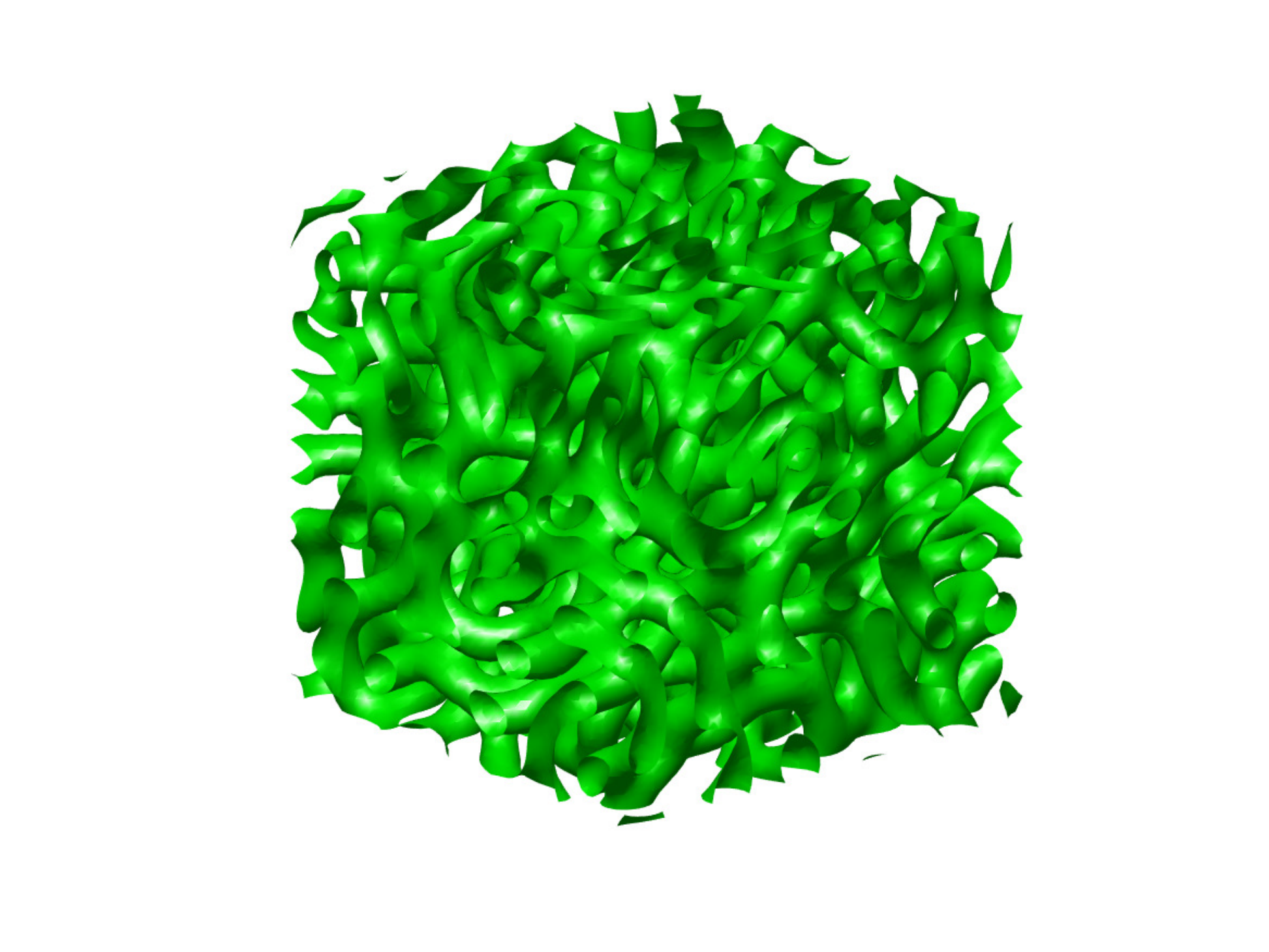}

} 
\subfigure[$\bar{\phi}=-0.35$]{
	\includegraphics[width=5.3cm]{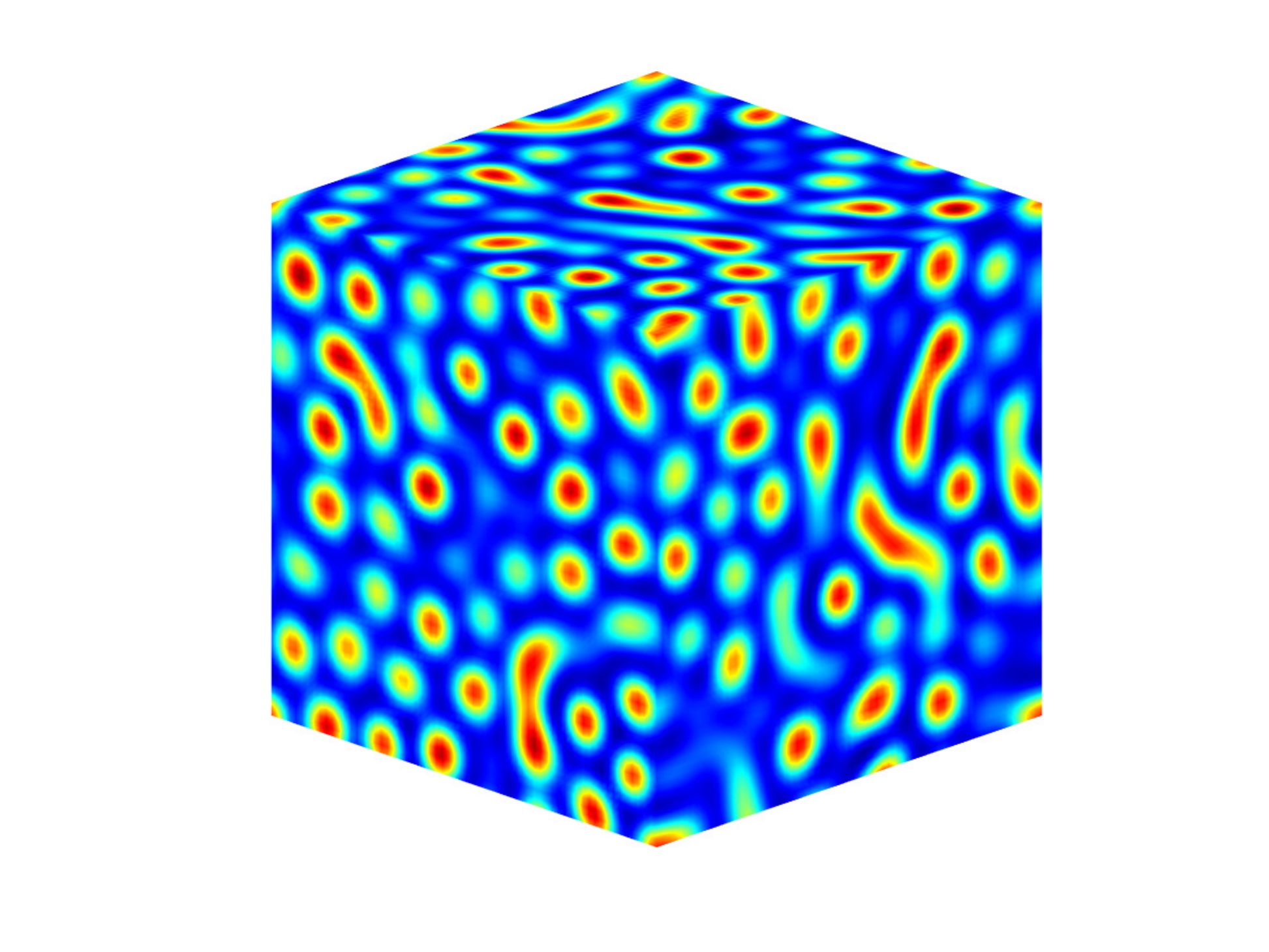}
	\includegraphics[width=5.3cm]{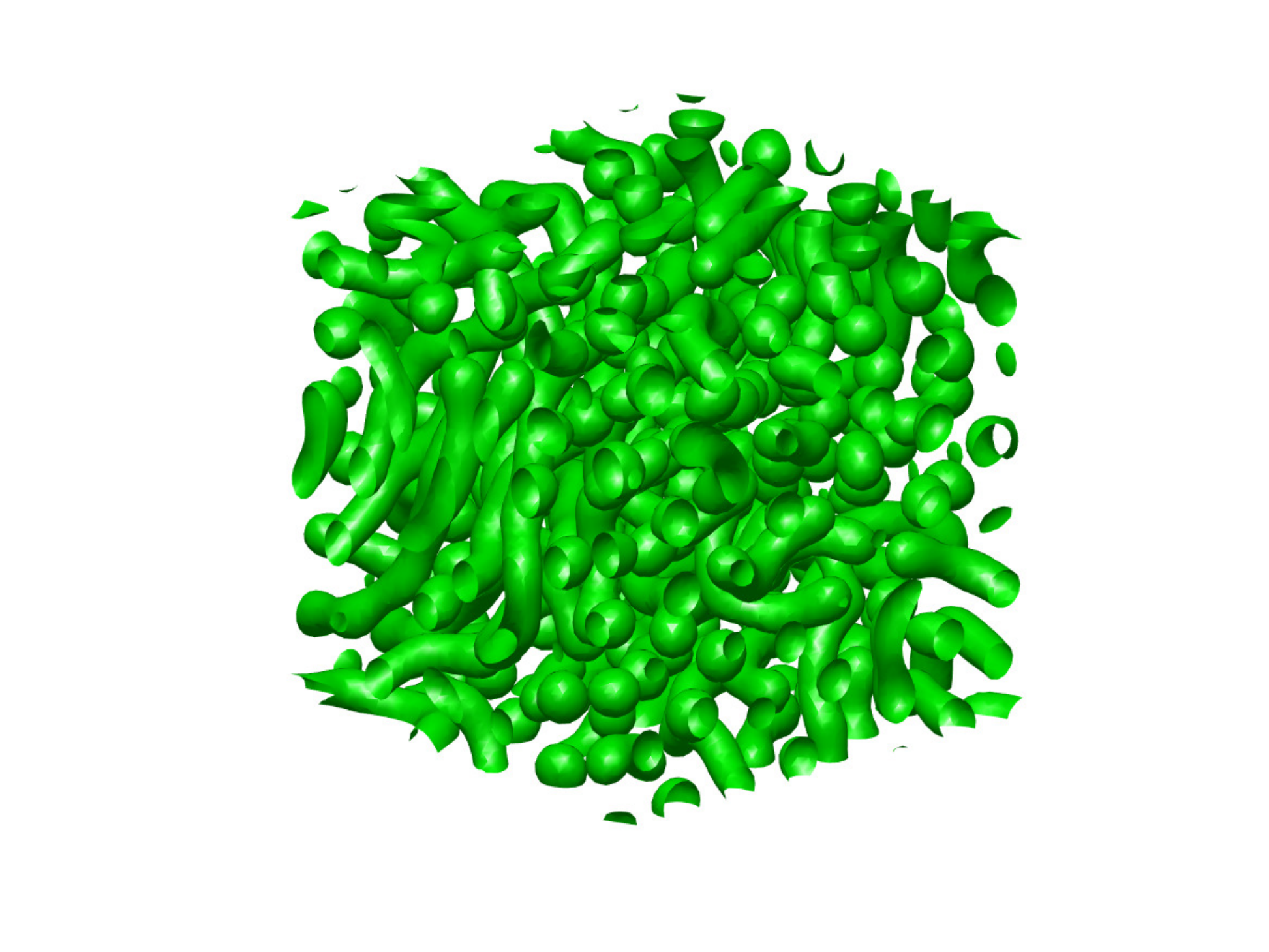}
}
\subfigure[$\bar{\phi}=-0.43$]{
	\includegraphics[width=5.3cm]{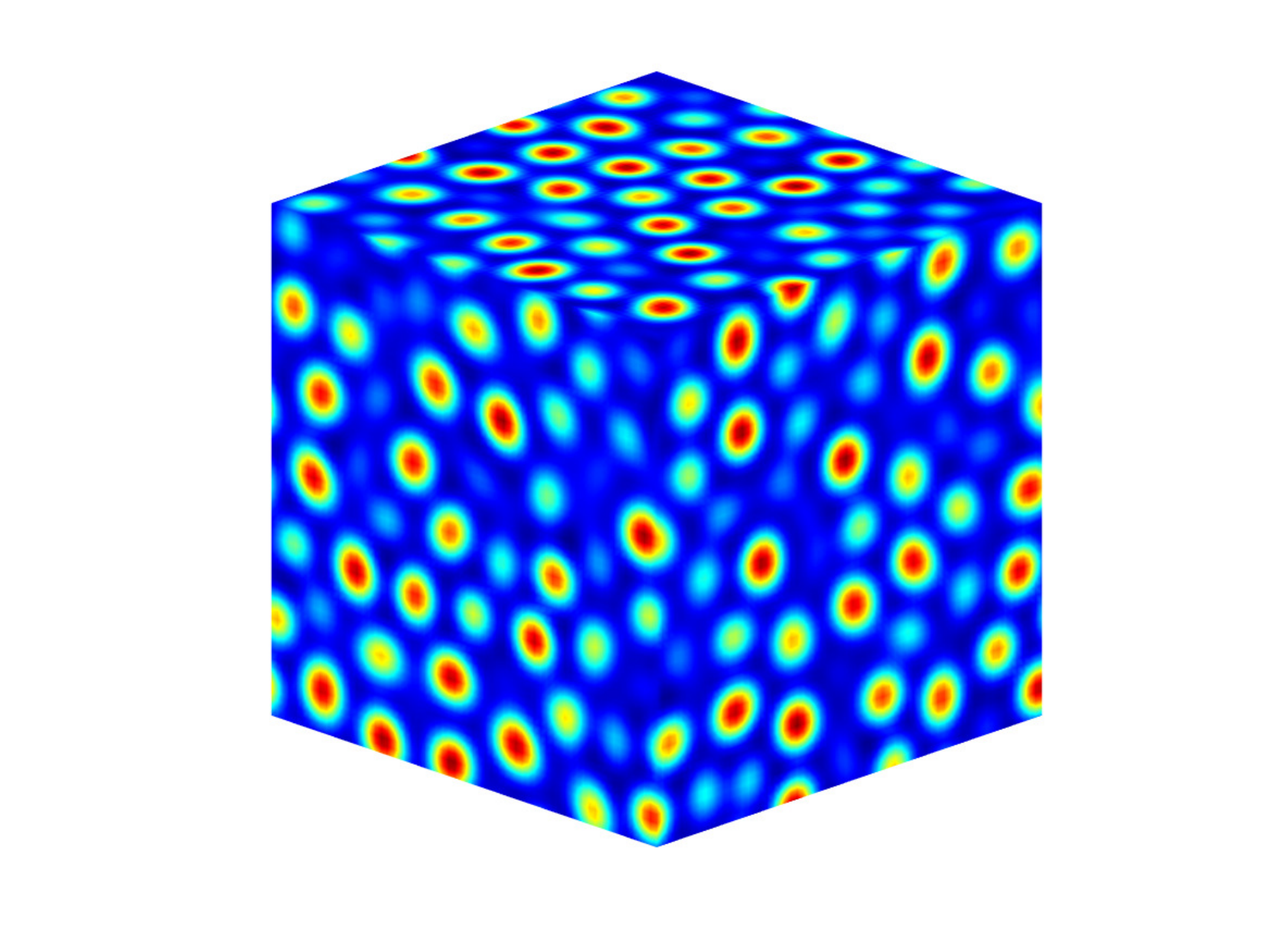}
	\includegraphics[width=5.3cm]{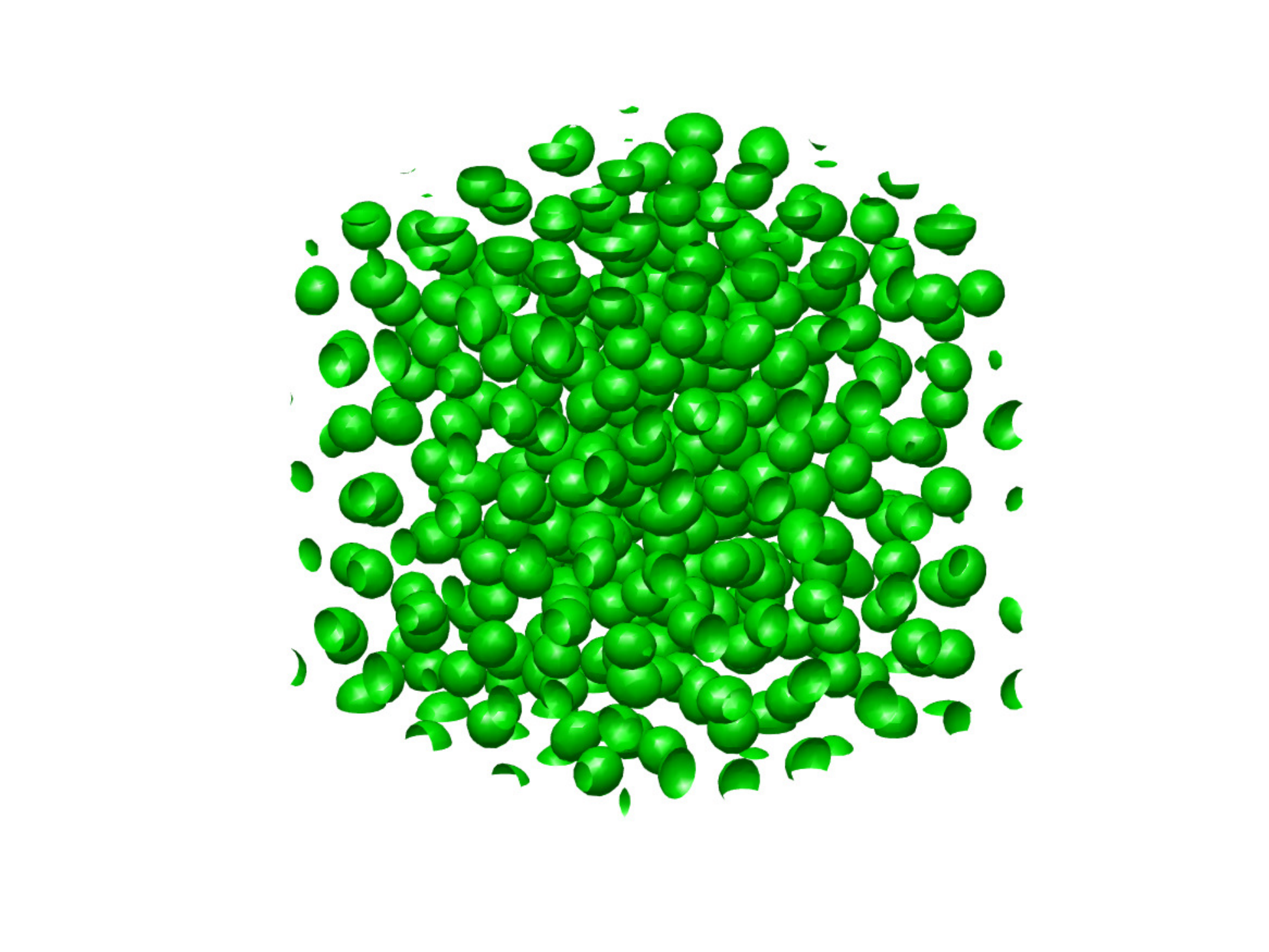}}
\label{Fig:PFC-R-newSAV-rand-3D}
\caption{Example 3B. Evolution of $\phi$ driven by the PFC equation using R-GSAV/BDF2 scheme. Snapshots of density field $\phi$ (left) and isosurface plots of $\phi=0$ (right) at $T=3000$.}
\end{figure}


\textbf{Example 4.} 
In this example, we  use the phase-field vesicle membrane (PFVM) model \cite{cheng2018multiple, cheng2020global} as an example to demonstrate how to construct relaxed MSAV schemes.

Since the vesicle membrane is area and volume preserving, we consider the following penalized free energy 
\begin{equation}
	E_{tot}(\phi)=E_{b}(\phi)+\frac{1}{2 \sigma_{1}}(A(\phi)-\alpha)^{2}+\frac{1}{2 \sigma_{2}}(B(\phi)-\beta)^{2},
\end{equation}
where $\sigma_{1}$ and $\sigma_{2}$ are two small parameters, and $\alpha, \beta$ represent the initial volume and surface area, the bending energy $E_b(\phi)$, volume $A(\phi)$ and surface area $B(\phi)$ of the vesicle are defined by 
\begin{equation}
	E_{b}(\phi)=\frac{\epsilon}{2} \int_{\Omega}\left(-\Delta \phi+\frac{1}{\epsilon^{2}} G(\phi)\right)^{2} \mathrm{d} \boldsymbol{x}=\frac{\epsilon}{2} \int_{\Omega} w^{2} \mathrm{d} \boldsymbol{x},
\end{equation}
\begin{equation}
	A(\phi)=\int_{\Omega}(\phi+1) \mathrm{d} \boldsymbol{x} \quad \text { and } \quad B(\phi)=\int_{\Omega}\left(\frac{\epsilon}{2}|\nabla \phi|^{2}+\frac{1}{\epsilon} F(\phi)\right) \mathrm{d} \boldsymbol{x},
\end{equation}
where
$$
w:=-\Delta \phi+\frac{1}{\epsilon^{2}} G(\phi), \quad G(\phi):=F^{\prime}(\phi), \quad F(\phi)=\frac{1}{4}\left(\phi^{2}-1\right)^{2}.
$$
Then,  the $L^2$ gradient flow associated with the above free energy is
\begin{equation}\label{eq:PFVM-model}
	\left\{
	\begin{array}{l}\phi_{t}=-M \mu, \\ 
		\mu=-\epsilon \Delta w+\frac{1}{\epsilon} G^{\prime}(\phi) w+\frac{1}{\sigma_{1}}(A(\phi)-\alpha)+\frac{1}{\sigma_{2}}(B(\phi)-\beta)\left(-\epsilon \Delta \phi+\frac{1}{\epsilon} F^{\prime}(\phi)\right), \\ 
		w=-\Delta \phi+\frac{1}{\epsilon^{2}} G(\phi),
	\end{array} \right.
\end{equation}
with the boundary conditions being
\begin{equation} \label{eq:PFVM-BC}
	\text { (i) periodic or (ii) }\left.\partial_{\mathbf{n}} \phi\right|_{\partial \Omega}=\left.\partial_{\mathbf{n}} \Delta \phi\right|_{\partial \Omega}=0,
\end{equation}
and $M$ is the mobility constant. 
Then, one can easily see  that the system \eqref{eq:PFVM-model} admits the following energy law
\begin{equation}
	\frac{\mathrm{d}}{\mathrm{d} t} E_{t o t}(\phi)=-M\|\mu\|^{2}.
\end{equation}
We observe that \eqref{eq:PFVM-model} contains small parameters $\epsilon,\,\sigma_1,\,\sigma_2$, but $\sigma_1$ is only associated with the linear non-local term $A(\phi)$ so it can be treated implicitly. Hence, we need to introduce two SAVs to deal with the nonlinear terms associated  with $\epsilon$ and $\sigma_2$ separately. More precisely, we set
\begin{eqnarray}
	E_{1}(\phi)=E_b(\phi),\quad E_{2}(\phi):=\frac{1}{2 \sigma_{2}}(B(\phi)-\beta)^{2}.
\end{eqnarray}
Then, we can apply the R-MGSAV scheme  \eqref{eq:R-MGSAV-BDFk-1}-\eqref{eq:R-MGSAV-BDFk-9} directly.

We consider the phase-field vesicle membrane model \eqref{eq:PFVM-model} in $\Omega=(-\pi, \pi)^{3}$ with  $\epsilon=\frac{6\pi}{128}, M=1$ and $\sigma_{1}=\sigma_{2}=0.01$. We use  the  R-MSAV/BDF$2$ scheme with  $\delta t=1e-4$ and $128^3$ Fourier modes.

{\em Case A. } We first simulate collision of two close-by spherical vesicles by consider the following initial condition for $\phi$ to describe two close-by spherical vesicles in 3D
\begin{equation}
\begin{aligned} 
\phi(x, y, z, 0) &=\tanh \left(\frac{0.28 \pi-\sqrt{x^{2}+y^{2}+(z-0.35 \pi)^{2}}}{\sqrt{2} \epsilon}\right) \\ 
&+\tanh \left(\frac{0.28 \pi-\sqrt{x^{2}+y^{2}+(z+0.35 \pi)^{2}}}{\sqrt{2} \epsilon}\right)+1.
\end{aligned}
\end{equation}
We depict the collision process in Fig.\,\ref{Fig:PFVM-R-newSAV-two-balls}.
We observe that two spheres connect within a small time interval, then merge into a capsule shape which is a steady state.   The results are consistent with those presented in \cite{cheng2018multiple}.

{\em Case B. } We start with six closeby spheres as initial condition given by
\begin{equation}
\phi(x, y, z, 0)=\sum_{i=1}^{6} \tanh \left(\frac{r_{i}-\sqrt{\left(x-x_{i}\right)^{2}+\left(y-y_{i}\right)^{2}+\left(z-z_{i}\right)^{2}}}{\sqrt{2} \epsilon}\right)+5,
\end{equation}
where $r_{i}=\frac{\pi}{6}, z_{i}=0$ for $i=1,2, \ldots, 6$,  $\left(x_{1}, x_{2}, x_{3}, x_{4}, x_{5}, x_{6}\right)=\left(-\frac{\pi}{4}, \frac{\pi}{4}, 0, \frac{\pi}{2},-\frac{\pi}{2}, 0\right)$, and \\
$\left(y_{1}, y_{2}, y_{3}, y_{4}, y_{5}, y_{6}\right)=\left(-\frac{\pi}{4},-\frac{\pi}{4}, \frac{\pi}{4}, \frac{\pi}{4}, \frac{\pi}{4},-\frac{3 \pi}{4}\right)$. 

In Fig.\,\ref{Fig:PFVM-R-newSAV-six-balls}, we plot snapshots of iso-surfaces $\phi=0$ at $t = 0, 0.02, 0.1, 0.5, 2$ by using the R-MGSAV/BDF$2$ scheme. 
It shows that the initially separated spheres connect with each other gradually and finally merge into a big vesicle with two small holes in upper and lower parts respectively. 
The results are also consistent with those presented in \cite{cheng2020global}.




\begin{figure}[htbp]
\centering
\subfigure[profiles of $\phi=0$ at $T=0, 0.02, 0.1$]{
	\includegraphics[width=5.3cm]{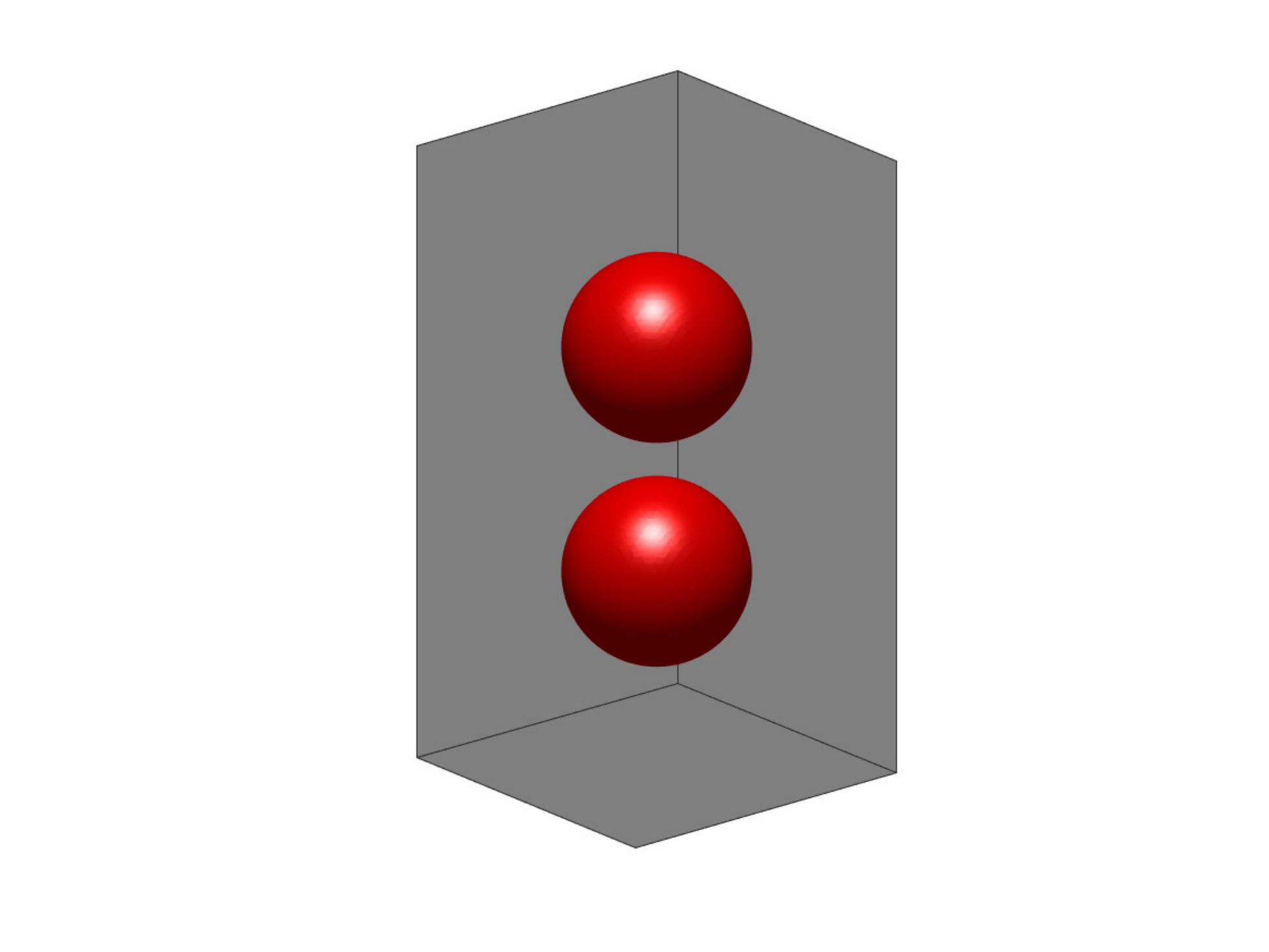}\hspace{-9mm}
	\includegraphics[width=5.3cm]{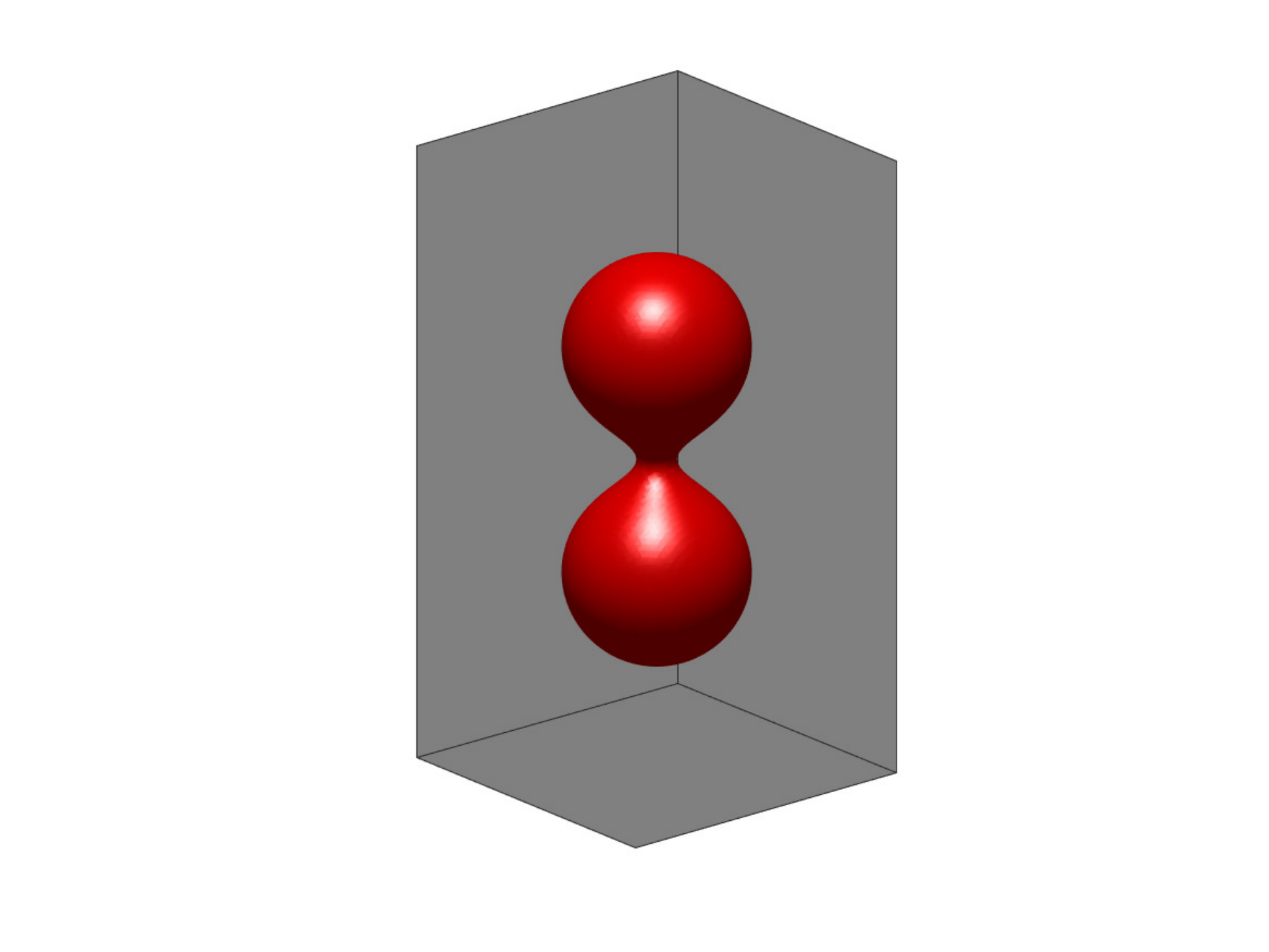}\hspace{-9mm}
	\includegraphics[width=5.3cm]{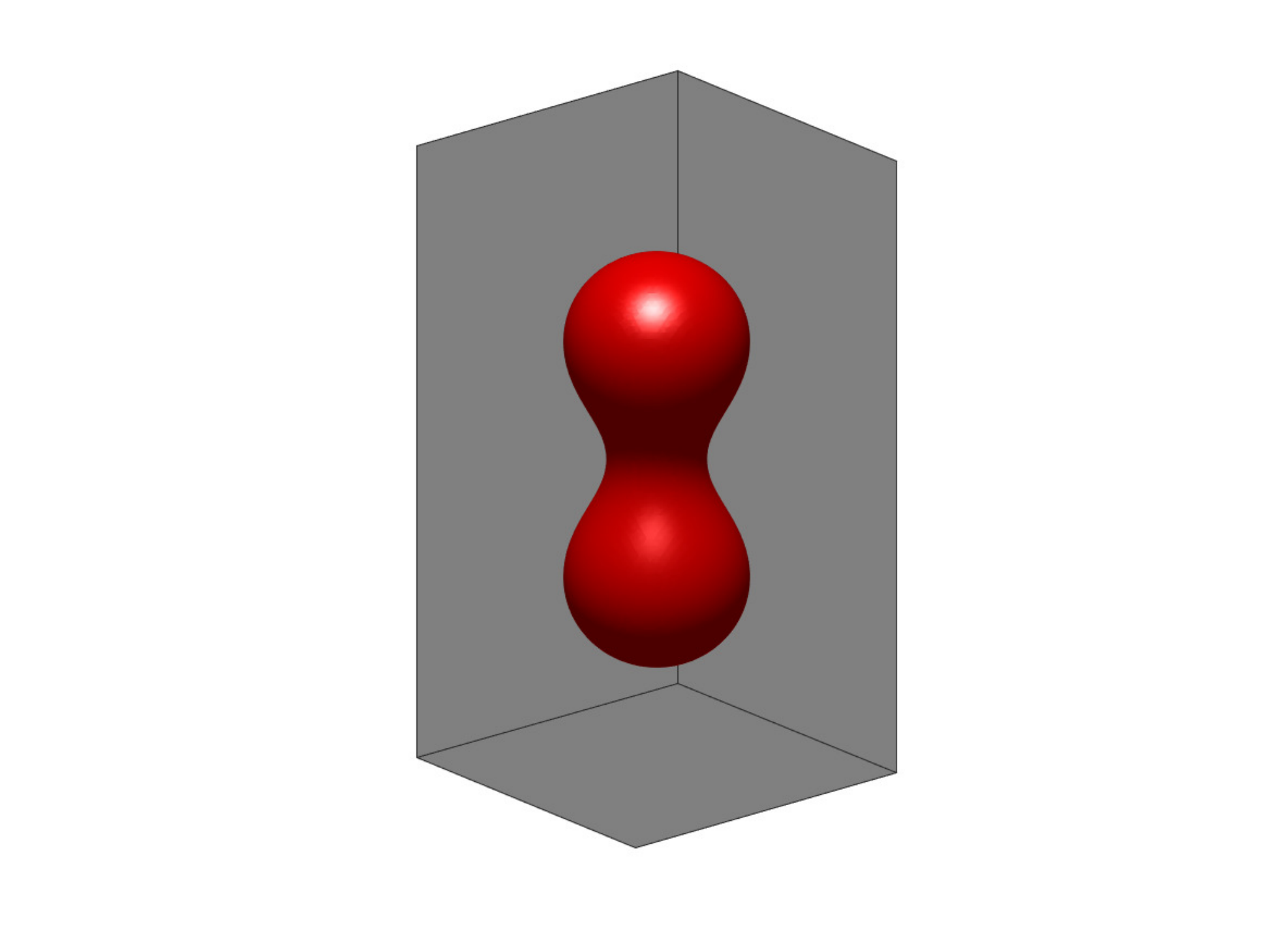}
} 
\subfigure[profiles of $\phi=0$ at $T=0.5, 1, 2$]{
	\includegraphics[width=5.3cm]{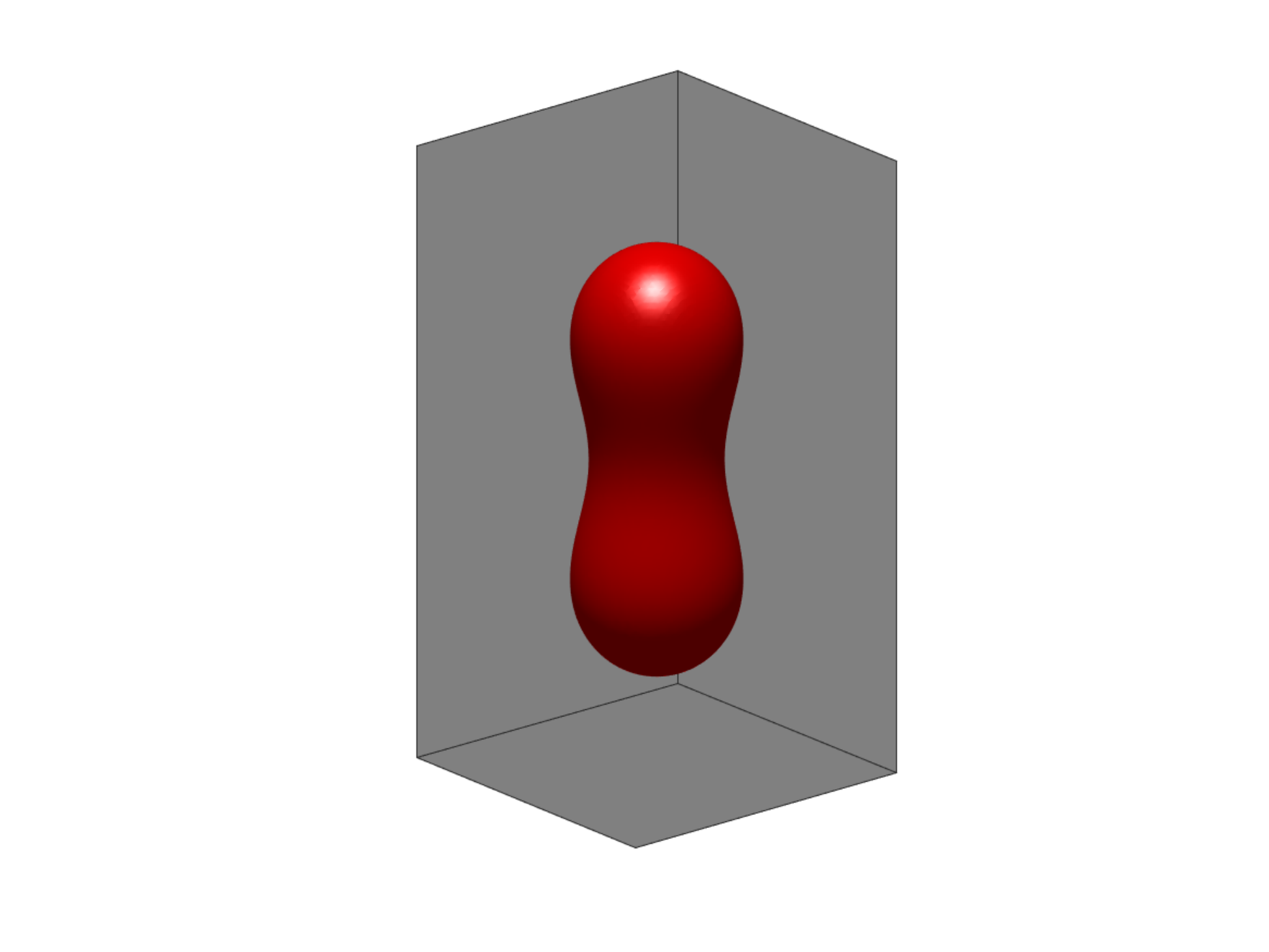}\hspace{-9mm}
	\includegraphics[width=5.3cm]{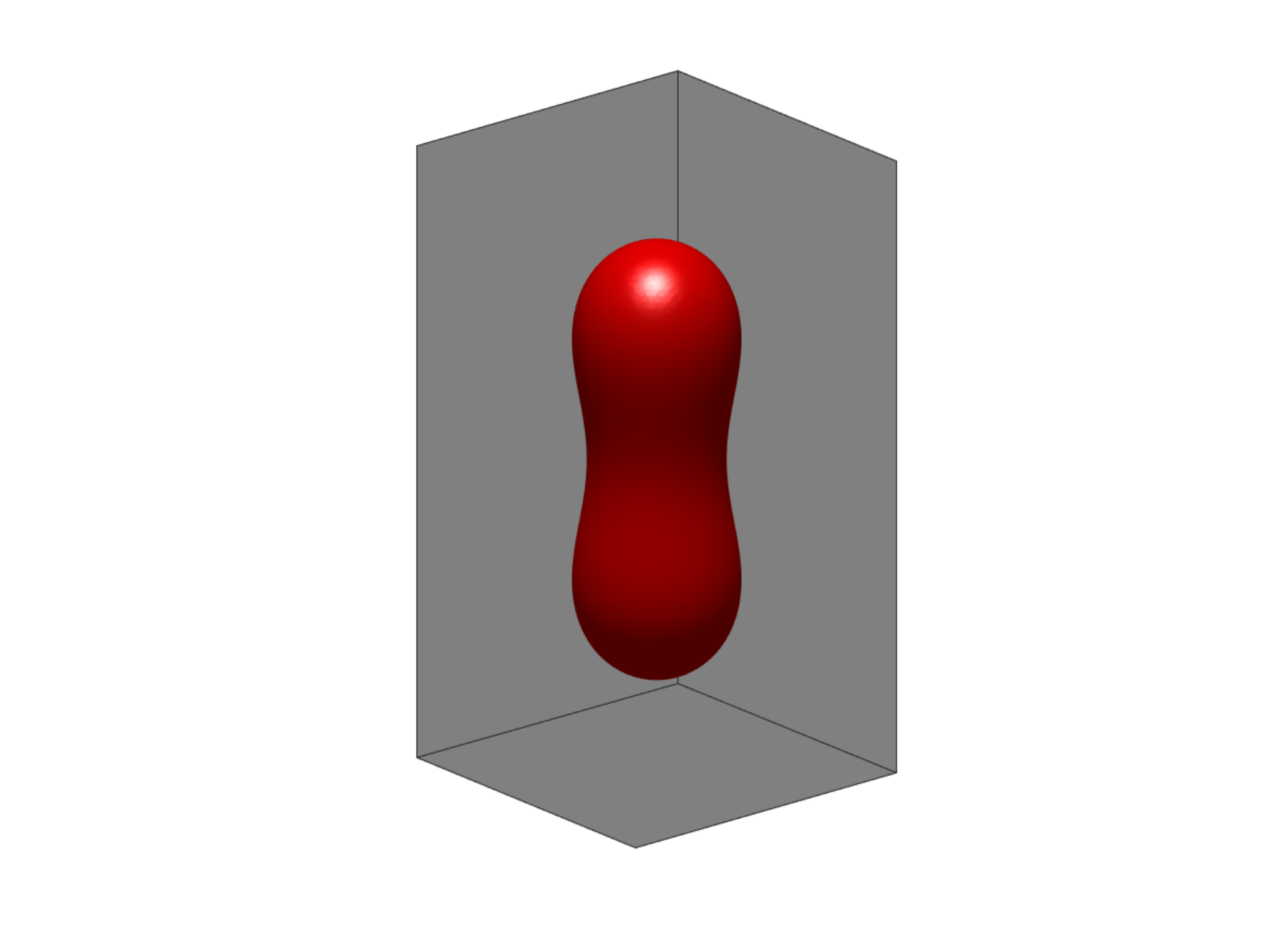}\hspace{-9mm}
	\includegraphics[width=5.3cm]{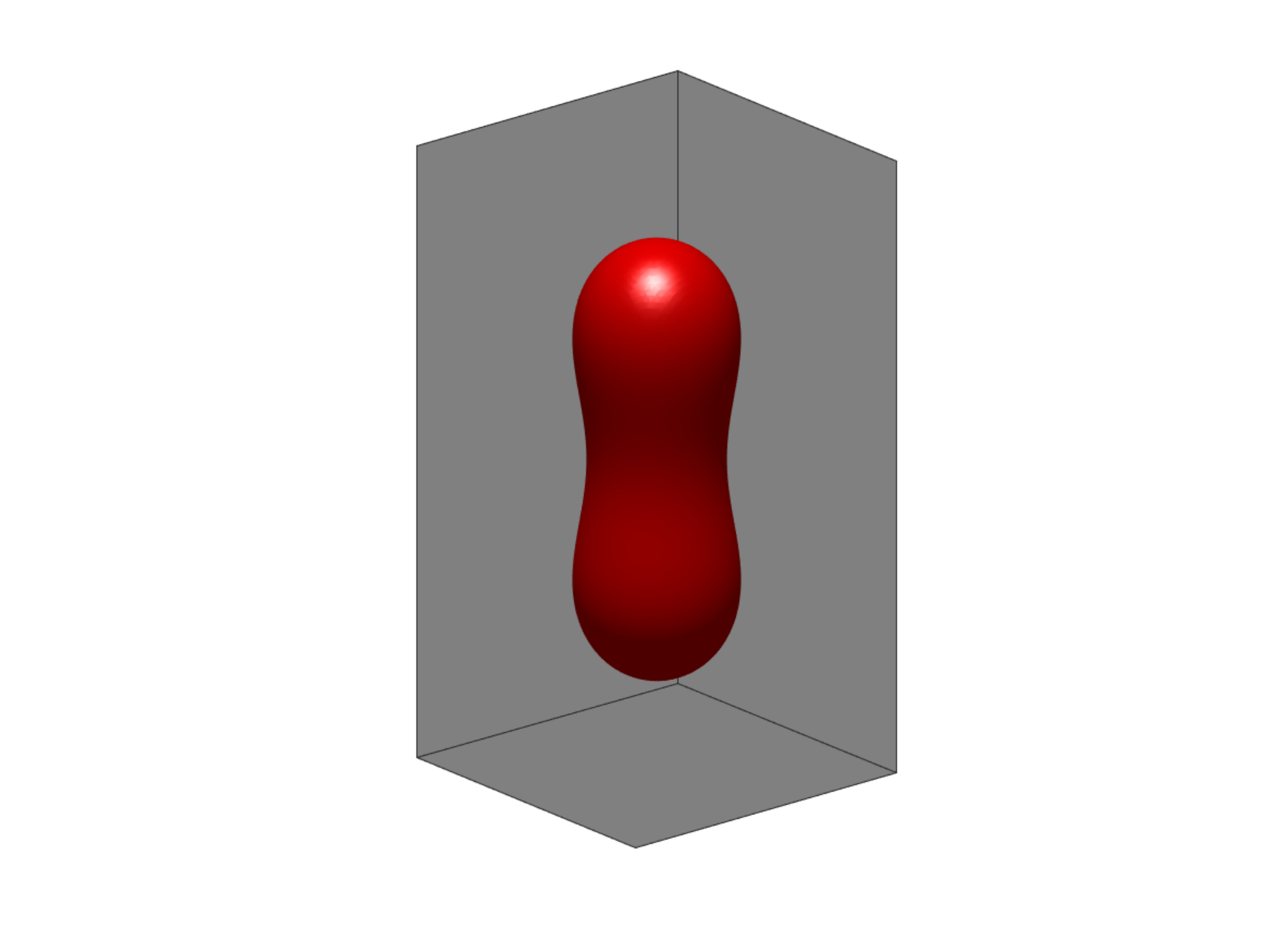}
}
\label{Fig:PFVM-R-newSAV-two-balls}
\caption{Example 4A. Collision of two close-by spherical vesicles: Snapshots of iso-surfaces of $\phi=0$ driven by the PFVM equation at $t=0, 0.02, 0.1, 0.5, 1, 2$.}
\end{figure}

\begin{figure}[htbp]
\centering
\subfigure[profiles of $\phi=0$ at $T=0, 0.02, 0.1$]{
	\includegraphics[width=5.3cm]{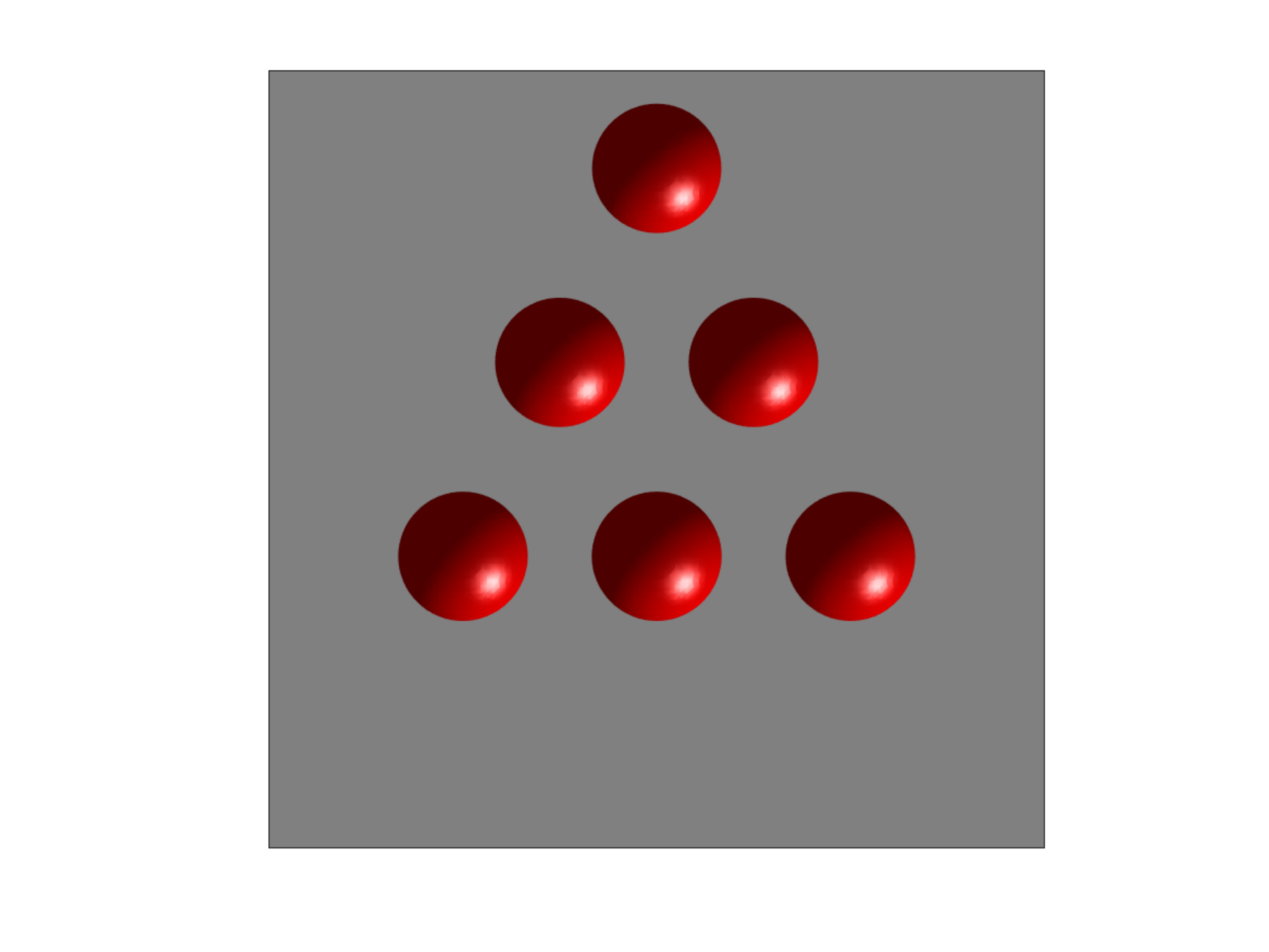}\hspace{-9mm}
	\includegraphics[width=5.3cm]{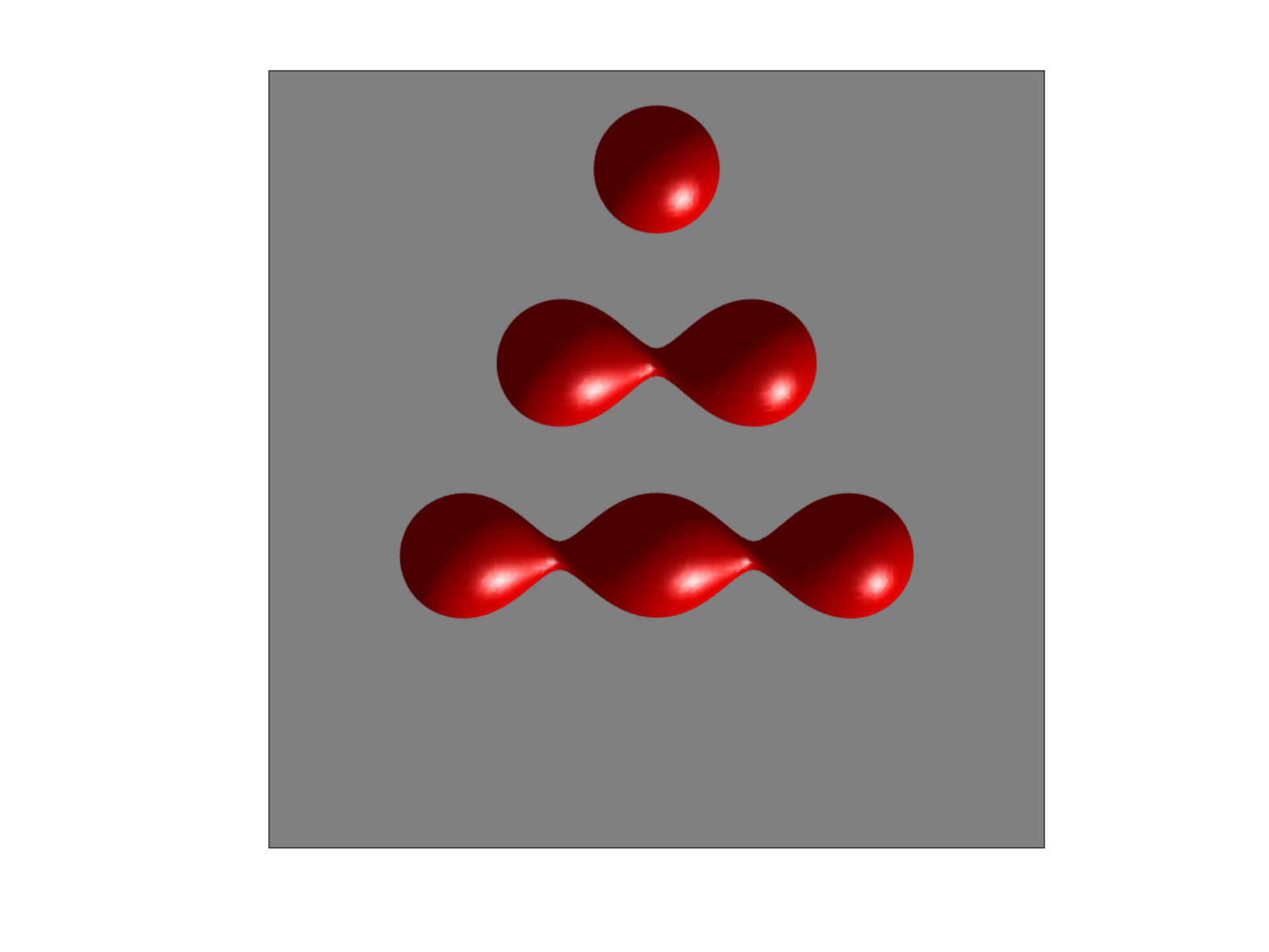}\hspace{-9mm}
	\includegraphics[width=5.3cm]{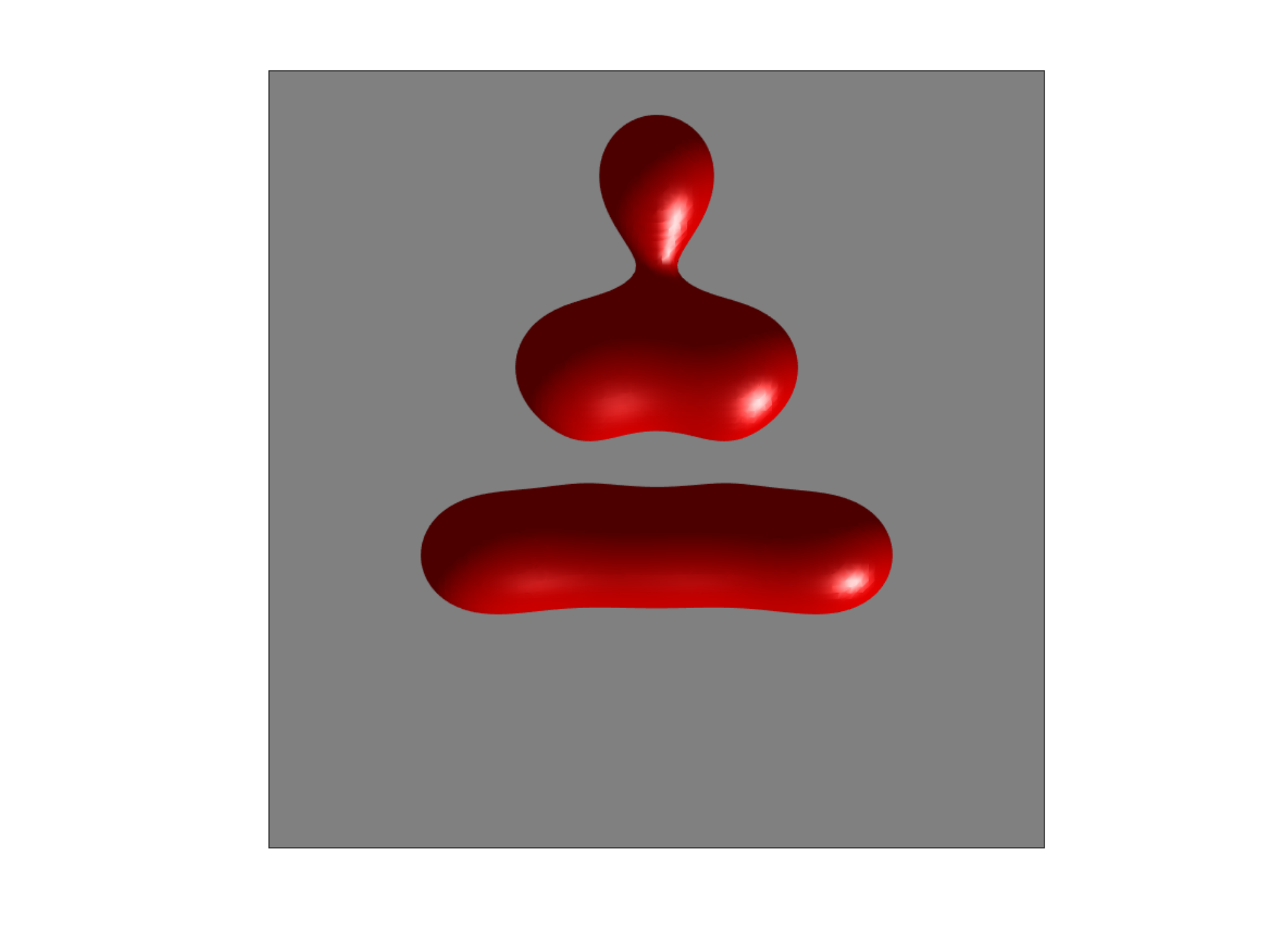}
} 
\subfigure[profiles of $\phi=0$ at $T=0.5, 1, 2$]{
	\includegraphics[width=5.3cm]{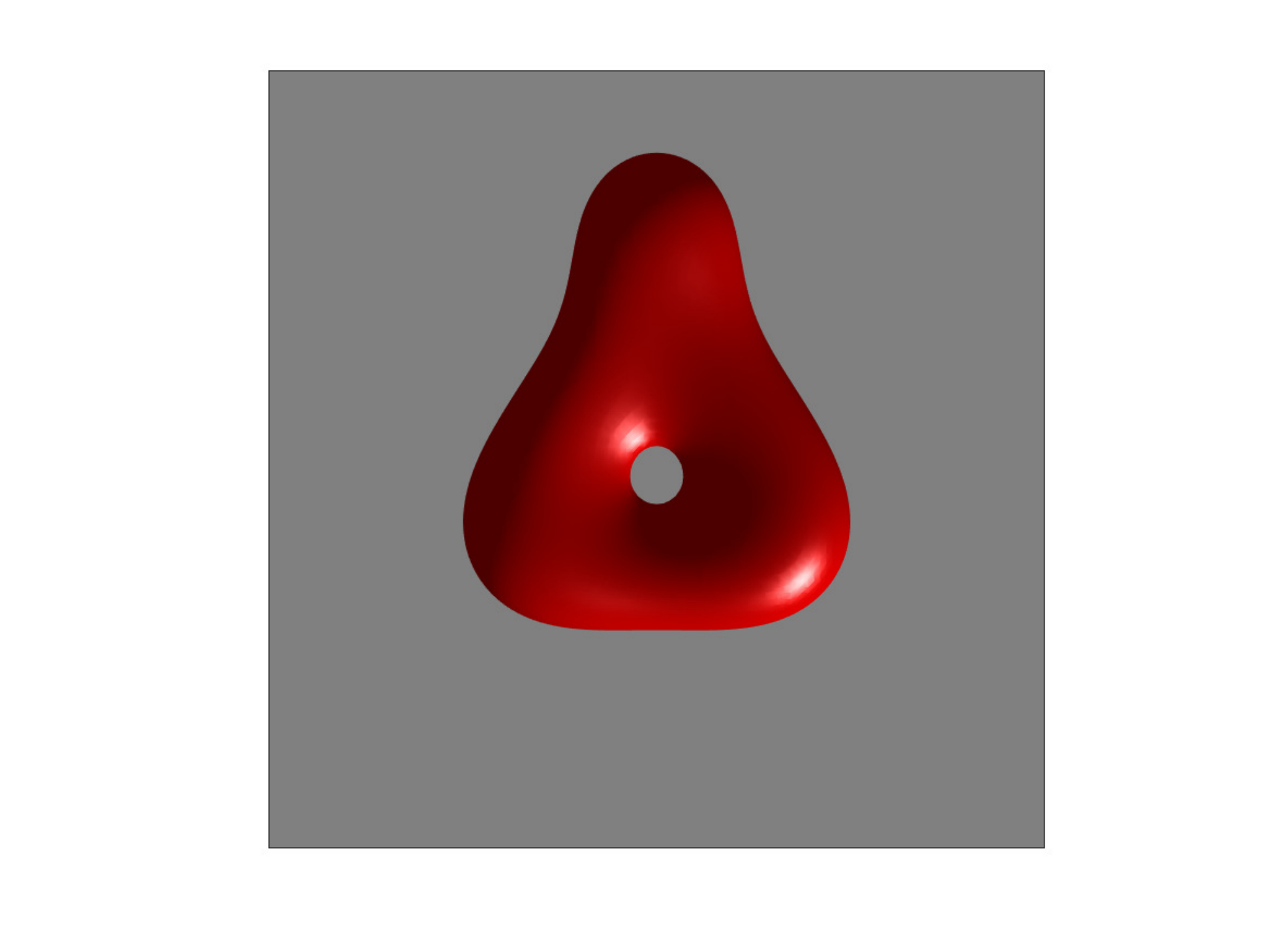}\hspace{-9mm}
	\includegraphics[width=5.3cm]{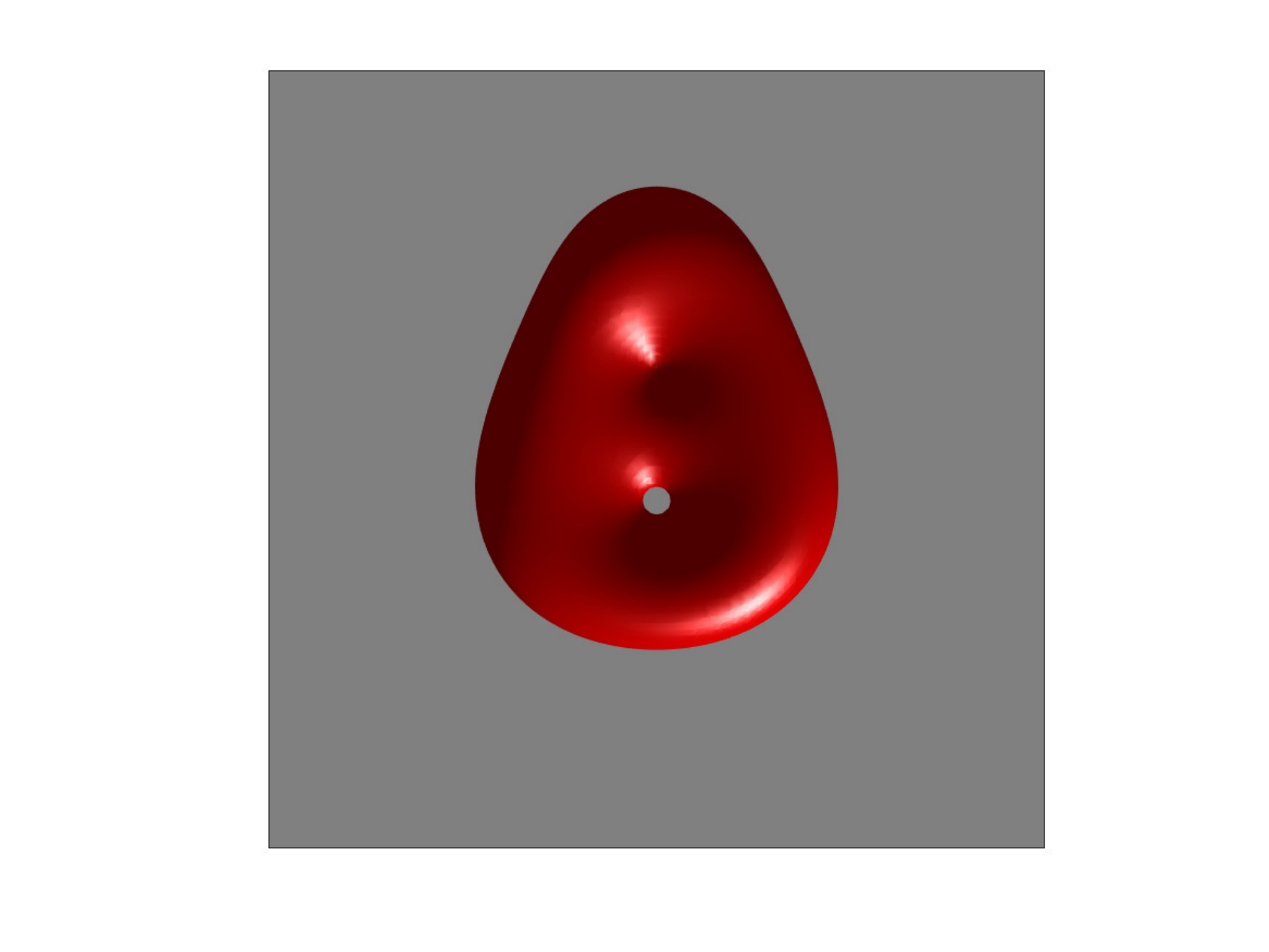}\hspace{-9mm}
	\includegraphics[width=5.3cm]{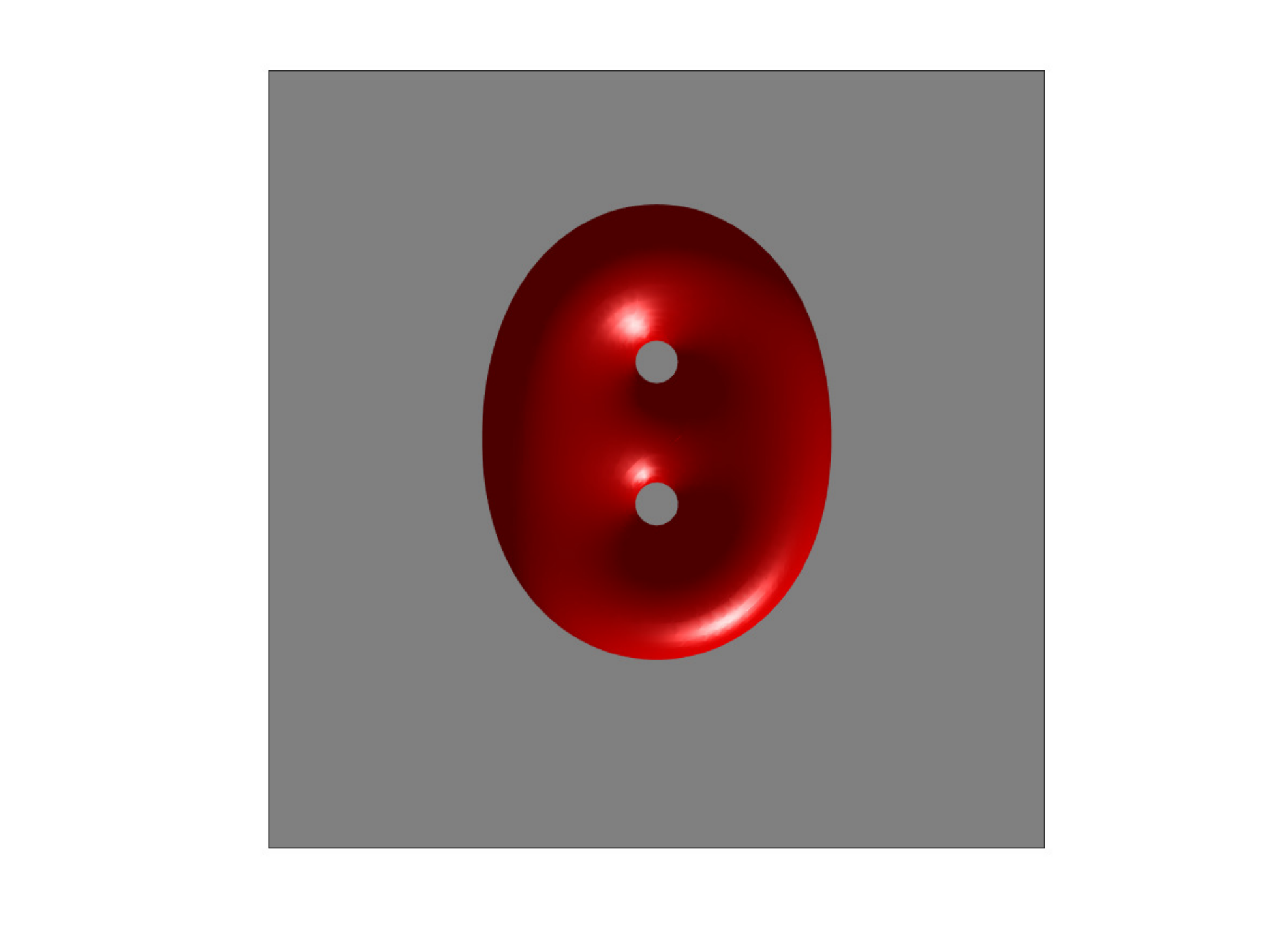}
}
\label{Fig:PFVM-R-newSAV-six-balls}
\caption{Example 4B. Collision of six close-by spherical vesicles. Snapshots of iso-surfaces of $\phi=0$ driven by the PFVM equation at $t=0, 0.02, 0.1, 0.5, 1, 2$.}
\end{figure}
Finally, we note that for the simulations in Examples 4-6, the relaxation parameter $\zeta_0$ is zero at all times.
This indicates that, at least  for these simulations, the modified energy is in fact  equal to the original energy, which means that the R-GSAV schemes are effectively energy stable with the original energy. 

In summary, the R-GSAV approach fixes a flaw in the GSAV approach and leads to more robust and accurate numerical schemes while keeping the simplicity, efficiency and generality of the GSAV approach.

\bibliographystyle{plain}
\bibliography{myref}

\end{document}